\documentclass[12pt]{article}

\usepackage[english]{babel}
\usepackage[utf8]{inputenc}
\usepackage[LGR, T1]{fontenc} 
\usepackage{lmodern}
\usepackage[a4paper, margin=2cm]{geometry}
\usepackage[onehalfspacing]{setspace}
\usepackage{microtype}
\usepackage{graphicx}
\graphicspath{{figures/}}
\usepackage{booktabs}
\usepackage{tabularx}
\usepackage{longtable}
\usepackage{array}
\usepackage{caption}
\usepackage{xcolor}
\usepackage{colortbl}
\definecolor{besttime}{RGB}{198, 239, 206} 
\definecolor{bestmem}{RGB}{221, 235, 247}  

\usepackage{mathtools}
\usepackage{amssymb}
\usepackage{amsthm}
\newtheorem{theorem}{Theorem}[section]
\newtheorem{lemma}[theorem]{Lemma}

\newtheorem{proposition}[theorem]{Proposition}
\newtheorem{definition}{Definition}[section]

\DeclareMathOperator{\dom}{dom}
\newcommand{\sgap}{\ensuremath{\varrho}}   
\newcommand{\ascore}{\ensuremath{\Lambda}} 
\newcommand{\tlim}{\ensuremath{\Gamma}} 
\newcommand{\bbt}{\ensuremath{\mathit{BBT}}} 

\usepackage{algorithm}
\usepackage{algpseudocode}

\usepackage{tikz}
\usetikzlibrary{arrows.meta}
\usetikzlibrary{calc}

\usepackage{csquotes}
\usepackage[
  backend=biber,
  style=alphabetic,
  sorting=none,
  language=american,
  maxcitenames=1,
  mincitenames=1,
  maxbibnames=7,
  minbibnames=7
]{biblatex}
\newcommand{\keywords}[1]{\textbf{Keywords}: \itshape #1 \normalfont}

\DeclareRobustCommand{\BoldSlashedO}{%
  \tikz[baseline=(char.base)]{
    \node[inner sep=0pt] (char) {O};
    \begin{scope}
    \clip (char.south west) rectangle (char.north east);
    \draw[line width=.55pt]
         ($(char.north west)+(4pt,-3pt)$) -- ($(char.south east)+(-4pt,3pt)$);
    \end{scope}
}}

\DeclareRobustCommand{\SlashedO}{%
  \tikz[baseline=(char.base)]{
    \node[inner sep=0pt] (char) {O};
    \begin{scope}
    \clip (char.south west) rectangle (char.north east);
    \draw[line width=.55pt]
         ($(char.north west)+(2.1pt,-1pt)$) -- ($(char.south east)+(-2.1pt,1pt)$);
    \end{scope}
}}

\DeclareRobustCommand{\KAYROS}{\textbf{KAYR\BoldSlashedO S}}
\DeclareRobustCommand{\kayros}{KAYR\SlashedO S}
\newcommand{\textgreek}[1]{\begingroup\fontencoding{LGR}\selectfont#1\endgroup}

\title{\Large \KAYROS: An Anytime and Exact Solver for the Time-Dependent Vehicle Routing Problem with Time Windows --- Extended Version}
\author{
  \normalsize Florian Rascoussier\textsuperscript{1,2}\thanks{ORCID identifier: \href{https://orcid.org/0009-0005-3253-9814}{0009-0005-3253-9814}.}\quad
  Romain Billot\textsuperscript{1}\thanks{ORCID identifier: \href{https://orcid.org/0000-0002-2614-3822}{0000-0002-2614-3822}.}\quad
  Lina Fahed\textsuperscript{1}\thanks{ORCID identifier: \href{https://orcid.org/0000-0002-7207-0580}{0000-0002-7207-0580}.}\quad
  Christine Solnon\textsuperscript{2}\thanks{ORCID identifier: \href{https://orcid.org/0000-0002-0919-496X}{0000-0002-0919-496X}.}\\[0.5ex]
  \small \textsuperscript{1}IMT Atlantique, Lab-STICC, CNRS, UMR 6285 (équipe DECIDE), France\\
  \small \textsuperscript{2}INSA Lyon, Inria, CITI, UR3720, 69621 Villeurbanne, France (équipe EMERAUDE)
}
\date{Version 3, 19/10/2026}

\usepackage[
    colorlinks=true,
    linkcolor=blue,
    citecolor=blue,
    urlcolor=blue
]{hyperref}
\hypersetup{
    pdftitle={KAYROS: An Anytime and Exact Solver for the Time-Dependent Vehicle Routing Problem with Time Windows --- Extended Version},
    pdfauthor={Florian Rascoussier, Romain Billot, Lina Fahed, Christine Solnon},
    pdfsubject={Time-dependent vehicle routing, anytime and exact optimization},
    pdfkeywords={vehicle routing, time-dependent travel times, anytime algorithms, branch-price-and-cut, benchmarking}
}

\begin{document}

\maketitle

\begin{abstract}
    The Time-Dependent Vehicle Routing Problem with Time Windows (TDVRPTW) captures a central difficulty of urban logistics: travel times vary with the time of day, so the cost of a route depends on when it is driven. This paper introduces \kayros, an open-source solver for the duration-minimization TDVRPTW that is both anytime and exact. It returns improving valid solutions throughout its time budget and can close instances with conditional computational optimality certificates through an integrated branch-price-and-cut component. We extend the composition of continuous arrival-time functions from the literature, previously described only at proof or proposition level, to the left-continuous functions that stepwise benchmarks require. We prove that the composition routine shipped in the checker computes that operation exactly. We analyze when its evaluation and breakpoint normalization are exact in IEEE-754 double-precision arithmetic, and argue that a canonical, epsilon-free checker must define the objective. The anytime layer combines greedy construction, granular time-dependent local search over balanced route trees, and iterated local search, with a fleet-aware descent for a fleet-cost objective. We further propose an anytime evaluation methodology based on a normalized signed primal integral under a pre-committed statistical design. On 212 instances from five families at a one-hour single-threaded budget, \kayros{} achieves a pooled anytime score 63.3\% lower than the strongest of three available contenders (Timefold, Hexaly, jsprit). All three contrasts are significant under Holm correction. \kayros{} improves all 30 high-effort references of Blauth et al. and 10 literature best-known solutions, and publishes 704 such certificates. The solver, benchmarks and campaign data are openly released.

    \vspace{2ex}

\end{abstract}

\setcounter{tocdepth}{2}
\tableofcontents
\clearpage

\section*{Glossary}
\addcontentsline{toc}{section}{Glossary}

\begingroup
\setstretch{1}
\begin{longtable}{@{}p{0.21\textwidth}@{}p{0.79\textwidth}@{}}
ACO & Ant Colony Optimization \\
API & Application Programming Interface \\
ATF & Arrival Time Function \\
B\&P & Branch \& Price \\
BBT & Balanced Binary Tree \\
BKS & Best-Known Solution \\
CBG & Customer-Based Graph \\
CVRPTW & Capacitated Vehicle Routing Problem with Time Windows \\
DDD & Dynamic Discretization Discovery \\
DIMACS & Center for Discrete Mathematics and Theoretical Computer Science \\
DM-TDVRPTW & Duration-Minimization Time-Dependent Vehicle Routing Problem with Time Windows \\
DP & Dynamic Programming \\
ESPPRC & Elementary Shortest Path Problem with Resource Constraints \\
FIFO & First-In First-Out \\
GA & Genetic Algorithm \\
HGS & Hybrid Genetic Search \\
IGP & Ichoua--Gendreau--Potvin (speed model) \\
ILP & Integer Linear Programming \\
ILS & Iterated Local Search \\
LCA-BST & Lowest-Common-Ancestor Binary Search Tree \\
LNS & Large Neighborhood Search \\
LP & Linear Programming \\
LS & Local Search \\
MAMUT & Machine learning And Matheuristics algorithms for Urban Transportation \\
MILP & Mixed Integer Linear Programming \\
NDCPWLF & Non-Decreasing Continuous PieceWise Linear Function \\
NDLCF & Non-Decreasing Left-Continuous PieceWise Linear Function (short for NDLCPWLF preferentially used to avoid confusion) \\
PP & Pricing Problem \\
PWL & PieceWise Linear \\
RDF & Route Duration Function \\
RNG & Road-Network Graph \\
RTF & Ready Time Function \\
SA & Simulated Annealing \\
SINTEF & Foundation for Industrial and Technical Research \\
TD & Time-Dependent \\
TDESPPRC & Time-Dependent Elementary Shortest Path Problem with Resource Constraints \\
TDSPP & Time-Dependent Shortest Path Problem \\
TDTSPTW & Time-Dependent Traveling Salesman Problem with Time Windows \\
TDVRPTW & Time-Dependent Vehicle Routing Problem with Time Windows \\
TS & Tabu Search \\
TTF & Travel Time Function \\
TW & Time Window \\
VRPTW & Vehicle Routing Problem with Time Windows \\
\end{longtable}
\endgroup

\section*{Table of Symbols}
\addcontentsline{toc}{section}{Table of Symbols}

The notation follows the global system of the thesis this paper belongs to: calligraphic letters are input sets, bold lowercase letters are vertex sequences, plain lowercase letters are scalars and set elements. The list is grouped by theme in definition order, and every symbol is also introduced with its domain at first use in the text.

\begingroup
\setstretch{1}
\begin{longtable}{@{}p{0.30\textwidth}@{}p{0.70\textwidth}@{}}
\multicolumn{2}{@{}l}{\textbf{Graph and instance data}} \\
\addlinespace[0.5ex]
$\mathcal{G} = (\mathcal{V}, \mathcal{A})$ & the directed customer-based graph of an instance \\
$\mathcal{V} = \mathcal{C} \cup \{o, d\}$ & the vertex set: the customers, plus the duplicated mono-depot as an origin $o$ and a destination $d$ \\
$\mathcal{C}$ & the customer set: the vertices to be served, i.e.\ every vertex of $\mathcal{V}$ but the two depot copies \\
$\mathcal{A} \subset \mathcal{V}^2$ & the arc set, excluding loops, arcs leaving $d$ and arcs entering $o$ \\
$n = \left| \mathcal{C} \right|$ & the number of customers, $n \in \mathbb{N}_{>0}$, so that $\left| \mathcal{V} \right| = n + 2$ \\
$K$ & the fleet-size bound, $K \in \mathbb{N}_{>0}$: the maximum number of routes a solution may use \\
$Q$ & the homogeneous vehicle capacity as a scalar, $Q \in \mathbb{R}_{>0}$ \\
$q_v$, $s_v$ & the demand and the service time duration of a vertex $v$, $q_v, s_v \in \mathbb{R}_{\geq 0}$, both zero at the depot \\
$[e_v, l_v] \subseteq \mathcal{T}$ & the time window of a vertex $v$, with $e_v \leq l_v$ the earliest and latest start of service \\
$\mathcal{T}_{\mathrm{ext}}$, $\mathcal{T}$, $\mathcal{T}_{\mathrm{arc}}$ & the time-set family: the ambient scheduling timestamps $\mathcal{T}_{\mathrm{ext}} = \mathbb{R}_{\geq 0}$, the finite planning horizon $\mathcal{T} = [0, T]$, and the arc-departure domain $\mathcal{T}_{\mathrm{arc}} = [a_{\mathrm{arc}}, T_{\mathrm{arc}}] \subseteq \mathcal{T}$ on which every arc function is materialized \\
$p$ & the instance-wide bound on the number of breakpoints of an arc travel time function, $p \in \mathbb{N}_{>0}$ \\
$c_{\mathrm{fleet}}$ & the fixed per-route cost of the FleetCostDuration objective, $c_{\mathrm{fleet}} \in \mathbb{R}_{\geq 0}$ \\
\addlinespace[1.5ex]
\multicolumn{2}{@{}l}{\textbf{Paths, routes and solutions}} \\
\addlinespace[0.5ex]
$\mathbf{p} = \langle v_1, \dots, v_m \rangle$ & a path: an elementary ordered sequence of vertices \\
$m$ & the size of a path or a route, i.e.\ its number of vertices, $m \in \mathbb{N}_{>0}$ \\
$\mathcal{P}_{\mathcal{V}}$ & the path set: all finite non-empty ordered sequences of vertices \\
$\mathbf{r} = \langle o, \mathbf{r}_{\mathcal{C}}, d \rangle$ & a route: a path from the origin depot to the destination depot, serving at least one customer \\
$\mathbf{r}_{\mathcal{C}}$ & the inner customer sequence of a route, depots excluded \\
$\Omega$ & the set of all feasible routes of an instance, $\Omega \subset \mathcal{P}_{\mathcal{V}}$ \\
$\mathcal{S} \subseteq \Omega$ & a solution: feasible routes whose customer sequences partition $\mathcal{C}$ \\
$\left| \mathcal{S} \right| \leq K$ & the fleet size of a solution, i.e.\ its number of routes \\
$c_{\mathbf{r}}$ & the cost of a route in the set-partitioning master problem: $\Delta_{\mathbf{r}}^{*}$ under \textit{Duration}, $c_{\mathrm{fleet}} + \Delta_{\mathbf{r}}^{*}$ under \textit{FleetCostDuration} \\
$\lambda_{\mathbf{r}}$ & the route-selection variable of the set-partitioning formulation \\
\addlinespace[1.5ex]
\multicolumn{2}{@{}l}{\textbf{Objectives and anytime evaluation}} \\
\addlinespace[0.5ex]
$z$, $z(\mathcal{S})$ & the value at a solution of whichever objective is selected, and $z$ alone a solution cost so evaluated \\
$z^*$ & the per-instance reference value of the frozen scoring snapshot, $z^* > 0$, an optimality certificate or a best-known solution; optimality is not implied \\
\tlim & the wall-clock time limit (runtime budget) of a run in seconds, $\tlim \in \mathbb{R}_{>0}$: an experimental parameter, not instance data \\
$z(t)$ & the best-so-far checker-evaluated incumbent value at elapsed time $t$, a right-continuous step function, $+\infty$ before the first incumbent \\
$g = (z - z^*)/z^*$ & the raw excess gap of an incumbent against the reference; $g_{\tlim}$ its end-of-budget value, the final gap \\
$\sgap(z, z^*)$ & the squeezed gap, $(z - z^*)/(z + z^*) = g/(g + 2) \in [-1, 1]$: signed, bounded, with the no-incumbent convention $\sgap = 1$ \\
$\ascore(\tlim)$ & the normalized signed primal integral (anytime score), the time average of \sgap{} over the horizon, $\ascore \in [-1, 1]$, lower is better \\
\addlinespace[1.5ex]
\multicolumn{2}{@{}l}{\textbf{Time-dependent function taxonomy}} \\
\addlinespace[0.5ex]
$\mathcal{F}$, $\mathcal{F}_{\text{ND}}$, $\mathcal{F}_{\text{ND}}^{\text{c}}$ & the real-valued left-continuous PWL functions on compact real intervals, their non-decreasing subclass (NDLCFs) and its continuous subclass (NDCPWLFs) \\
$\bot$ & the empty chain and the empty partial function: the infeasibility answer of route evaluation \\
$\beta_f = ((x^f_1, y^f_1), \dots)$ & a chain representing $f$: the stored breakpoint sequence, non-decreasing in both coordinates, written with round outer parentheses as a tuple-typed structure \\
$\phi_f$ & the breakpoint count of a chain, duplicate abscissae included and no normalization assumed \\
$\mathsf{s}(\beta)$ & the selected function of a chain: the smallest ordinate at a chain abscissa, linear interpolation elsewhere \\
$\dom(f) = [a_f, u_f]$ & the domain of a member of $\mathcal{F}_{\text{ND}}$; on an arc, a path or a route, shorthand for the domain of its ready time function \\
(H) & the domain hypothesis $a_f \leq g(a_g)$ under which the class is closed and the event merge exact \\
$\mathcal{F}^{\text{out}}_a$, $\mathcal{F}^{\text{in}}_a$, $\mathcal{F}^{\text{anc}}_a$ & the outer, inner and anchored temporal maps at anchor $a$, the subclasses route evaluation composes \\
$\tau_{ij} : \mathcal{T}_{\mathrm{arc}} \rightarrow \mathbb{R}_{\geq 0}$ & the arc Travel Time Function (TTF): the duration of $\langle i, j \rangle$ when leaving $i$ at $t$ \\
$\alpha_{ij}(t) = t + \tau_{ij}(t)$ & the arc Arrival Time Function (ATF), non-decreasing exactly when the arc is FIFO \\
$\theta_i$ & the vertex time-window Ready Time Function (vertex RTF), encoding the window $[e_i, l_i]$ and the service time $s_i$; both depot maps are identities \\
$\delta_{ij} = \theta_j \circ \alpha_{ij}$ & the arc Ready Time Function (arc RTF), an NDLCF over the feasible departure times of the arc \\
$\delta_{\mathbf{p}_k}$, $\delta_{\mathbf{r}_k}$ & the prefix RTF at the $k^\text{th}$ vertex of a path (of a route) \\
$\delta_{\mathbf{r}}$ & the route RTF: the ready time at the end depot $d$ for a departure from $o$ at $t$ \\
$\bbt^{\mathbf{r}}$, $\delta_{v_i v_j}^{\mathbf{r}}$ & the Balanced Binary Tree of a route $\mathbf{r}$ and its node RTFs, one per contiguous segment \\
$\mathit{lf}_k$ & the $k^\text{th}$ leaf of the route folding, so that $\delta_{\mathbf{r}} = \mathit{lf}_{m-1} \circ \dots \circ \mathit{lf}_1$ \\
$\Delta_{\mathbf{r}}(t) = \delta_{\mathbf{r}}(t) - t$ & the Route Duration Function (RDF), the only function the route algebra builds that may fall outside $\mathcal{F}_{\text{ND}}$, the arc TTF $\tau_{ij}$ being the other potentially non-monotone member of the taxonomy \\
$\Delta_{\mathbf{r}}^{*}$, $t_{\mathbf{r}}^{*}$ & the minimum duration of a route and the dispatch time attaining it, the quantities minimized by the \textit{Duration} objective \\
\addlinespace[1.5ex]
\multicolumn{2}{@{}l}{\textbf{Time warp (penalty-tolerant squeeze)}} \\
\addlinespace[0.5ex]
$\tilde{\Delta}_{\mathbf{r}}$ & the duration of a route under the warp-tolerant evaluation, the tilde distinguishing it from the exact $\Delta_{\mathbf{r}}^{*}$ \\
$W_{\mathbf{r}}$ & the accumulated route time warp (lateness), $W_{\mathbf{r}} \geq 0$ \\
$\Psi_{\mathbf{r}} = \tilde{\Delta}_{\mathbf{r}} + \eta\, W_{\mathbf{r}}$ & the penalized route duration of the squeeze \\
$\eta$ & the time-warp penalty weight, $\eta > 0$ \\
\addlinespace[1.5ex]
\multicolumn{2}{@{}l}{\textbf{Recurring decorations and conventions}} \\
\addlinespace[0.5ex]
$x^{*}$ & star: a curated reference or best-found value of $x$, optimality not implied, as in $z^*$, $\mathcal{S}^*$; on $\Delta_{\mathbf{r}}^{*}$ and $t_{\mathbf{r}}^{*}$ it keeps the star of the TD literature, there marking the exact minimum of one function rather than a reference \\
$\tilde{x}$ & tilde: a clamped or relaxed variant of $x$, as in $\tilde{\Delta}_{\mathbf{r}}$ \\
$h(t^+)$ & the right limit of a function of $\mathcal{F}_{\text{ND}}$ at $t$ \\
$\langle i, j \rangle$ & angle brackets: reserved for arcs and vertex sequences, paths and routes included, whereas $\{ \dots \}$ is unordered; other vectors or tuple-typed objects, breakpoint chains included, keep round parentheses \\
$[a, b]_{\mathbb{N}}$ & an integer range, plain $[a, b]$ being a real interval \\
\end{longtable}
\endgroup
\clearpage

\section{Introduction}\label{sec:intro}

Urban travel times vary with the time of day \cite{savelsberghCityLogisticsChallenges2016}, so in time-dependent routing the cost of a route is not a sum of arc constants but a composition of non-decreasing piecewise-linear functions, continuous or not, evaluated in floating-point arithmetic \cite{lera-romeroEnhancedBranchPrice2018, lera-romeroLinearEdgeCosts2020}. Every duration-minimization method rests on that one operation. Yet, to the best of our knowledge, the literature has never stated the arithmetic under which a route cost, a published solution and an optimality claim mean the same thing from one implementation to the next. The consequences are measurable rather than hypothetical. Benchmark instances circulate with their defining preprocessing never published as data, while published solutions lose feasibility under an equally defensible reading of the same instance; and optimality claims rest on conditions that are written down nowhere. What a solver delivers over time is unsettled in a related way. Instruments that price the whole trajectory a dispatcher consumes, rather than the value held at one time limit, are established, but a comparison built on them still has to fix what they leave open: the score of a run before its first solution or with no solution at all, and the published reference every gap is read against, together with when that reference is cut. No such comparison has ever been reported on this problem class.

In the classic Vehicle Routing Problem with Time Windows (VRPTW), a fleet of homogeneous vehicles must serve every customer within its time window while respecting vehicle capacities (a.k.a CVRPTW), at minimum travel cost and under fixed travel times that ignore urban road congestion \cite{gendreauTimedependentRoutingProblems2015, vidalUnifiedSolutionFramework2014, dabiaBranchPriceTimeDependent2013, lera-romeroLinearEdgeCosts2020}. Ignoring that variability yields routes that are late or infeasible once driven \cite{rincon-garciaHybridMetaheuristicTimedependent2017, blauthVehicleRoutingTimedependent2024}. The Time-Dependent VRPTW (TDVRPTW, a.k.a TDCVRPTW) addresses this by making travel times vary with the time of departure \cite{malandrakiTimeDependentVehicle1992}.

In Time-Dependent (TD) First-In First-Out (FIFO) routing contexts including Shortest Path Problems (TDSPP), multiple objectives can be distinguished based on decision variables for departure times. The \textit{Makespan} objective minimizes the total elapsed time (earliest arrival time at return depot) when vehicle departure times from the depot are fixed \cite{fontaineExactAnytimeApproach2023}, while the \textit{Duration} objective allows departure times from the depot to be decision variables, thus minimizing total route duration \cite{visserEfficientMoveEvaluations2020, lera-romeroLinearEdgeCosts2020}. A third category, \textit{Travel Time}, further generalizes this by also allowing departure times at every customer to be decision variables, substantially increasing complexity \cite{heDynamicDiscretizationDiscovery2022}. This paper focuses on the \textit{Duration} objective, together with a fleet-cost variant of it defined in section \ref{sec:math} for comparison with the seminal work of \citeauthor{blauthVehicleRoutingTimedependent2024} \cite{blauthVehicleRoutingTimedependent2024}.

The TDVRPTW has historically been tackled by two families of methods: exact algorithms, which return optimality claims but scale poorly, and (meta-)heuristics, which scale at the cost of those claims \cite{fontaineExactAnytimeApproach2023, marti50YearsMetaheuristics2024, adamoReviewRecentAdvances2024}. Branch-price-and-cut dominates the exact side of this problem class \cite{dabiaBranchPriceTimeDependent2013, lera-romeroLinearEdgeCosts2020}. Both families in fact produce improving solutions while they run, yet both are compared on the single value they hold when the clock stops. Two runs ending at that value can differ by orders of magnitude in the moment they first become usable. The primal-integral lineage of section \ref{subsec:anytime-background} prices that difference. Until recently, there was no set of solvers a practitioner could obtain and run against on this problem class, hence nothing to apply it to.

\subsection{Contributions}\label{subsec:contributions}

This paper makes four contributions. The first is \kayros{} \cite{rascoussierKAYROS2026} itself, named after the ancient Greek \textit{kairós} (\textgreek{καιρός}), the right moment to act, as opposed to \textit{chronos} (\textgreek{χρόνος}) as mere elapsed time. \kayros{} is an open-source solver for the duration-minimization (DM-)TDVRPTW that is anytime and exact over one time-dependent engine, evaluated here as release \texttt{1.6.0} with one thread per run. Its governing rule is that the checker is the referee: an external, epsilon-free, canonical checker defines the objective, the solver engine is a bit-identical port of that checker's arithmetic, and every value the solver publishes is the checker's recomputation of the routes it returns. The same rule reaches inside the search that we call \emph{trees rank, the fold accounts}, i.e. fast structures ranks candidate moves while every accepted move is repriced by the checker-consistent sequential fold before it can become an incumbent.

The second contribution is the algorithmic core that any duration-minimization solver rests on. The literature composes non-decreasing continuous piecewise-linear functions (NDCPWLFs), an operation described through a proof \cite{visserEfficientMoveEvaluations2020} and at proposition level \cite{blauthVehicleRoutingTimedependent2024} but never as an algorithm. Yet the canonical stepwise benchmarks need functions that jump. We identify the semantic gap this leaves and define the left-continuous class the checker actually evaluates, separating the stored breakpoint chain from the function it selects. We then prove under one explicit domain hypothesis that the class is closed under composition. The composition routine shipped in the checker computes that operation exactly (edge cases included: vertical runs, plateaus, their tie, boundary ties), and the continuous results are recovered as the jump-free case. We add to this an analysis of when its evaluation and breakpoint normalization are exact in IEEE-754 double-precision arithmetic, and of the three independent families of defects that implicit descriptions of this operation have caused in the literature. The stakes are measured rather than asserted: 72 of the 146 published Best-Known Solutions (BKS) of \citeauthor{lera-romeroLinearEdgeCosts2020} \cite{lera-romeroLinearEdgeCosts2020} carry at least one exact floating-point tie, through 188 of their 898 routes, and re-deriving the instance distances at full precision flips 11 of those solutions to infeasible. This is why the solver and its companion benchmark suite distribute instances as byte-exact canonical data rather than as re-derivation recipes.

The third contribution is an evaluation methodology for anytime solvers, applied here rather than merely proposed: a normalized signed primal integral scored against a published reference, aggregated under panel-equal weighting and a confirmatory design pre-committed before any campaign record existed. The design rationale of the metric itself is developed in a companion paper. \cite{rascoussierAnytimeSolverEvaluation2026} We apply the methodology in a cross-solver campaign of 212 instances against the three \emph{available} solvers, in the sense section \ref{subsec:solver-landscape} fixes: obtainable, time-dependent at the model level, and exposing an interface through which an instance can be fed and a solution stream read back. The fourth contribution is the exact component and its certificate semantics. Building on the open-source branch-price-and-cut solver of \citeauthor{lera-romeroLinearEdgeCosts2020} \cite{lera-romeroLinearEdgeCosts2020}, we state precisely what an optimality certificate establishes, as a conditional computational claim gated by a four-run publication protocol. We publish 704 such certificates over 2\,072 instances as falsifiable, retractable artifacts rather than as proof objects.

On the experimental side, the benchmark of section \ref{sec:results} holds 212 instances drawn from five families, split into 10 panels, on which every arm receives one hour of single-threaded time per run, for a campaign of 8\,800 runs. \kayros{} attains a pooled anytime score 63.3\% lower than the strongest of the three available contenders, all three pre-committed contrasts being significant under Holm correction, each at the resolution floor of the bootstrap. It holds the lowest mean score on 7 of the 10 panels. On the classic benchmark of the exact literature it establishes 10 values strictly below published references that earlier work left unproven. Under the fleet-cost objective it improves all 30 high-effort reference solutions of \citeauthor{blauthVehicleRoutingTimedependent2024} \cite{blauthVehicleRoutingTimedependent2024} at an unchanged route count. Section \ref{sec:results} reports all of this in full, including the three panels Timefold takes, the fleet ceiling our fleet descent hits at $n \geq 1000$, and the certificates we have withdrawn.

\subsection{Outline}

The paper is built as one argument: sections \ref{sec:background} and \ref{sec:math} establish the composition operation and the arithmetic under which a route cost is well defined, section \ref{sec:solver} builds the solver on them, and section \ref{sec:results} measures what that buys against the competition a practitioner can obtain. Section \ref{sec:discussion} states what the measurement does not establish, and section \ref{sec:conclusion} closes on the four contributions above. Appendices \ref{app:bks} to \ref{app:components} carry the material the running text points to: per-instance records, per-panel statistics, the composition grid, the full function taxonomy with the proofs of the composition theorems, the solver internals, the protocol reference and the component studies.

\section{Background and Related Work}\label{sec:background}

\subsection{Time-Dependent Literature}\label{subsec:td-literature}

The TDVRPTW of \citeauthor{malandrakiTimeDependentVehicle1992} \cite{malandrakiTimeDependentVehicle1992} rests on time-dependent shortest path problems, whose simplest objective is minimizing the arrival time at the destination (\textit{Makespan}). Early works build upon classic Dynamic Programming (DP) algorithms in discretized time settings \cite{cookeShortestRouteNetwork1966}. The seminal work of \citeauthor{deanShortestPathsFIFO2004} \cite{deanShortestPathsFIFO2004} discusses many essential components of TD routing problems, including the FIFO property, Inverse Arrival Time (IAT), Earliest Arrival Time (EAT), Earliest Departure Time (EDT), Latest Departure Time (LDT), and discretization versus continuous time models. It nonetheless makes incorrect claims about the complexity of the problem in the continuous setting, corrected in \cite{foschiniComplexityTimeDependentShortest2011}. In the context of the TDVRPTW, the presence of hard time windows makes the \textit{Makespan} objective for a given route polynomial and \textit{Duration} pseudo-polynomial, as it depends on the number of breakpoints in the PieceWise Linear (PWL) Travel Time Functions (TTFs) \cite{omerTimedependentShortestPath2019, omerPolynomialAlgorithmMinimizing2019}.

Building FIFO TTFs generally relies on the speed-profile model of Ichoua, Gendreau and Potvin (IGP) \cite{ichouaVehicleDispatchingTimedependent2003}, which gives speeds as piecewise-constant profiles and travel times as the continuous piecewise-linear functions they induce, and is the model behind most published TDVRPTW benchmarks. Time discretization and a realistic non-IGP road-network benchmark are discussed in \cite{rifkiImpactSpatiotemporalGranularity2020} and used in a traveling-salesman context in \cite{fontaineExactAnytimeApproach2023}. The FIFO property of that benchmark, lost with discretization, is restored by preprocessing, and \kayros{} draws on that work for its own continuous-time loading of non-IGP travel times. Discretization also runs the other way, the breakpoints of the TTFs being the structure that Dynamic Discretization Discovery (DDD) \cite{bolandPerspectivesIntegerProgramming2019} exploits to build its discretization \cite{heDynamicDiscretizationDiscovery2022}. Another stream of research simplifies the TTFs through some form of controlled approximation \cite{blauthVehicleRoutingTimedependent2024} or optimizes their computation with tailored data structures \cite{visserEfficientMoveEvaluations2020}. \kayros{} takes structures from both and approximation from neither: the route trees of section \ref{subsec:kayros-folding} rank candidate moves, while every accounted cost is recomputed exactly.

Underneath all of these lines sits one primitive: evaluating a route under the \textit{Duration} objective means composing the non-decreasing piecewise-linear functions carried by its successive arcs, continuous in the literature and left-continuous on stepwise data and vertices. The ready time of a prefix becomes the argument of the next function, so that every move evaluation, pricing label and published cost is an instance of that composition. The operation is long-standing prior art, from time-dependent shortest paths \cite{ordaShortestpathMinimumdelayAlgorithms1990} to time-dependent contraction hierarchies on road networks \cite{batzMinimumTimedependentTravel2013}. Within TD routing it has been stated at two levels for continuous functions: through the proof of a theorem \cite{visserEfficientMoveEvaluations2020} and at proposition level with complexity and correctness \cite{blauthVehicleRoutingTimedependent2024}. To the best of our knowledge, it has never been stated as an algorithm with its edge cases, and never for the discontinuous functions of stepwise benchmarks. Section \ref{subsec:composition} states it as one, floating-point behavior included.

\subsection{Exact Methods}\label{subsec:exact-methods}

Exact methods for this problem class rest on column generation \cite{desaulniers2006column}: the Linear Program (LP) is decomposed into a restricted master problem over a small pool of routes and a Pricing Problem (PP), typically NP-hard, that generates routes of negative reduced cost. Branch \& Price (B\&P) embeds that process in a branch-and-bound tree to restore the integrality of the Integer Linear Program (ILP). Section \ref{subsec:spf} states the reformulation \kayros{} uses, and the tutorial literature covers the machinery \cite{feilletTutorialColumnGeneration2010, uchoaOptimizingColumnGeneration, petrisTutorialBranchPriceandCutAlgorithms2023}.

B\&P methods dominate exact approaches in the TDVRPTW literature. The pioneering work of Dabia et al.~\cite{dabiaBranchPriceTimeDependent2013} solves the Time-Dependent Elementary Shortest Path Problem with Resource Constraints (TDESPPRC) as the pricing subproblem with a DP algorithm (a.k.a Labeling Algorithm). Lera-Romero et al.~\cite{lera-romeroLinearEdgeCosts2020} later refined that scheme with cutting planes and additional pricing relaxations into branch-price-and-cut, releasing the open-source solver the exact component of \kayros{} is built on. The same frame has been carried to related variants: up to 25 customers with heuristic pricing \cite{huartHeuristicTimeDependentVehicle2016}, 45 requests for the pickup-and-delivery variant \cite{sunTimedependentPickupDelivery2018}, 25 customers on road networks \cite{bentichaTimeDependentVehicleRouting2021} and 25 customers for time-window assignment \cite{splietTimeWindowAssignment2018}. Outside that frame, the DP, constraint programming and Mixed Integer Linear Programming (MILP) formulations common in the literature of the Time-Dependent Traveling Salesman Problem with Time Windows (TDTSPTW) \cite{dabiaDynamicProgrammingApproach2010, fontaineExactAnytimeApproach2023} do not scale well once the Hamiltonian property is relaxed in the multi-vehicle context \cite{toth2002vehicle}. This work also builds on the time-dependent traveling-salesman research developed with Christine Solnon, from Pénélope Aguiar-Melgarejo's constraint-programming thesis \cite{aguiar-melgarejoConstraintProgrammingApproach2016} to Romain Fontaine's thesis and the exact-and-anytime approach of Fontaine, Dibangoye and Solnon \cite{fontaineExactAnytimeHeuristic2024,fontaineExactAnytimeApproach2023}. These studies provide the single-vehicle foundations for the multi-vehicle setting considered here.

That lineage states its own warning, which section \ref{subsec:exactness-background} takes up: what the exact component of \kayros{} reports (section \ref{subsec:kayros-exact}) is a conditional computational claim, made under explicit arithmetic, tolerance and search assumptions, and not a proof object a reader can check without rerunning the solver.

\subsection{Heuristic Methods}

Population-based methods range from Genetic Algorithms (GAs), adapted to the time-dependent setting with custom encodings and evaluation functions~\cite{kumarTimeDependentVehicleRouting2015, kumarDevelopmentEfficientGenetic2017}, to the Hybrid Genetic Search (HGS) of \citeauthor{zhaoHybridGeneticSearch2024}~\cite{zhaoHybridGeneticSearch2024}, whose tailored DP time-dependent split addresses the multi-trip TDVRPTW. Single-solution and hybrid frameworks are as common. Destroy-and-repair schemes built on Large Neighborhood Search (LNS) reinsert removed customers through DP routines~\cite{rincon-garciaHybridMetaheuristicTimedependent2017}, also on a multi-trip variant~\cite{panMultitripTimedependentVehicle2021}. Tabu Search (TS) is competitive on its own, both on time-dependent speed profiles~\cite{ichouaVehicleDispatchingTimedependent2003} and on road networks~\cite{gmiraTabuSearchTimedependent2021}, and it has also been combined with LNS~\cite{panHybridAlgorithmTimedependent2021}. Simulated Annealing (SA) has been driven by time-dependent shortest paths~\cite{jaballahTimedependentShortestPath2021}.

Local Search (LS) adapted to the time-dependent setting is the intensification layer common to those frameworks, embedded in Iterated Local Search (ILS)~\cite{hashimotoIteratedLocalSearch2007}, in LNS~\cite{panMultitripTimedependentVehicle2021} and in heuristics under realistic traffic models~\cite{andresfigliozziTimeDependentVehicle2012, blauthVehicleRoutingTimedependent2024}. The efficient evaluation of time-dependent move operators is its recurring concern~\cite{hashimotoIteratedLocalSearch2007, visserEfficientMoveEvaluations2020, blauthVehicleRoutingTimedependent2024}. Section \ref{subsec:kayros-folding} treats that evaluation as the layer every other component prices through. A time-dependent Ant Colony Optimization (ACO) stream ran from 2008 to 2011. Yildirim~\cite{yildirimANTCOLONYALGORITHM2008} proposed an ACO framework augmented with local search, and Donati et al.~\cite{donatiTimeDependentVehicle2008} introduced a Multi-Ant Colony System, later extended by Balseiro et al.~\cite{balseiroAntColonyAlgorithm2011} to a bi-objective TDVRPTW. That stream has been dormant since.

The anytime layer of \kayros{} sits in the single-solution branch: a first-improvement time-dependent descent restricted to granular candidate lists, perturbed by ruin-and-recreate kicks and driven by late acceptance, in the shape of the modern single-trajectory references, notably the lessons learned from PyVRP~\cite{PyVRP, hashimotoIteratedLocalSearch2007}. Section \ref{subsec:kayros-anytime} states it, and section \ref{subsec:components} reports the head-to-head against an ant-colony arm over the same descent that fixed the default. Objectives pricing the fleet alongside duration make the route count a search target of its own, which the VRPTW literature addresses with dedicated route-minimization heuristics~\cite{nagataPowerfulRouteMinimization2009}. The high-effort reference solutions of \citeauthor{blauthVehicleRoutingTimedependent2024}~\cite{blauthVehicleRoutingTimedependent2024} are published under such an objective. Section \ref{subsec:kayros-fleet} presents the fleet-aware descent \kayros{} answers them with.

\subsection{Relaxations, Approximations and Matheuristics}

\citeauthor{foschiniComplexityTimeDependentShortest2011}~\cite{foschiniComplexityTimeDependentShortest2011} describe a bounded sampling-based approximation of TTFs that reduces their breakpoint count. Matheuristics instead reuse the master-and-pricing decomposition itself, solving the master with an ILP or MILP solver and the PP with a heuristic \cite{huartHeuristicTimeDependentVehicle2016, zhaoTimedependentBiobjectiveVehicle2019}. \kayros{} follows neither: every composition stays exact and un-normalized (section \ref{subsec:composition}), and its exact mode prices through valid relaxations rather than through a heuristic (section \ref{subsec:kayros-exact}).

The PP is the Elementary Shortest Path Problem with Resource Constraints (ESPPRC), which is NP-hard \cite{dror1994note}, so relaxing elementarity for faster pricing is standard practice \cite{petrisTutorialBranchPriceandCutAlgorithms2023}. The ng-route relaxation \cite{baldacciNewRouteRelaxation2011} sits between the elementary and the fully relaxed problem, allowing a revisit only once the route has left a per-customer neighborhood set. The routes it yields form a superset of the elementary ones and admit stronger label domination \cite{bulhoesBranchandpriceAlgorithmMinimum2018}, at a label count growing with the neighborhood size. Neighborhoods can be augmented dynamically \cite{tilkDynamicProgrammingMinimum2017}, in the spirit of decremental state-space relaxation \cite{pecinEfficientNgRoutePricing2012}, and a time-dependent traveling-salesman variant of the relaxation uses non-time-dependent neighborhoods \cite{lera-romeroDynamicProgrammingTimeDependent2022}. The base solver of \citeauthor{lera-romeroLinearEdgeCosts2020}~\cite{lera-romeroLinearEdgeCosts2020} prices the elementary TDESPPRC exactly, preceded by cheaper heuristic labeling levels that relax its cost and elementarity dominance checks. An ng-route level was added to the vendored solver in the course of this work, but it is dormant by default. The exact component of \kayros{} prices under full elementarity and capacity (section \ref{subsec:kayros-exact}). Exact solvers of this kind stop near 100 customers while the heuristics above reach urban scale, but both price every route through the same composition, and neither literature says in what arithmetic.

\subsection{Numerical Exactness and Data Provenance}\label{subsec:exactness-background}

Section \ref{subsec:td-literature} left the composition at proof \cite{visserEfficientMoveEvaluations2020} or proposition \cite{blauthVehicleRoutingTimedependent2024} level, on continuous functions, and without the arithmetic that would make two independent implementations agree. \citeauthor{dabiaBranchPriceTimeDependent2013} \cite{dabiaBranchPriceTimeDependent2013} published a corrected version of their branch-and-price paper \cite{dabiaErratumBranchPrice2024}, and \citeauthor{lera-romeroLinearEdgeCosts2020} state that implementation details matter on this problem class \cite{lera-romeroLinearEdgeCosts2020}. Neither paper sets out that specification, which section \ref{subsec:composition} gives, together with an analysis of when evaluation and breakpoint normalization are exact in IEEE-754 double-precision arithmetic (appendix \ref{app:taxonomy}).

The same uncertainty reaches the instances. \textit{Dabia2013} \cite{dabiaBranchPriceTimeDependent2013}, the reference benchmark of the exact time-dependent literature, is not fully specified by its published description: the preprocessing that defines it was never released as data such that it was necessary to rely on the authors' help to re-derive it. That preprocessing scales every Solomon quantity by 10 and floor-truncates every arc distance to an integer, choices that are invisible in the papers yet part of the definition of the benchmark (section \ref{subsec:benchmarks}, appendix \ref{app:protocol}). Scaling and rounding choices of that kind change the feasible region of every instance \cite{rascoussierImpactScalingRounding2026}. The consequence is measurable: re-deriving the distances at full precision, an equally defensible reading of the same data, makes some of the 146 published best-known solutions (BKS) of \citeauthor{lera-romeroLinearEdgeCosts2020} \cite{lera-romeroLinearEdgeCosts2020} infeasible, while those same solutions remain exactly correct on the canonical instances. Section \ref{subsec:benchmarks} gives the count, appendix \ref{app:bks} the per-instance record and appendix \ref{app:taxonomy-exactness} the tie census behind it.

Exactness is also a property of the travel-time model. Stepwise instance families carry arc travel times that genuinely jump at certain instants \cite{rifkiImpactSpatiotemporalGranularity2020}, which the speed-profile model of \citeauthor{ichouaVehicleDispatchingTimedependent2003} \cite{ichouaVehicleDispatchingTimedependent2003} cannot express. Hence a solver accepting both fixes what evaluation at a jump returns (section \ref{subsec:kayros-arch}). This requirement reaches the exact side as well, where an optimality claim rests both on the arithmetic of its route costs and on the tolerances of its linear-programming and pricing steps. Section \ref{subsec:kayros-exact} therefore states a certificate as a conditional computational result whose conditions are named in the stamp it carries, in the practical tradition of the exact routing literature.

\kayros{} answers with one principle, and the rest of this paper is its application: a canonical, epsilon-free checker defines the objective, every value the solver publishes is that checker's recomputation of the routes returned, and benchmark instances are distributed as byte-exact canonical data with checksums rather than as re-derivation recipes. Section \ref{sec:math} builds the arithmetic this requires, section \ref{sec:solver} holds the solver to it, and section \ref{sec:results} reports the campaign it makes comparable.

\subsection{Anytime Resolution and its Evaluation}\label{subsec:anytime-background}

An algorithm is \emph{anytime} when it can be interrupted at any moment and still return a valid answer, whose quality improves with the time it is granted \cite{zilbersteinUsingAnytimeAlgorithms1996}. Most practical optimization methods already have this shape: a branch-and-bound search streams incumbents while it proves, and a metaheuristic streams improvements until its budget runs out. However, their trajectories differ sharply: one run may spend most of its budget without a feasible solution and then improve quickly, while another may construct a good solution immediately and plateau. Combining both paradigms inside one method is a research direction of its own: \citeauthor{fontaineExactAnytimeApproach2023} \cite{fontaineExactAnytimeApproach2023, fontaineExactAnytimeHeuristic2024} build an exact and anytime search for the TDTSPTW on an anytime extension of A*, with local search, bounding and time-window propagation. It is a single-vehicle counterpart of the combination \kayros{} pursues in the multi-vehicle setting.

Treating that trajectory, rather than its last point, as the object being measured is established practice. The reference instrument is the \emph{primal integral} of Berthold \cite{bertholdMeasuringImpactPrimal2013}, which integrates the primal gap of the incumbent over elapsed time and reduces a run to one number charging both how long a solver takes to become useful and how good its final solution is. It was introduced because a final-value table cannot separate a primal heuristic producing its solution in the first seconds from one producing the same solution at the limit. Its reference is the best value known at analysis time, updated whenever a run improves it \cite{bertholdMeasuringImpactPrimal2013}, which keeps the gap nonnegative at the price of a reference state that cannot be pinned before the campaign it scores. Successors refine the construction, notably the confined variant of \citeauthor{bertholdConfinedPrimalIntegral2021} \cite{bertholdConfinedPrimalIntegral2021}, which changes the time weighting of the integral so that different parts of a run carry different influence, a design choice orthogonal to the gap function itself. The measure now serves in benchmarking practice, both as the target of automatic tuning of anytime behavior \cite{lopez-ibanezAutomaticallyImprovingAnytime2014} and as the scoring rule of the vehicle-routing track of the 12th DIMACS Implementation Challenge \cite{DIMACS2021} and subsequent competitions. That rule departs from the original reference policy: it averages an excess gap over a declared budget against a reference frozen before the experiment, and lets a solution better than that reference contribute negatively.

A single-time-limit comparison loses all of that. The ordering it reports is a property of the one budget that produced it, and nothing in it constrains the ordering at any other. For urban logistics the point is operational rather than aesthetic: a dispatcher's decision horizon is often minutes, while published comparisons are run in hours. What this lineage leaves open is therefore not whether to measure the trajectory but how to instantiate the measurement. A score or a curve still requires a gap function, a reference value together with a policy fixing when it is cut, a value for the state before the first incumbent, a convention for invalid or empty runs, and an aggregation rule over heterogeneous instances. That rule must neither let the family with the most instances carry the headline nor drop the runs that never produced a feasible solution. The profile-based aggregations of optimization benchmarking answer that last difficulty in their own way \cite{dolanBenchmarkingOptimizationSoftware2002}. Those choices, and the measurement of when they change the conclusions of a study, are the subject of a companion paper \cite{rascoussierAnytimeSolverEvaluation2026}.

None of those choices is well defined without a timing and validation contract, which an anytime comparison needs and a final-value table does not. It requires per-run trajectories composed of timestamped incumbents on a common clock. Fairness is ensured by an external evaluator, rather than each solver's own accounting, which recomputes the objective of every candidate before it may enter a run's best-so-far envelope. It also requires a per-instance reference value, since a gap needs a denominator and the optimum is unknown on most instances of interest. The instrument of section \ref{subsec:metrics} settles that contract within the lineage above as a bounded, threshold-free, signed and reference-based variant. The gap is put on a bounded scale, so the state before the first incumbent has a finite natural value rather than an arbitrary cap or an acceptance threshold below which a poor incumbent and no incumbent at all score alike. The score is signed rather than clipped at zero, so that beating the reference stays visible, as it already does under the DIMACS rule. The reference is a frozen published snapshot rather than a solver-derived bound, a choice whose cost sections \ref{subsec:benchmarks} and \ref{sec:discussion} state plainly. What is missing on this problem class is not the instrument but its application: to the best of our knowledge, no cross-solver anytime comparison has ever been reported on the TDVRPTW. The reason has more to do with available software than with metrics, as stated in section \ref{subsec:solver-landscape}.

\subsection{The Available-Solver Landscape}\label{subsec:solver-landscape}

Comparing anytime solvers on the TDVRPTW presupposes solvers to compare against, which is a stronger requirement than it sounds. We call a solver \emph{available} for such a comparison when three conditions hold together: it accepts genuinely time-dependent travel times at the model level, ideally the rich per-arc piecewise-linear functions of section \ref{sec:math} rather than a fixed time discretization; it exposes an Application Programming Interface (API) or an input format through which an instance can be fed and a solution stream read back, without modifying the solver itself; and it can be obtained and run, license included. Methods described in papers whose code was never released fail the third condition, however strong they are, and a comparison that ignores this reports on a literature rather than on software a practitioner can run. What the mainstream routing engines do and do not offer on this point, and the decade of user requests for time-dependent travel times their public issue trackers record, are documented in the companion technical report \cite{rascoussierKAYROSTechReport2026}.

Two open-source solvers meet all three conditions on this problem class. Timefold Solver \cite{timefoldSolver2026}, the continuation of OptaPlanner, is a general planning and constraint-solving platform whose vehicle-routing model consumes the rich per-arc travel-time functions of an instance without approximation, what section \ref{sec:results} calls a \emph{Tier-1} arm. It is the strongest contender in the campaign of section \ref{sec:results}. jsprit \cite{graphhopperJsprit2026} is a rich vehicle-routing toolkit consuming time-dependent transport-time callbacks, but it keeps every route at its earliest feasible departure, a documented model contract. Section \ref{subsec:contenders} reports its consequences under the \textit{Duration} objective rather than hiding them. A third open solver deserves mention for what it is not: PyVRP \cite{PyVRP} is a strong single-trajectory ILS reference for the VRPTW, and it inspired the anytime layer of section \ref{subsec:kayros-anytime}, but its model currently targets the static problem family rather than the duration-minimization TDVRPTW studied here, though following the open-source release of \kayros{} the PyVRP team has \href{https://github.com/PyVRP/PyVRP/issues/867#issuecomment-5356658546}{announced a time-dependent extension} soon for their entreprise version.

On the commercial side, Hexaly \cite{hexalyTDCVRPTW2026} publishes a model template for the TDVRPTW and grants academic licenses, which makes it available in the sense above. The solver is designed for multi-threaded use, which section \ref{subsubsec:threads} prices separately. It is also the one arm offering two genuinely different encodings of the same instance: a rich external-function encoding, which is Tier-1, and a time-sliced approximation of those same functions, the campaign's one \emph{Tier-2} arm. The campaign reports both, as two disclosed model tiers of one solver rather than as one number, and never ranks the second against the Tier-1 arms. The binding through which those external functions are evaluated is itself a disclosed choice and not a neutral implementation detail, and section \ref{subsubsec:threads} records which one the campaign uses and why. Details including the effect of the Tier-2 approximation against the Tier-1 results, and the effect of multi-threading on both arms are explored in a companion report \cite{rascoussierKAYROSTechReport2026}.

Two research codes a reader would expect to see are absent. The industrial solver behind the reference solutions of \citeauthor{blauthVehicleRoutingTimedependent2024} \cite{blauthVehicleRoutingTimedependent2024} is not published, so what can be compared is its solutions and not its behavior. Section \ref{subsec:bks} does exactly that, and section \ref{sec:discussion} states the boundary of such a comparison. The branch-price-and-cut solver of \citeauthor{lera-romeroLinearEdgeCosts2020} \cite{lera-romeroLinearEdgeCosts2020} is open and would qualify, but it is the very component \kayros{} vendors as its exact mode (section \ref{subsec:kayros-exact}), so running it as a contender would compare the solver against a part of itself. This is why the campaign of section \ref{sec:results} is run against Timefold, Hexaly and jsprit, and against nothing else. To the best of our knowledge, they are the solvers a practitioner could obtain and run on this problem class at the time of the campaign. Their pinned versions and model contracts are given with the protocol, in table \ref{tab:arms}.

\section{Problem Formulation and Time-Dependent Machinery}\label{sec:math}

This section defines the problem and the objectives on which \kayros{} and its external checker operate. It provides the duration-minimization TDVRPTW and its fleet-cost variant, the route-based formulation that exact resolution rests on, and the taxonomy of travel-time-related functions both solving modes of \kayros{} operate on. Throughout, we follow notations introduced in the literature \cite{feilletTutorialColumnGeneration2010, dabiaBranchPriceTimeDependent2013, lera-romeroEnhancedBranchPrice2018, lera-romeroLinearEdgeCosts2020, visserEfficientMoveEvaluations2020,blauthVehicleRoutingTimedependent2024}.

\subsection{Formal definition of the DM-TDVRPTW}\label{subsec:dm-def}

\paragraph{Base model and instance data.} The TDVRPTW is defined on a directed Customer-Based Graph (CBG) $\mathcal{G} = (\mathcal{V}, \mathcal{A})$, where the vertex set $\mathcal{V} = \mathcal{C} \cup \{o, d\}$ includes $n = \left| \mathcal{C} \right| \in \mathbb{N}_{>0}$ customers and a duplicated mono-depot as origin $o$ and destination $d$ vertices. The arc set $\mathcal{A} = \{ \langle i, j \rangle \in \mathcal{V} \times \mathcal{V} \:|\: i \neq j, i \neq d, j \neq o\} \subset \mathcal{V}^2$ is taken as given, and possibly derived from a Road-Network Graph (RNG) \cite{blauthVehicleRoutingTimedependent2024}. Three members of one family of time sets are used. The ambient set of scheduling timestamps is $\mathcal{T}_{\mathrm{ext}} := \mathbb{R}_{\geq 0}$, in which an arrival or a service completion past the planning horizon keeps a well-defined value. The finite planning horizon is $\mathcal{T} = [0, T] \subset \mathcal{T}_{\mathrm{ext}}$, which contains every time window. The arc-departure domain is $\mathcal{T}_{\mathrm{arc}} = [a_{\mathrm{arc}}, T_{\mathrm{arc}}] \subseteq \mathcal{T}$, the common interval on which every arc function is materialized.\footnote{Note that the canonical data contract of section \ref{subsec:kayros-arch} makes all arc functions of an instance total on exactly this interval. Fixing one common interval is a normalization of the data rather than a restriction of the model: a formulation in which each arc function carries its own departure domain, a subinterval of $\mathcal{T}$, is recovered by substituting that domain for $\mathcal{T}_{\mathrm{arc}}$ arc by arc in what follows, provided the per-arc domains keep a common lower endpoint: the machinery of section \ref{subsec:taxonomy} is built on partial functions and their domains, but its composition guarantee anchors every arc function at the shared $a_{\mathrm{arc}}$ (proposition \ref{prop:taxonomy}), so only the right endpoints may vary freely.} Durations are values of $\mathbb{R}_{\geq 0}$, so that a clock date and an elapsed duration remain distinct objects even though the underlying numerical sets coincide. A fleet of at most $K \in \mathbb{N}_{>0}$ homogeneous vehicles with capacity $Q \in \mathbb{R}_{>0}$, starting from vertex $o$, must serve all customers $v \in \mathcal{C}$, each with a demand $q_v \in \mathbb{R}_{\geq 0}$, a service time $s_v \in \mathbb{R}_{\geq 0}$ (a duration, not a point of the horizon), and a Time Window (TW) $[e_v, l_v] \subseteq \mathcal{T}$ where $e_v \leq l_v$ are the earliest and latest times of service at $v$ \cite{ichouaVehicleDispatchingTimedependent2003}. Arrival before $e_v$ is allowed as waiting time, but arrival after $l_v$ is infeasible \cite{fontaineExactAnytimeHeuristic2024}. Without loss of generality, we assume $s_o = s_d = q_o = q_d = 0$. A route departs $o$ within the depot window $[e_o, l_o]$ and must reach $d$ no later than $l_d$.\footnote{The depot window $[e_o, l_o]$ and the depot due date $l_d$ are ordinary data of $\mathcal{T}$ and are not identified with the horizon: $l_d$ need not equal $T_{\mathrm{arc}}$ or $T$ (on the Blauth2024 instances of section \ref{subsec:bks}, the depot closes 2 hours after the last materialized departure). Departures being materialized on $\mathcal{T}_{\mathrm{arc}}$ only, the effective departure window is $[e_o, l_o] \cap \mathcal{T}_{\mathrm{arc}}$, which the path recursion of the next paragraph takes as its base domain.} In this model, every valid arc $\langle i, j \rangle \in \mathcal{A}$ carries a given PWL arc TTF $\tau_{ij} : \mathcal{T}_{\mathrm{arc}} \to \mathbb{R}_{\geq 0}$ (one function per arc, the subscript binding the arc) that gives, for every departure time $t$ in the arc-departure domain, the duration it takes to go from $i$ to $j$ when leaving $i$ at $t$. Such functions are piecewise linear with at most $p \in \mathbb{N}_{>0}$ breakpoints, i.e. $\tau_{ij}$ has at most $p - 1$ pieces \cite{visserEfficientMoveEvaluations2020}. These functions are not required to be continuous: the stepwise families of section \ref{subsec:benchmarks} carry genuine travel-time jumps, which section \ref{subsec:taxonomy} formalizes through left-continuous functions. They respect the FIFO property, meaning that the associated PWL Arrival Time Function (ATF) $\alpha_{ij} : \mathcal{T}_{\mathrm{arc}} \to \mathcal{T}_{\mathrm{ext}}$, $\alpha_{ij}(t) = t + \tau_{ij}(t)$, are non-decreasing \cite{ichouaVehicleDispatchingTimedependent2003}: departing later never allows arriving strictly earlier.

\paragraph{Paths, routes, and the objective.} The definitions below use the composed-function notations of the TD literature \cite{lera-romeroLinearEdgeCosts2020, visserEfficientMoveEvaluations2020}. A reader meeting them for the first time may prefer to start with section \ref{subsec:taxonomy}, which introduces each of the functions composed here one by one and instantiates them on the toy route of figure \ref{fig:toy-route}. The set of all finite non-empty ordered subsets (sequences) of vertices $\mathcal{P}_{\mathcal{V}} \subseteq \bigcup_{k=1}^{\left| \mathcal{V} \right|} \mathcal{V}^{k}$ is called the \textit{path set}, containing any \textit{path} $\mathbf{p} = \langle v_1, v_2, \dots, v_m \rangle \in \mathcal{P}_{\mathcal{V}}$ of size $m \in \mathbb{N}_{>0}$. The path ready time function $\delta_{\mathbf{p}_k} : \dom(\mathbf{p}_k) \to \mathcal{T}_{\mathrm{ext}}$, $1 \leq k \leq m$, gives the earliest time of service completion at $v_k \in \mathbf{p}$ when the vehicle is ready to leave $v_1 \in \mathbf{p}$ at $t \in \dom(\mathbf{p}_k)$. It is defined by the recursion $\delta_{\mathbf{p}_1}(t) = t$ and $\delta_{\mathbf{p}_k} = \theta_{v_k} \circ \alpha_{v_{k-1} v_k} \circ \delta_{\mathbf{p}_{k-1}}$ over the vertex and arc functions $\theta$ and $\alpha$ formalized in section \ref{subsec:taxonomy}, the composition being the pointwise composition of partial functions defined there \cite{lera-romeroLinearEdgeCosts2020}. Note that the base case deliberately applies no $\theta_{v_1}$: the argument $t$ is the \emph{ready} time at $v_1$, i.e. service at $v_1$ (if any) is already completed at $t$. Hence for a route, whose first vertex is the depot $o$ with $s_o = 0$, $t$ is simply the depot departure time and $\dom(\mathbf{p}_1) = [e_o, l_o] \cap \mathcal{T}_{\mathrm{arc}}$. At the destination the convention is $\theta_d = \mathrm{Id}_{[0, l_d]}$: a route ends upon arrival, with no waiting and no service, and the arrival must respect the depot due date. Here, $\dom(\mathbf{p}_k) \subseteq \dom(\mathbf{p}_1)$ is the set of \textit{feasible departure times} of the path $\mathbf{p}$ up to vertex $v_k \in \mathbf{p}$, the notation $\dom(\cdot)$ on a path, a route or an arc being shorthand for the domain of its ready time function. Explicitly, $t \in \dom(\mathbf{p}_k)$ if, for every $i \in [2, k]_\mathbb{N}$, the ready time at the preceding vertex is an admissible departure on the arc, $\delta_{\mathbf{p}_{i-1}}(t) \in \mathcal{T}_{\mathrm{arc}}$, and the vehicle reaches $v_i$ before its TW deadline, $\alpha_{v_{i-1} v_i}(\delta_{\mathbf{p}_{i-1}}(t)) \leq l_{v_i}$. Both conditions together define exactly the domain of the recursive composition, which section \ref{subsec:taxonomy} characterizes. Note that when considering a whole path, the notation $\mathbf{p}_m$ is shortened to $\mathbf{p}$. A \textit{feasible path} respects the elementarity constraint $\forall i, j \in [1, m]_\mathbb{N}, i \neq j \Rightarrow v_i \neq v_j$, the capacity constraint $\left( \sum_{v \in \mathbf{p}} q_v \right) \leq Q$, and has at least one feasible departure time, $\dom(\mathbf{p}) \neq \emptyset$ \cite{lera-romeroLinearEdgeCosts2020}. A \textit{route} is any path of size $m \in \mathbb{N}_{\geq 3}$ starting and ending at the depot and serving at least one customer, $\mathbf{r} = \langle o, \mathbf{r}_{\mathcal{C}}, d \rangle \in \mathcal{P}_{\mathcal{V}}$ with $\mathbf{r}_{\mathcal{C}} \in \mathcal{C}^{m - 2}$. The set of \textit{feasible routes} is $\Omega \subset \mathcal{P}_{\mathcal{V}}$. For a feasible route $\mathbf{r}$, let $\Delta_{\mathbf{r}}^{*} \in \mathbb{R}_{\geq 0}$ be its minimum duration over its feasible departure times (a duration, not a date of the horizon, formally derived in section \ref{subsec:taxonomy}), and let $\Delta^{*} : \Omega \to \mathbb{R}_{\geq 0}$, $\Delta^{*}(\mathbf{r}) = \Delta_{\mathbf{r}}^{*}$, be the route-cost function. This paper considers the common Duration-Minimization objective (DM-TDVRPTW), a.k.a \textit{Duration}: find a solution $\mathcal{S} \subseteq \Omega$ of at most $K$ routes ($\left| \mathcal{S} \right| \leq K$) whose customer sequences form a partition of the customers, $\bigcup_{\mathbf{r} \in \mathcal{S}} \mathbf{r}_{\mathcal{C}} = \mathcal{C}$ with pairwise disjoint customer sets, while minimizing the sum of the route durations $\sum_{\mathbf{r} \in \mathcal{S}} \Delta_{\mathbf{r}}^{*}$ \cite{visserEfficientMoveEvaluations2020}. Since every route serves a customer, $\left| \mathcal{S} \right|$ counts non-empty routes only, which is also what the solver and the fleet objective below count.

\paragraph{FleetCostDuration.} Delivery-oriented settings often price the fleet itself alongside driver time. Given a fixed per-route cost $c_{\mathrm{fleet}} \in \mathbb{R}_{\geq 0}$, the \textit{FleetCostDuration} objective minimizes $c_{\mathrm{fleet}} \left| \mathcal{S} \right| + \sum_{\mathbf{r} \in \mathcal{S}} \Delta_{\mathbf{r}}^{*}$ over the same feasible solutions $\mathcal{S} \subseteq \Omega$. It is the objective under which \citeauthor{blauthVehicleRoutingTimedependent2024}~\cite{blauthVehicleRoutingTimedependent2024} publish their reference solutions. A fixed cost of $c_{\mathrm{fleet}} = 36 \cdot 10^{6}$ milliseconds per used route reproduces their economic model exactly (\$200 per vehicle at \$20 per working hour, the dollar value being the millisecond cost divided by $180\,000$). Pricing the fleet changes the search problem substantially. Section \ref{subsec:kayros-fleet} presents the fleet-aware search this calls for. The two objectives are deliberately not symmetric in this paper: \textit{Duration} is the primary axis, on which the cross-solver campaign of section \ref{sec:results} and the exact component operate, while \textit{FleetCostDuration} enters the model so that the Blauth2024 results of section \ref{subsec:bks} are self-supported. The exact component does not support it. In the remainder, $z(\mathcal{S})$ denotes the value at a solution $\mathcal{S}$ of whichever of the two objectives is selected, and $z$ alone a solution cost so evaluated.

\subsection{Set-partitioning formulation and pricing}\label{subsec:spf}

The exact component of \kayros{} rests on the classic route-based Dantzig-Wolfe reformulation of the TDVRPTW \cite{feilletTutorialColumnGeneration2010, dabiaBranchPriceTimeDependent2013, lera-romeroEnhancedBranchPrice2018, lera-romeroLinearEdgeCosts2020}. Let $x_v^{\mathbf{r}} \in \{0, 1\}$ indicate that route $\mathbf{r} \in \Omega$ serves customer $v$, let $c_{\mathbf{r}}$ be the route cost ($c_{\mathbf{r}} = \Delta_{\mathbf{r}}^{*}$ under \textit{Duration}, $c_{\mathbf{r}} = c_{\mathrm{fleet}} + \Delta_{\mathbf{r}}^{*}$ under \textit{FleetCostDuration}), and let $\lambda_{\mathbf{r}} \in \{0, 1\}$ decide whether route $\mathbf{r}$ is selected. The set-partitioning master problem reads:

\begin{align}
\min \quad & \sum_{\mathbf{r} \in \Omega} c_{\mathbf{r}} \lambda_{\mathbf{r}} & \label{eq:sp-obj} \\
\text{s.t.} \quad & \sum_{\mathbf{r} \in \Omega} x_v^{\mathbf{r}} \lambda_{\mathbf{r}} = 1 & \forall v \in \mathcal{C} \label{eq:sp-partition} \\
& \sum_{\mathbf{r} \in \Omega} \lambda_{\mathbf{r}} \leq K & \label{eq:sp-fleet} \\
& \lambda_{\mathbf{r}} \in \{0, 1\} & \forall \mathbf{r} \in \Omega \label{eq:sp-integrality}
\end{align}

Since $\Omega$ grows factorially with $n$, the LP relaxation of \eqref{eq:sp-obj}--\eqref{eq:sp-integrality} is solved by column generation over a restricted route pool $\Omega_1 \subset \Omega$, giving the Restricted Master Problem. Writing $\pi_v$ for the dual values of the partition constraints \eqref{eq:sp-partition} and $\mu \leq 0$ for the dual of the fleet bound \eqref{eq:sp-fleet}, the pricing problem searches for a feasible route of minimum reduced cost $\bar{c}_{\mathbf{r}} = c_{\mathbf{r}} - \sum_{v \in \mathcal{C}} x_v^{\mathbf{r}} \pi_v - \mu$. Any route with $\bar{c}_{\mathbf{r}} < 0$ may enter the pool, and the relaxation is solved when none exists. In the time-dependent setting, this pricing problem is the TDESPPRC of section \ref{subsec:exact-methods}, which is NP-hard \cite{dror1994note}. It is solved by a labeling dynamic program whose states carry the very ready-time functions formalized below: this is where the function taxonomy and the composition algorithm of this section do their computational work \cite{dabiaBranchPriceTimeDependent2013, lera-romeroLinearEdgeCosts2020}. Integrality is restored by branching, with cutting planes strengthening the relaxation (Branch-Price-and-Cut). The \kayros{} implementation and its certificate semantics are presented in section \ref{sec:solver}.

\subsection{A taxonomy of travel-time-related functions}\label{subsec:taxonomy}

This section formalizes the travel-time-related functions underpinning any TD routing solver, from the given arc travel times up to complete route durations. Each function is presented with its signature and the properties the rest of the paper relies on, the full property and complexity statements and all proofs being collected in appendix \ref{app:taxonomy}. Figure \ref{fig:toy-route} illustrates the whole taxonomy on a toy route. Throughout the taxonomy, subscripts index families of functions (one $\tau_{ij}$ per arc $\langle i, j \rangle \in \mathcal{A}$, one $\theta_i$ per vertex $i \in \mathcal{V}$, one $\delta_{\mathbf{r}}$ per route $\mathbf{r} \in \Omega$): each signature below therefore types a single member of its family over its own domain, the index being bound by the subscript rather than passed as an argument.

\paragraph{Representation and represented object.} The literature composes Non-Decreasing Continuous PieceWise Linear Functions (NDCPWLF) \cite{visserEfficientMoveEvaluations2020, blauthVehicleRoutingTimedependent2024}. The canonical checker of section \ref{subsec:kayros-arch} must also evaluate stepwise instances whose arc travel times genuinely jump \cite{rifkiImpactSpatiotemporalGranularity2020}, so the objects it manipulates are not continuous. The first task of this section is therefore to separate the stored representation from the single-valued function it stands for. A \emph{chain} is a finite sequence $\beta = ((x_1, y_1), \dots, (x_\phi, y_\phi)) \in (\mathbb{R}^2)^\phi$ of $\phi \in \mathbb{N}_{\geq 0}$ points with $x_k \leq x_{k+1}$, $y_k \leq y_{k+1}$ and no two consecutive points equal, its length $\phi$ being its \emph{breakpoint count}. Every non-decreasing function $f$ of the taxonomy is stored as a chain $\beta_f$ of $\phi_f$ breakpoints, i.e. $2 \phi_f$ scalars, while its two non-monotone members, the arc TTF and the route duration function, are read off the corresponding ready-time chains rather than stored. A maximal group of consecutive points sharing an abscissa $t_0$, with ordinates $y_{k_0} < \dots < y_{k_1}$, is a \emph{vertical run} with \emph{jump abscissa} $t_0$ and \emph{connector} $\{t_0\} \times [y_{k_0}, y_{k_1}]$. The chain with $\phi = 0$ is the \emph{empty chain}, written $\bot$.

\begin{definition}[Selected function]\label{def:selected}
The \emph{selected function} $\mathsf{s}(\beta)$ of a non-empty chain is the function on $[x_1, x_\phi]$ that takes, at an abscissa $t = x_k$ of the chain, the \emph{smallest} ordinate carried at $t$, and, at any other $t$, the linear interpolation between the two consecutive chain points enclosing $t$. The sans-serif face is deliberate: $\mathsf{s}$ is the one named operator of this paper, and the lowercase letter is otherwise the service time $s_v$. The empty chain selects the empty partial function, $\mathsf{s}(\bot) = \bot$. A chain $\beta$ \emph{represents} $f$ when $\mathsf{s}(\beta) = f$.
\end{definition}

The connector of a vertical run is part of the representation, not of the graph of the represented function: at a jump abscissa the function takes the bottom of the connector. The values strictly above it are taken only after $t_0$, or never when the run is the last group of the chain (a \emph{trailing vertical}). This is the evaluation rule of the canonical checker, and it makes the represented functions left-continuous.

\begin{definition}[The function classes]\label{def:ndlcf}
A function $h : [a_h, u_h] \to \mathbb{R}$, with $a_h, u_h \in \mathbb{R}$ and $a_h \leq u_h$, belongs to the class $\mathcal{F}$ of left-continuous PWL functions if it is left-continuous at every $t \in (a_h, u_h]$ and affine on each of finitely many open pieces cutting $[a_h, u_h]$. We write $\mathcal{F}_{\text{ND}} \subset \mathcal{F}$ for the non-decreasing members, the Non-Decreasing Left-Continuous PieceWise Linear Functions (NDLCF, shortened from NDLCPWLF so as not to read as a variant of NDCPWLF), and $\mathcal{F}_{\text{ND}}^{\text{c}} \subset \mathcal{F}_{\text{ND}}$ for the continuous ones, the Non-Decreasing Continuous PieceWise Linear Functions (NDCPWLF). We write $\dom(h) = [a_h, u_h]$ and $h(t^+)$ for the right limit at $t < u_h$. The empty partial function $\bot$ is adjoined to all three classes as an absorbing element of composition. Routing times and durations remain non-negative.
\end{definition}

Chains represent exactly the NDLCFs (lemma \ref{lemma:chains} in appendix \ref{app:taxonomy}), but not uniquely: a chain may carry \emph{redundant} points, lying inside a straight run (a vertical run included), and a trailing vertical, all invisible to $\mathsf{s}$. We deliberately do \emph{not} make any normal form a class invariant. The removal of redundant points is an explicit, optional operation whose floating-point semantics section \ref{subsec:composition} discusses, and the canonical cost semantics keeps compositions un-normalized. Breakpoint counts below are chain lengths, duplicate abscissae counted and no normalization assumed. This counting convention is coarser than those of \cite{visserEfficientMoveEvaluations2020, blauthVehicleRoutingTimedependent2024}, and it is all the complexity statements need. An NDCPWLF is an NDLCF with no jump, and its canonical chain has no vertical run. The arc TTF and the route duration function belong to $\mathcal{F}$ but need not belong to $\mathcal{F}_{\text{ND}}$, since neither is required to be non-decreasing.

\paragraph{Composition.} For $f, g \in \mathcal{F}_{\text{ND}}$, the composition $f \circ g$ is the pointwise composition of partial functions, with $\dom(f \circ g) = \{t \in \dom(g) : g(t) \in \dom(f)\}$ and $(f \circ g)(t) = f(g(t))$. It is associative by definition, whatever the representations. The whole route algebra rests on one hypothesis, which every composition performed by the checker and the solver satisfies (proposition \ref{prop:taxonomy}):

\begin{equation}\label{eq:H}
\text{(H)}\qquad a_f \leq g(a_g),
\end{equation}

\noindent the domain of the outer function does not start strictly above the first value of the inner one. Since $g$ is non-decreasing, (H) ensures $g(t) \geq a_f$ throughout $\dom(g)$, leaving only the upper bound of $\dom(f)$ to restrict the composition domain. It also keeps the outer domain away from the connectors of the inner chain, which carry values the inner function never takes: an outer domain starting inside such a connector makes the pointwise domain left-open at the jump, or empty. In that situation, a merge that reads the connector reports the jump abscissa as feasible. (H) is sufficient and not necessary. It is nevertheless a property of the functions rather than of one representation, holds for every chain representing them, and is closed along a fold (theorem \ref{theorem:closure}).

\begin{theorem}[Closure under (H)]\label{theorem:closure}
Let $f, g \in \mathcal{F}_{\text{ND}}$ be non-empty, with $\dom(f) = [a_f, u_f]$ and $\dom(g) = [a_g, u_g]$, and satisfy (H). Then $f \circ g$ is empty if and only if $g(a_g) > u_f$, and otherwise $\dom(f \circ g) = [a_g, u_{f \circ g}]$ with $u_{f \circ g} = \max\{t \in \dom(g) : g(t) \leq u_f\}$, $f \circ g \in \mathcal{F}_{\text{ND}}$, and $a_{f \circ g} = a_g$, so that (H) for $(f \circ g) \circ g'$ is (H) for $g \circ g'$.
\end{theorem}

\begin{theorem}[Exact composition]\label{theorem:merge}
Let $f, g \in \mathcal{F}_{\text{ND}}$ be non-empty and satisfy (H), and let $\beta_f$, $\beta_g$ be any chains representing them. Algorithm \ref{algo:compose_NDLCF} of section \ref{subsec:composition} returns $\bot$ if and only if $f \circ g$ is empty, and otherwise a chain representing $f \circ g$, of at most $\phi_f + \phi_g$ breakpoints, in $\mathcal{O}(\phi_f + \phi_g)$ operations. The returned chain is the canonical chain of $f \circ g$ up to redundant points and a trailing vertical.
\end{theorem}

Both theorems are proved in appendix \ref{app:taxonomy}. The image-domain intersection that the continuous theorem of \citeauthor{visserEfficientMoveEvaluations2020} \cite{visserEfficientMoveEvaluations2020} uses as a non-emptiness test reads, on chains, $\max(a_f, y^g_1) \leq \min(u_f, y^g_{\phi_g})$, and it is exactly the test algorithm \ref{algo:compose_NDLCF} performs. Under (H) it reduces to $g(a_g) \leq u_f$. Without (H) it is not an attained-value test: on the smallest example appendix \ref{app:taxonomy-proofs} gives, the merge returns a non-empty chain for an empty composition, and two bracketings of one chain even disagree. The subclasses actually composed by route evaluation satisfy (H) for every pair, which lets the solver of section \ref{sec:solver} evaluate any bracketing by the event merge:

\begin{proposition}[Temporal maps]\label{prop:temporal}
For an anchor $a \in \mathbb{R}$, let $\mathcal{F}^{\text{out}}_a$ be the NDLCFs with $f(t) \geq t$ and $a_f \leq a$ (outer maps), $\mathcal{F}^{\text{in}}_a$ those with $g(t) \geq t$ and $a_g \geq a$ (inner maps), and $\mathcal{F}^{\text{anc}}_a = \mathcal{F}^{\text{out}}_a \cap \mathcal{F}^{\text{in}}_a$ (anchored maps), $\bot$ belonging to all three. If $f \in \mathcal{F}^{\text{out}}_a$ and $g \in \mathcal{F}^{\text{in}}_a$ then (H) holds and $f \circ g \in \mathcal{F}^{\text{in}}_a$, with $f \circ g \in \mathcal{F}^{\text{anc}}_a$ when $g \in \mathcal{F}^{\text{anc}}_a$. Consequently, for $f_1 \in \mathcal{F}^{\text{in}}_a$ and $f_2, \dots, f_k \in \mathcal{F}^{\text{anc}}_a$, every bracketing of $f_k \circ \dots \circ f_1$ satisfies (H) at each internal composition, all bracketings are one and the same partial function, and algorithm \ref{algo:compose_NDLCF} evaluated along any of them returns a chain representing it.
\end{proposition}

\begin{proposition}[The taxonomy is temporal]\label{prop:taxonomy}
Let $a = a_{\mathrm{arc}} \geq 0$. Every customer function $\theta_i$, $i \in \mathcal{C}$, and the destination clamp $\theta_d = \mathrm{Id}_{[0, l_d]}$ belong to $\mathcal{F}^{\text{out}}_a$, every arc ATF $\alpha_{ij}$ belongs to $\mathcal{F}^{\text{anc}}_a$, and the origin identity $\theta_o = \mathrm{Id}_{[e_o, l_o] \cap \mathcal{T}_{\mathrm{arc}}}$, the depot departure, belongs to $\mathcal{F}^{\text{in}}_a$. Hence the sequential fold of the checker, which composes one outer map onto an inner accumulator at each step, satisfies (H) at every step. Every arc ready time function $\delta_{ij}$ (arc RTF, defined below) and every leaf $\mathit{lf}_k$, $k \geq 2$, of section \ref{sec:solver} is in $\mathcal{F}^{\text{anc}}_a$ while $\mathit{lf}_1$, which contains the depot departure, is in $\mathcal{F}^{\text{in}}_a$, so proposition \ref{prop:temporal} applies to every bracketing of the leaf list, which is what the composition tree and the route trees of section \ref{sec:solver} evaluate. Continuity is required nowhere.
\end{proposition}

\begin{figure}[htbp]
\centering
\includegraphics[width=\textwidth]{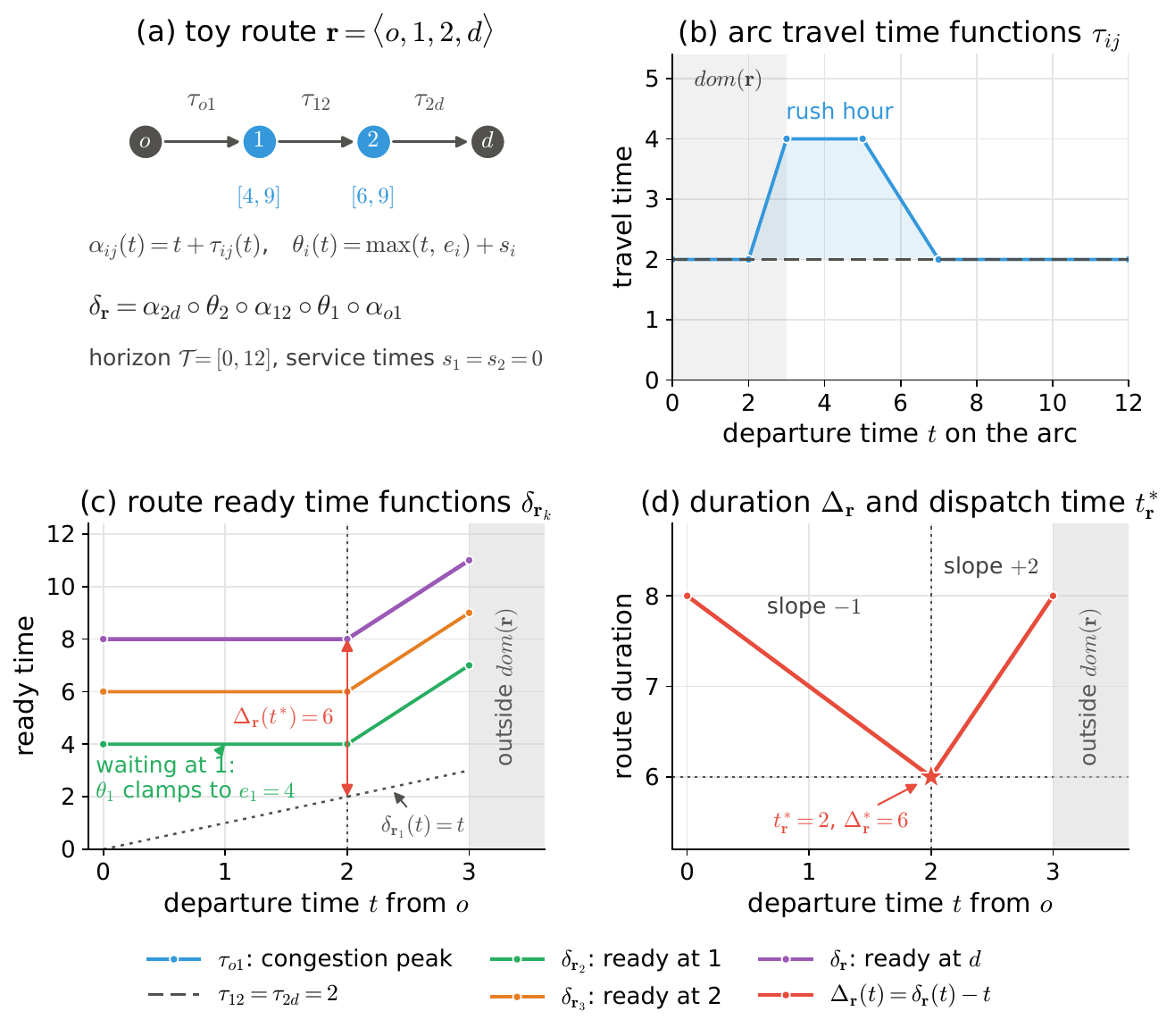}
\caption{The function taxonomy on a toy route $\mathbf{r} = \langle o, 1, 2, d \rangle$ over the horizon $\mathcal{T}_{\mathrm{arc}} = \mathcal{T} = [0, 12]$, with zero service times and TWs $[4, 9]$ at customer 1 and $[6, 9]$ at customer 2. (a) The route and the composition chain producing its route ready time function (RTF). (b) The arc TTFs $\tau_{ij}$: a rush-hour bump on $\tau_{o1}$, constant travel times elsewhere. (c) The prefix route RTFs $\delta_{\mathbf{r}_k}$ over $\dom(\mathbf{r}) = [0, 3]$: the TW at customer 1 forces waiting, and the resulting flat clamp propagates through every downstream composition. (d) The route duration function (RDF) $\Delta_{\mathbf{r}}(t) = \delta_{\mathbf{r}}(t) - t$. Its minimum over the breakpoints gives the best duration $\Delta_{\mathbf{r}}^{*} = 6$, and its earliest minimizer gives the dispatch time $t_{\mathbf{r}}^{*} = 2$. Every value is computed by the event merge of section \ref{subsec:composition}. The figure regenerates deterministically from \texttt{scripts/gen\_toy\_figure.py}.}
\label{fig:toy-route}
\end{figure}

Time-dependent solvers build on different members of the taxonomy below \cite{dabiaBranchPriceTimeDependent2013, lera-romeroLinearEdgeCosts2020, visserEfficientMoveEvaluations2020, blauthVehicleRoutingTimedependent2024}. The preprocessing of \citeauthor{lera-romeroEnhancedBranchPrice2018} \cite{lera-romeroEnhancedBranchPrice2018} can be understood as precomputing all the arc Ready Time Functions (RTF) $\delta_{ij}$ defined below. It then recomputes modified TTFs $\tau'$ from those arc RTFs, which is unnecessary and can introduce errors under \textit{Travel Time} objectives, since it incorporates waiting times before TW openings into the TTFs. \citeauthor{panHybridAlgorithmTimedependent2021} \cite{panHybridAlgorithmTimedependent2021} and \citeauthor{visserEfficientMoveEvaluations2020} \cite{visserEfficientMoveEvaluations2020} worked with route RTFs $\delta_\mathbf{r}$, the latter introducing the Balanced Binary Tree (BBT) structure $\bbt^{\mathbf{r}}$ for efficient local-search move evaluation over the neighbors of a given route $\mathbf{r}$. \citeauthor{blauthVehicleRoutingTimedependent2024} \cite{blauthVehicleRoutingTimedependent2024} focused on arc and route ATFs $\alpha$ in a similar fashion, continuous ones on a lower-unbounded domain made total on the left by an initially-constant piece. Proposition \ref{prop:continuous} in appendix \ref{app:taxonomy} replaces that device by hypothesis (H) with a common lower endpoint.

\paragraph{Arc Travel Time Function (TTF)}

\[
\begin{matrix}
\tau_{ij} : & \mathcal{T}_{\mathrm{arc}} & \rightarrow & \mathbb{R}_{\geq 0} \\
& t & \mapsto & \tau_{ij}(t)
\end{matrix}
\]

The function $\tau_{ij}$ \cite{dabiaBranchPriceTimeDependent2013, lera-romeroLinearEdgeCosts2020} representing the time it takes to traverse arc $\langle i, j \rangle \in \mathcal{A}$ when departing at $t \in \mathcal{T}_{\mathrm{arc}}$ is a left-continuous PWL function over $\mathcal{T}_{\mathrm{arc}}$ whose values are durations, hence the codomain $\mathbb{R}_{\geq 0}$ rather than a time set. It belongs to $\mathcal{F}$ but need not be non-decreasing, so membership of $\mathcal{F}_{\text{ND}}$ is not required: the property the taxonomy uses is FIFO in the strong sense, i.e. the associated arrival function $\alpha_{ij}(t) = t + \tau_{ij}(t)$ is non-decreasing, which bounds the slope of $\tau_{ij}$ below by $-1$ \cite{fontaineExactAnytimeHeuristic2024}. Its number of breakpoints is bounded by the instance-wide parameter $p \in \mathbb{N}_{>0}$. Being given by the instance, it costs $\mathcal{O}(1)$ in time and $\mathcal{O}(p)$ in memory.

\paragraph{Arc Arrival Time Function (ATF)}

\[
\begin{matrix}
\alpha_{ij} : & \mathcal{T}_{\mathrm{arc}} & \rightarrow & \mathcal{T}_{\mathrm{ext}} \\
& t & \mapsto & \alpha_{ij}(t) = t + \tau_{ij}(t)
\end{matrix}
\]

The function $\alpha_{ij}$ \cite{visserEfficientMoveEvaluations2020, blauthVehicleRoutingTimedependent2024, fontaineExactAnytimeHeuristic2024} representing the time of arrival at $j$ when departing $i$ of the arc $\langle i, j \rangle \in \mathcal{A}$ at $t \in \mathcal{T}_{\mathrm{arc}}$ is an NDLCF over $\mathcal{T}_{\mathrm{arc}}$, FIFO compliance being exactly its monotonicity. It satisfies $\alpha_{ij}(t) \geq t$, has at most $p$ breakpoints, and is computed in $\mathcal{O}(1)$ from the TTF with no memory duplicated. It is the arc member of $\mathcal{F}^{\text{anc}}_{a_{\mathrm{arc}}}$ in proposition \ref{prop:taxonomy}. Its codomain is $\mathcal{T}_{\mathrm{ext}}$ rather than $\mathcal{T}$, since late departures may arrive past the horizon end $T$. Note that the triangle inequality need not hold \cite{lera-romeroLinearEdgeCosts2020}. On stepwise instances its chain carries vertical runs at the jump instants.

\paragraph{Vertex TW Ready Time Lower Bound Function (vertex RTF)}

For a customer $i \in \mathcal{C}$,
\[
\begin{matrix}
\theta_i : & [0, l_i] & \rightarrow & [e_i + s_i, l_i + s_i] \\
& t & \mapsto &
\theta_i(t) = \begin{cases}
e_i + s_i & \text{if } t < e_i \\
t + s_i & \text{if } t \in [e_i, l_i]
\end{cases}
\end{matrix}
\]

The function $\theta_i$ \cite{visserEfficientMoveEvaluations2020} is non-decreasing by construction and encodes both the time window and the service duration at a customer $i \in \mathcal{C}$: a flat waiting piece when $e_i > 0$, then a slope-one piece, giving one to three breakpoints in all\footnote{Three when $0 < e_i < l_i$, two when $e_i = 0 < l_i$ or $0 < e_i = l_i$, and one in the degenerate single-instant window $e_i = l_i = 0$, where the function is the single point $(0, s_i)$. The degenerate case changes neither the constant-time nor the constant-memory conclusion. Appendix \ref{app:taxonomy} gives the case analysis.}. It is stored as the time window itself, in constant time and memory. Its domain starts at $0 \leq a_{\mathrm{arc}}$ and it satisfies $\theta_i(t) \geq t$, so it is an outer map of $\mathcal{F}^{\text{out}}_{a_{\mathrm{arc}}}$ in proposition \ref{prop:taxonomy}. At the two depots the two-case definition above is deliberately not instantiated: it would clamp an early ready time to the window start, whereas the depot conventions of section \ref{subsec:dm-def} let the depot windows act through domain restriction only. Both depot maps are identities. At the destination, $\theta_d = \mathrm{Id}_{[0, l_d]}$ (the route ends upon arrival, no later than $l_d$), also an outer map. At the origin, $\theta_o = \mathrm{Id}_{[e_o, l_o] \cap \mathcal{T}_{\mathrm{arc}}}$, the identity on the feasible departure window, which is exactly the base case $\delta_{\mathbf{p}_1}$ of the path recursion and the innermost leaf of the route folding of section \ref{subsec:kayros-folding} (its domain starts at $\max(e_o, a_{\mathrm{arc}}) \geq a_{\mathrm{arc}}$, so it sits in the inner position of proposition \ref{prop:taxonomy}, not the outer one).

\paragraph{Arc Ready Time Function (arc RTF)}

\[
\begin{matrix}
\delta_{ij} : & \dom(\langle i, j \rangle) & \rightarrow & [e_j + s_j, l_j + s_j] \\
& t & \mapsto & \delta_{ij}(t) = (\theta_j \circ \alpha_{ij})(t)
\end{matrix}
\]

The pair $(\theta_j, \alpha_{ij})$ satisfies (H) since $a_{\theta_j} = 0 \leq a_{\mathrm{arc}} \leq \alpha_{ij}(a_{\mathrm{arc}})$, so by theorem \ref{theorem:closure} $\delta_{ij}$ is an NDLCF over $\dom(\langle i, j \rangle) = [a_{\mathrm{arc}}, u_{\delta_{ij}}]$, the largest initial segment of $\mathcal{T}_{\mathrm{arc}}$ on which the arrival at $j$ respects its deadline $l_j$. The definition of \citeauthor{visserEfficientMoveEvaluations2020} \cite{visserEfficientMoveEvaluations2020} leaves that domain restriction implicit. The domain is empty exactly when $\alpha_{ij}(a_{\mathrm{arc}}) > l_j$, in which case the arc can never be traversed feasibly and it is removed in preprocessing. Each arc RTF has $\mathcal{O}(p)$ breakpoints, so precomputing all of them costs $\mathcal{O}(n^2 p)$ in time and memory.

\paragraph{Route Ready Time Function (route RTF)} Instantiating the path recursion of section \ref{subsec:dm-def} on a route $\mathbf{r} = \langle v_1 = o, \dots, v_m = d \rangle \in \Omega$ gives the prefix functions $\delta_{\mathbf{r}_k} = \delta^{\mathbf{r}}_{v_1 v_k}$ and the route ready time function of the complete route:

\[
\begin{matrix}
\delta_{\mathbf{r}} : & \dom(\mathbf{r}) & \rightarrow & [0, l_d] \\
& t & \mapsto & \delta_{\mathbf{r}}(t) = \delta_{\mathbf{r}_m}(t) = \delta_{od}^{\mathbf{r}}(t)
\end{matrix}
\]

By propositions \ref{prop:temporal} and \ref{prop:taxonomy}, each prefix function $\delta_{\mathbf{r}_k}$ \cite{dabiaBranchPriceTimeDependent2013, lera-romeroLinearEdgeCosts2020, visserEfficientMoveEvaluations2020} is an NDLCF of $\mathcal{F}^{\text{in}}_{a_{\mathrm{arc}}}$. The route RTF $\delta_{\mathbf{r}}$ is an NDLCF over its recursively computed domain $\dom(\mathbf{r}) \subseteq [e_o, l_o] \cap \mathcal{T}_{\mathrm{arc}}$, non-empty exactly when $\mathbf{r}$ is feasible (section \ref{subsec:dm-def}), and its values are arrival times at $d$, hence at most $l_d$. Since all bracketings of a temporal chain coincide, $\delta_{\mathbf{r}}$ is also the chain $(\delta_{v_{m-1} d} \circ \dots \circ \delta_{o v_2})(t)$ of arc RTFs restricted to the departure window. Following \citeauthor{visserEfficientMoveEvaluations2020} \cite{visserEfficientMoveEvaluations2020} and \citeauthor{blauthVehicleRoutingTimedependent2024} \cite{blauthVehicleRoutingTimedependent2024}, that chain is evaluated through a BBT $\bbt^{\mathbf{r}}$ whose $m - 1$ leaves are the arc RTFs and whose root is $\delta_{\mathbf{r}}$, over $\mathit{BBTL} = \left\lceil \log_2(m - 1) \right\rceil + 1$ levels, following the literature recursion:

\[
\delta_{v_i v_j}^{\mathbf{r}} = \begin{cases}
\delta_{v_i v_j}(t) & \text{if } j = i + 1 \text{ (at level } \ell = 1 \text{)} \\
(\delta_{v_k v_j}^{\mathbf{r}} \circ \delta_{v_i v_k}^{\mathbf{r}})(t) & \text{otherwise (for some } i < k < j \text{, at level } \ell \in [2, \mathit{BBTL}]_\mathbb{N} \text{)}
\end{cases}
\]

Every node of the tree is in $\mathcal{F}^{\text{anc}}_{a_{\mathrm{arc}}}$, so (H) holds at every internal composition and theorem \ref{theorem:merge} applies at each node. The complexity of computing this structure, and hence $\delta_\mathbf{r}$, is given by the following theorem, proved in appendix \ref{app:taxonomy}.

\begin{theorem}[BBT-based calculation of route RTFs]\label{theorem:comp_of_rrtf}
Let $\mathbf{p} \in \mathcal{P}_{\mathcal{V}}$ be a path with $m = \left| \mathbf{p} \right| \in \mathbb{N}_{\geq 2}$ vertices, whose $m - 1$ arc RTFs have $\mathcal{O}(p)$ breakpoints each. The path RTF $\delta_{\mathbf{p}}$ can be computed in $\mathcal{O}(m p \log_2(m))$ time and memory by building the BBT $\bbt^{\mathbf{p}}$, and the resulting representation of $\delta_{\mathbf{p}}$ has $\mathcal{O}(m p)$ breakpoints.
\end{theorem}

\paragraph{Route Duration Function (RDF)}

\[
\begin{matrix}
\Delta_{\mathbf{r}} : & \dom(\mathbf{r}) & \rightarrow & \mathbb{R}_{\geq 0} \\
& t & \mapsto & \Delta_{\mathbf{r}}(t) = \delta_{\mathbf{r}}(t) - t
\end{matrix}
\]

The function $\Delta_\mathbf{r}$ \cite{lera-romeroLinearEdgeCosts2020, panHybridAlgorithmTimedependent2021, visserEfficientMoveEvaluations2020} is a left-continuous PWL function giving the duration (a value of $\mathbb{R}_{\geq 0}$, not a date of the horizon) of route $\mathbf{r}$ when the vehicle leaves the depot at $t$. It belongs to $\mathcal{F}$ but need not be non-decreasing, so it is the only function the route algebra builds that may fall outside $\mathcal{F}_{\text{ND}}$, the given arc TTF being the other potentially non-monotone member of the taxonomy. Its minimum $\Delta_{\mathbf{r}}^{*} = \min_{t \in \dom(\mathbf{r})} \Delta_\mathbf{r}(t)$ is the best route duration, the quantity minimized by the DM-TDVRPTW objective of section \ref{subsec:dm-def}. It is attained at the route \textit{dispatch time} $t_\mathbf{r}^{*} = \min \operatorname{arg\,min}_{t \in \dom(\mathbf{r})} \Delta_\mathbf{r}(t)$, the earliest minimizer, which solves the Optimal Starting Time Problem \cite{hashimotoIteratedLocalSearch2007, visserEfficientMoveEvaluations2020}. Both are read off the chain of $\delta_{\mathbf{r}}$ in $\mathcal{O}(m p)$ time, as the minimum of $y_k - x_k$ over its breakpoints and the earliest abscissa attaining it. Lemma \ref{lemma:minimum} in appendix \ref{app:taxonomy} proves that reading exact and independent of the representing chain.

Every route cost in the solver is a composition of these functions, so their complexity and the exactness of their implementation are load-bearing rather than incidental \cite{dabiaBranchPriceTimeDependent2013, lera-romeroLinearEdgeCosts2020, visserEfficientMoveEvaluations2020, panHybridAlgorithmTimedependent2021, blauthVehicleRoutingTimedependent2024}. Appendix \ref{app:taxonomy} states the properties and the complexity of each function in full.

\subsection{Exact composition of NDLCFs}\label{subsec:composition}

Every route cost under the \textit{Duration} objective is produced by composing NDLCFs, so this single operation decides the arithmetic in which a route cost, a published solution and an optimality claim refer to the same object \cite{visserEfficientMoveEvaluations2020, blauthVehicleRoutingTimedependent2024}. The operation itself is prior art \cite{ordaShortestpathMinimumdelayAlgorithms1990, deanShortestPathsFIFO2004, foschiniComplexityTimeDependentShortest2011, batzMinimumTimedependentTravel2013} (section \ref{subsec:td-literature}). Within TD routing, \citeauthor{visserEfficientMoveEvaluations2020} \cite{visserEfficientMoveEvaluations2020} describe it inside the proof of their continuous theorem. \citeauthor{blauthVehicleRoutingTimedependent2024}~\cite{blauthVehicleRoutingTimedependent2024} come closest, being to the best of our knowledge the first to lay out its principles at proposition level with complexity and correctness, yet they stop short of a precise algorithm. Both restrict themselves to continuous functions. Our contribution is therefore not the operation but its precision and its scope: to the best of our knowledge, the first complete pseudocode-level description of this composition, on the left-continuous class the canonical stepwise benchmarks require, with its edge cases handled explicitly. It is proved equivalent to the pointwise composition (theorem \ref{theorem:merge}) and coupled with the analysis of when it is evaluated exactly in IEEE-754 double precision. The listing is the one shipped in \kayros{} and the canonical checker. An earlier two-pointer formulation of the same sweep, which we extended from \cite{lera-romeroLinearEdgeCosts2020, visserEfficientMoveEvaluations2020} with an active normalization and explicit tail-flush steps, is the one benchmarked in section \ref{subsec:ndcpwlf-benchmark} and appendix \ref{app:composition}. The full benchmarking code is available \href{https://github.com/0nyr/pwlf_compare}{on GitHub}.

\noindent\textbf{Notations:} A chain $\beta_f = ((x^f_1, y^f_1), \dots, (x^f_{\phi_f}, y^f_{\phi_f}))$ represents $f \in \mathcal{F}_{\text{ND}}$ as in section \ref{subsec:taxonomy}, both coordinate sequences being non-decreasing. Unlike \cite{lera-romeroLinearEdgeCosts2020, visserEfficientMoveEvaluations2020, blauthVehicleRoutingTimedependent2024}, we allow both horizontal runs (equal ordinates, modeling waiting) and vertical runs (equal abscissae, modeling a jump). The latter arise from stepwise data and, after many compositions, from floating-point rounding, with consequences for reproducibility if they are not treated as first-class objects. The listing forms no slope and no intercept: the only arithmetic it performs is a linear interpolation on one input piece, in the abscissa parameter for $f$ and in the ordinate parameter for $g$. A vertical run of $g$ is thereby pulled back to its jump abscissa without any division by zero.

\begin{algorithm}
\setstretch{1}\footnotesize
\caption{$\textit{compose}(f, g)$: composition of two NDLCFs by an event merge over the value axis}\label{algo:compose_NDLCF}
\begin{algorithmic}[1]
\Require Chains $\beta_f$ and $\beta_g$ representing $f, g \in \mathcal{F}_{\text{ND}}$ with $a_f \leq g(a_g)$ (hypothesis (H))
\Ensure A chain $\beta_h$ representing $h = f \circ g$, or $\bot$ when $h$ is empty
\State \textbf{if} $\phi_f = 0$ \textbf{or} $\phi_g = 0$ \textbf{then} \Return $\bot$ \textbf{end if}
\State $\mathit{lo} := \max(x^f_1, y^g_1)$ and $\mathit{hi} := \min(x^f_{\phi_f}, y^g_{\phi_g})$ \Comment{Value band of the merge, equal to $[g(a_g), \mathit{hi}]$ under (H)}
\State \textbf{if} $\mathit{lo} > \mathit{hi}$ \textbf{then} \Return $\bot$ \textbf{end if} \Comment{Under (H): $g(a_g) > u_f$, the composition is empty}
\State $i := \min\{i : x^f_i \geq \mathit{lo}\}$ and $j := \min\{j : y^g_j \geq \mathit{lo}\}$, $\beta_h := \bot$
\While{($i \leq \phi_f$ and $x^f_i \leq \mathit{hi}$) or ($j \leq \phi_g$ and $y^g_j \leq \mathit{hi}$)}
    \State $\mathit{ev} :=$ the smallest of the available $x^f_i$ and $y^g_j$ \Comment{Next event value}
    \State $\mathit{ys} := (y^f_i, y^f_{i+1}, \dots)$ over all points with $x^f_i = \mathit{ev}$, advancing $i$ past them \Comment{Ordinates of $f$ at $\mathit{ev}$, several on a vertical run}
    \State $\mathit{ts} := (x^g_j, x^g_{j+1}, \dots)$ over all points with $y^g_j = \mathit{ev}$, advancing $j$ past them \Comment{Abscissae of $g$ at value $\mathit{ev}$, several on a plateau}
    \State \textbf{if} $\mathit{ys} = ()$ \textbf{then} $\mathit{ys} := (f(\mathit{ev}))$ by linear interpolation on the piece $x^f_{i-1} < \mathit{ev} < x^f_i$ \textbf{end if}
    \State \textbf{if} $\mathit{ts} = ()$ \textbf{then} $\mathit{ts} := (t)$ with $t$ interpolated in the ordinate parameter on the piece $y^g_{j-1} < \mathit{ev} < y^g_j$ \textbf{end if} \Comment{On a vertical run of $g$ this returns its jump abscissa}
    \State \textbf{for} $t \in \mathit{ts}$ \textbf{do} $\textit{emit}(t, \mathit{ys}_1)$ \textbf{end for} \Comment{Bottom edge: the plateau of $g$ at value $\mathit{ev}$ carries $f(\mathit{ev})$}
    \State \textbf{for} $y \in (\mathit{ys}_2, \mathit{ys}_3, \dots)$ \textbf{do} $\textit{emit}(\mathit{ts}_{\mathrm{last}}, y)$ \textbf{end for} \Comment{Right edge: the jump of $f$ at $\mathit{ev}$ climbs where $g$ leaves $\mathit{ev}$}
\EndWhile
\State \Return $\beta_h$
\Statex \textit{emit}$(x, y)$: append $(x, y)$ to $\beta_h$ unless it equals the last point of $\beta_h$ exactly.
\end{algorithmic}
\end{algorithm}

\noindent\textbf{Reading the listing:} The merge sweeps the value axis $[\mathit{lo}, \mathit{hi}]$ of the inner function once, stopping in increasing order at every \emph{event}, an abscissa of $\beta_f$ or an ordinate of $\beta_g$. At an event $\mathit{ev}$ it gathers the ordinates $\mathit{ys}$ that $f$ carries at $\mathit{ev}$ (one, or the whole connector when $f$ jumps at $\mathit{ev}$) and the abscissae $\mathit{ts}$ at which $g$ takes the value $\mathit{ev}$ (one, or the whole plateau when $g$ waits at $\mathit{ev}$). Whichever side has no breakpoint there is interpolated. It then emits the bottom edge of the rectangle $\mathit{ts} \times \mathit{ys}$ followed by its right edge. This \emph{double-tie} rule, one plateau of $g$ meeting one jump of $f$ at the same value, is the only place where the discontinuous extension changes the continuous sweep of \cite{visserEfficientMoveEvaluations2020, blauthVehicleRoutingTimedependent2024}. It is forced by left continuity: on the plateau $g \equiv \mathit{ev}$, so $h \equiv f(\mathit{ev})$, and $h$ jumps to $f(\mathit{ev}^+)$ where $g$ leaves $\mathit{ev}$, at the plateau's right end. The exact-duplicate drop of \textit{emit} is load-bearing, since two events can produce the same point. Appendix \ref{app:taxonomy-exactness} shows what a naive pairwise two-pointer merge does on such a tie, and states the floating-point clamp the shipped implementation adds on top of the listing.

\noindent\textbf{Complexity:} Every event consumes at least one input point and emits at most as many, which is the bound and the cost of theorem \ref{theorem:merge}. Redundant points and a trailing vertical can be emitted, and both are invisible to evaluation and to the duration minimum. Exact boundary ties, where the last value of one operand coincides with a breakpoint of the other, need no special step: the band $[\mathit{lo}, \mathit{hi}]$ is closed on both sides and the loop guard admits every event up to and including $\mathit{hi}$.

\noindent\textbf{Exactness, in brief:} The shipped merge performs no normalization. Removing redundant points is a separate, optional operation, and only part of it is bit-neutral. Dropping interior points of exactly horizontal or exactly vertical runs provably changes no value, whereas dropping a point on a \emph{sloped} piece does, even under an exact collinearity predicate, because floating-point interpolation is not transitive. On ATFs with genuine vertical steps \cite{rifkiImpactSpatiotemporalGranularity2020}, we measured the resulting pointwise deviations reaching full step heights. \kayros{} therefore keeps the \emph{un-normalized} composition as its canonical cost semantics, the checker and the solver staying bit-identical by construction, while the restricted flat/vertical deduplication remains available as a provably neutral optimization. Our validation campaign around this operation uncovered three independent families of defects: silent index-out-of-bounds errors in the reference open-source composition of \cite{lera-romeroLinearEdgeCosts2020}, the exact floating-point boundary ties that optimal \textit{Duration} solutions systematically produce and that random testing cannot reach, and defects of benchmark data provenance. For that last family, section \ref{subsec:benchmarks} gives the count of published BKS of \cite{lera-romeroLinearEdgeCosts2020} that the full-precision reading of \textit{Dabia2013} flips to infeasible. These findings are the origin of the two rules that sections \ref{subsec:kayros-arch} and \ref{subsec:benchmarks} enforce, the checker defining the objective and instances shipping as byte-exact canonical data. Appendix \ref{app:taxonomy} gives the full analysis and the evidence behind each family.

\section{The KAYROS Solver}\label{sec:solver}

This section presents \kayros{} as evaluated in this paper, release \texttt{1.6.0}, an open-source Python package over a C++ core. It covers the architecture and the principle governing every reported value (section \ref{subsec:kayros-arch}), the route-evaluation layer the heuristic stack rests on (section \ref{subsec:kayros-folding}), the anytime search (section \ref{subsec:kayros-anytime}), its fleet-aware extension (section \ref{subsec:kayros-fleet}), and the exact component with the precise content of an optimality certificate (section \ref{subsec:kayros-exact}).

\subsection{Architecture and the Checker as Referee}\label{subsec:kayros-arch}

\kayros{} runs two solving modes on one time-dependent engine. The \emph{anytime} mode returns a feasible solution within the first second and improves it until the deadline. The \emph{exact} mode searches for an optimum and reports whether it met the conditions required for a computational certificate. Both share one absolute deadline, stream incumbents through the same callback, and return the latest valid solution when interrupted. The core has four conceptual layers. (i) The NDLCF engine implementing the algebra of section \ref{sec:math} (breakpoint chains, the event merge of algorithm \ref{algo:compose_NDLCF} in its canonical un-normalized form, and the one-pass minimum of lemma \ref{lemma:minimum} yielding $\Delta_{\mathbf{r}}^{*}$ and $t_{\mathbf{r}}^{*}$). (ii) A route-evaluation layer folding arcs and vertices into route ready-time functions $\delta_{\mathbf{r}}$. (iii) A local-search layer holding the route trees, move evaluators and fleet machinery of sections \ref{subsec:kayros-folding} to \ref{subsec:kayros-fleet}. (iv) And the two search drivers. The vendored branch-price-and-cut component of section \ref{subsec:kayros-exact} sits alongside the drivers, in a contained subtree keeping its upstream style. One run uses one thread, parallelism belonging to the experiment layer above (GNU \texttt{parallel} in the experiment runner used with Grid'5000).

The governing principle is that \emph{the checker is the referee}. The canonical checker is not part of \kayros{}: it is the pure-Python reference implementation of the benchmark platform, which defines the objectives of section \ref{subsec:dm-def} constructively. It composes $\alpha$ and $\theta$ left to right into $\delta_{\mathbf{r}}$, reads $\Delta_{\mathbf{r}}^{*}$ and the earliest minimizer $t_{\mathbf{r}}^{*}$ off the breakpoints (lemma \ref{lemma:minimum}), and sums the per-route durations \emph{in canonical route order} (routes sorted by first customer). Floating-point addition is order-sensitive, and only a fixed order makes the total reproducible. Its arithmetic is plain IEEE-754 double precision with exact comparisons and no epsilon thresholds anywhere, which is what allows an independent reimplementation to produce bit-identical results. As such, the \kayros{} engine is a bit-identical port of that arithmetic, compiled with fused multiply-add contraction disabled so the port survives the optimizer. A standing equivalence suite requires the compiled core's $\delta_{\mathbf{r}}$, $\Delta_{\mathbf{r}}^{*}$ and $t_{\mathbf{r}}^{*}$ to be identical to the checker's on random routes over real instances. Any divergence is a solver defect, never an epsilon to add.

Two consequences follow. The value \kayros{} reports is the checker's recomputation, not the core's: before an incumbent is published it is handed to the checker, which recomputes the objective and rejects any mismatch bitwise. And the epsilon-free claim is a claim about the \emph{arithmetic}: the search layer does carry one screening threshold of $10^{-9}$ on tree-ranked move deltas, so that noise of the order of a unit in the last place cannot trigger a commit. No accounted value, however, is ever compared with a tolerance. Section \ref{subsec:kayros-folding} makes the resulting discipline concrete as the rule \emph{trees rank, the fold accounts}. The exact component obeys the same rule, every column entering its master problem being repriced by that fold.

\kayros{} performs no family-specific parsing: the benchmark library normalizes every family and accepts three travel-time models, all reaching the solver as explicit arc ATFs. The general model is a sidecar carrying, per arc, the raw breakpoints of $\alpha_{ij}$, validated on load to span $\mathcal{T}_{\mathrm{arc}}$ exactly, to be non-decreasing in both coordinates, and to satisfy $\alpha_{ij}(t) \geq t$. This non-IGP path admits arbitrary piecewise-linear data, in particular the stepwise families of section \ref{subsec:exactness-background}, whose arc travel times genuinely jump \cite{rifkiImpactSpatiotemporalGranularity2020}. Such value jumps appear as vertical runs of the chain, evaluated by the selected-function rule of definition \ref{def:selected}. IGP profiles and road graphs are compact encodings whose ATFs are materialized deterministically at load time and behave identically downstream, each carrying a digest of the canonical serialization of the materialized set. Verifying that digest is an explicit argument and not a default, since it costs roughly 80 seconds and 10 GB of peak memory on a one-million-arc instance. That verification is therefore never enabled inside a campaign run, and the file digests that pin which data a run consumed are taken before its clock starts. Loading itself is charged to the run under the protocol of section \ref{subsec:benchmarks}, where each solver receives $\tlim$ less the start-up and loading time already spent.

\begin{figure}[htbp]
\centering
\includegraphics[width=\linewidth]{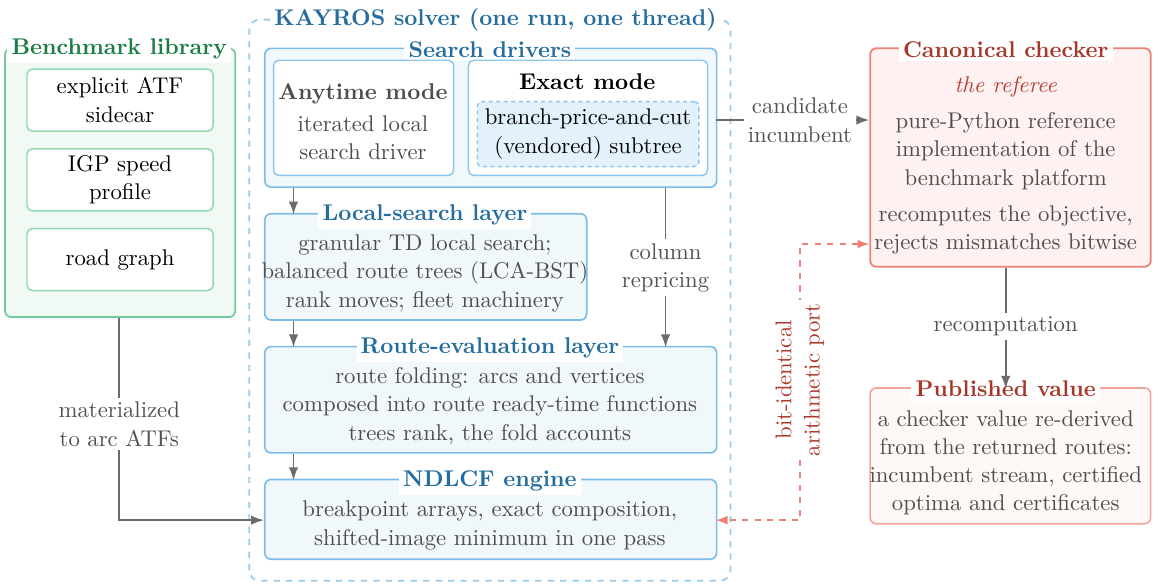}
\caption{Architecture of KAYROS. Two solving modes, anytime and exact, share one NDLCF engine that ports the reference checker's arithmetic bit-identically, with no approximation between instance data and reported costs. The local-search layer sits between the anytime driver and the engine, holding the balanced route trees that rank candidate moves. The canonical checker itself stays \emph{outside} the solver as the referee: every value KAYROS reports, incumbent or certified optimum, is a checker value re-derived from the returned routes.}
\label{fig:kayros-architecture}
\end{figure}

\subsection{Exact Route Folding and Balanced Route Trees}\label{subsec:kayros-folding}

We call \emph{route folding} the evaluation of a route by composing its arc and vertex functions in sequence, the operation that theorem \ref{theorem:comp_of_rrtf} organizes into a balanced tree. Write $\mathbf{r} = \langle v_1 = o, \dots, v_m = d \rangle \in \Omega$ as a chain of $m - 1$ \emph{leaves} with $\delta_{\mathbf{r}} = \mathit{lf}_{m-1} \circ \dots \circ \mathit{lf}_1$, where $\mathit{lf}_1 = \theta_{v_2} \circ \alpha_{v_1 v_2} \circ \mathrm{Id}_{[e_o, l_o] \cap \mathcal{T}_{\mathrm{arc}}}$ absorbs the feasible depot departure window, $\mathit{lf}_k = \theta_{v_{k+1}} \circ \alpha_{v_k v_{k+1}}$ for $1 < k < m-1$, and $\mathit{lf}_{m-1} = \mathrm{Id}_{[0, l_d]} \circ \alpha_{v_{m-1} v_m}$ closes on the depot due date with no waiting clamp, the route ending upon arrival. Two conventions are deliberate, and both exist to keep the leaves aligned with the checker. The grouping is $\theta \circ \alpha$ rather than the $\alpha \circ \theta$ of \citeauthor{blauthVehicleRoutingTimedependent2024} \cite{blauthVehicleRoutingTimedependent2024}, equal by proposition \ref{prop:temporal} in exact arithmetic but making every left-fold prefix literally one of the checker's accumulator values. And the depots receive no waiting clamp, unlike in \citeauthor{visserEfficientMoveEvaluations2020} \cite{visserEfficientMoveEvaluations2020}: $\theta_o$ and $\theta_d$ are the depot identities of section \ref{subsec:taxonomy}, the windows acting through domain restriction instead, which leaves $\Delta_{\mathbf{r}}^{*}$ and $t_{\mathbf{r}}^{*}$ unchanged but changes the function. Feasibility is then structural rather than flagged, since clamping lives inside $\theta$ only: waiting is a plateau, a deadline is a domain cap, and infeasibility is the empty function propagating through every enclosing composition. A query returning the empty function \emph{is} the infeasibility answer, a reading that proposition \ref{prop:taxonomy} licenses: every leaf is a temporal map, hypothesis (H) holds at every composition the tree performs, and theorem \ref{theorem:merge} makes the empty chain and the empty composition the same thing.

Over these leaves \kayros{} builds the balanced binary search tree (BST) of \citeauthor{blauthVehicleRoutingTimedependent2024} (their Theorem 12), which we call the LCA-BST, after the lowest common ancestor (LCA) its queries walk to. Boundaries delimit the leaves, $\mathit{lf}_k$ lying between boundaries $k$ and $k+1$. For $b_1 < b_2$ we write $\mathit{fold}(b_1, b_2) = \mathit{lf}_{b_2 - 1} \circ \dots \circ \mathit{lf}_{b_1}$, so that $\mathit{fold}(b_1, b_3) = \mathit{fold}(b_2, b_3) \circ \mathit{fold}(b_1, b_2)$ for any intermediate $b_2$, an identity from associativity of partial-function composition. Proposition \ref{prop:temporal} ensures that the event merge remains within its valid temporal subclass for these leaves. The tree is static and balanced over the boundary keys, and every node $b$ stores $\mathit{fold}(b', b)$ for each descendant $b' < b$ and $\mathit{fold}(b, b')$ for each descendant $b' > b$, one composition each. Hence the tree costs $\mathcal{O}(m \log m)$ compositions to build and stores as many functions per route. The payoff is at query time: after an $\mathcal{O}(\log m)$ walk, the composition of any contiguous leaf range follows from the lowest common ancestor of its two boundaries in \emph{at most one} composition, and in none when one boundary is an ancestor of the other. That is the structural gain over the composition tree $\bbt^{\mathbf{r}}$ of theorem \ref{theorem:comp_of_rrtf}, which answers the same query by folding $\mathcal{O}(\log m)$ stored nodes. A localized update recomputing only the $\mathcal{O}(m)$ stored functions straddling one changed leaf is implemented and gate-tested bitwise against a full rebuild. We report the shipped search honestly: its commit path rebuilds both changed routes from scratch, leaves, tree and repricing, and never calls that update.

Every candidate move is an instance of one splice primitive, replacing a window of the receiving route by a window donated from another route, one of the two windows being empty for a deletion or an insertion. Such a move is priced with a constant number of tree queries plus a constant number of seam leaves (algorithm \ref{algo:splice} in appendix \ref{app:solver}). The tree, however, only \emph{ranks}. Partial-function composition is associative in exact real arithmetic. Proposition \ref{prop:temporal} establishes that the temporal subclass is closed under its anchor hypothesis. Each computed composition interpolates, and floating-point interpolation is not associative. Hence different bracketings of one chain can differ by units in the last place and, on a stepwise instance, by a full step height. Appendix \ref{app:solver} reports the measurement on 4\,800 routes. \kayros{} therefore enforces the rule \emph{trees rank, the fold accounts}: an accepted move is committed only after both changed routes have been rebuilt and repriced by the sequential checker-identical fold. The acceptance test is a strict inequality on repriced totals with no tolerance, and the number of tree-ranked candidates the fold then rejects is instrumented rather than assumed to be zero. Measured against naive recomposition, the structures speed move evaluation up by a factor 2 at a mean route length of about 6 and by up to a factor 17 at about 50. We found no size at which the sequential fold overtakes them.

\subsection{The Anytime Stack}\label{subsec:kayros-anytime}

The anytime mode chains a construction heuristic, a granular time-dependent local search and an iterated local search driver, all pricing through the fold of section \ref{subsec:kayros-folding} and publishing through one strictly decreasing incumbent stream.

\paragraph{Construction.} A deterministic greedy heuristic builds routes one at a time, departing at the earliest feasible depot time and repeatedly appending the free customer with the earliest direct ready time, ties broken by the smallest index. A route is closed when the chosen customer would violate capacity or the depot return. It costs $\mathcal{O}(n^2)$ and runs in 0.03 s at $n = 1000$. Its output is published as the first incumbent before any descent, so the anytime stream opens in well under a second at every size. Under \textit{Duration} the seed is then diversified before the first descent: its routes are split until the count reaches a factor 2 of the constructed one by default. Indeed, the search sheds routes freely and effectively never adds one, so the seed's route count decides the final one. The split seed is republished only when it strictly improves on the raw one. It usually does not: splitting costs duration up front and pays off only through the search that follows. The device is disarmed under \textit{FleetCostDuration}, where every extra route is priced and splitting is the wrong direction.

\paragraph{Granular candidate lists.} Exhaustive enumeration dominates the runtime at scale, so all scans are restricted to granular candidate lists, in the tradition of granular search \cite{tothGranularTabuSearch2003} and of the proximity-based variant of \citeauthor{vidalHybridGeneticAlgorithm2013} \cite{vidalHybridGeneticAlgorithm2013}. Since \textit{Duration} carries no separate distance term, the proximity is read off the ATFs: a pair is priced by the minimum travel time over the departures the two time windows and the arc domain jointly allow, plus a fraction of the waiting that even the latest feasible departure cannot avoid at the head. Pairs whose earliest feasible departure already arrives past the head's deadline are screened out. The relation is symmetrized, each customer keeps its 50 nearest others by default, and the adjacency is union-symmetrized so membership tests are sound in both directions. Appendix \ref{app:solver-anytime} gives the exact formula and the constants. On 1000-customer instances, granular enumeration accelerates a full descent by about a factor 13 for a quality gap of 0.4 to 1.0\%. An exhaustive sentinel restores complete enumeration.

\paragraph{The descent.} The local search is a first-improvement variable-neighborhood descent over four operators, cycled as inter-route relocate, intra-route relocate, swap and 2-opt*, any commit restarting the cycle. Each inter-route operator carries its own granular justification. Cheap screens precede any function algebra, notably prefix-demand capacity screens and a cached deletion ranking that gives a relocation's donor gain before its insertion is priced (appendix \ref{app:solver-anytime}). Termination is guaranteed because every commit strictly decreases the repriced total. A deadline is checked between operator passes, so the overshoot is bounded by one pass and the exit state is always repriced. Classic intra-route 2-opt is deliberately absent: reversing a segment replaces every $\alpha_{ij}$ by $\alpha_{ji}$, a different function, so no stored composition can be reused.

\paragraph{The driver.} Algorithm \ref{algo:ils} states the iterated local search loop, whose shape and defaults follow the modern single-trajectory references \cite{PyVRP, hashimotoIteratedLocalSearch2007} adapted to feasible-only, checker-consistent search. The kick is a ruin-and-recreate perturbation \cite{schrimpfRecordBreakingOptimization2000}: seed customers are visited in random order and each drags its granular neighbors out. The removed customers are then reinserted in random order at their best tree-ranked feasible position, each insertion fold-repriced and falling through to the next-best position on a disagreement. Because the search is feasible-only, the iteration owns its repair: a customer with no feasible position may open a singleton route when $K$ allows, and otherwise the whole kick is undone and redrawn. Acceptance is Late-Acceptance Hill Climbing (LAHC) \cite{burkeLateAcceptanceHillClimbing2017} with both enhancements of the authors' section 4.2, so worse acceptances stay internal and the published stream remains strictly decreasing. Two elements matter at scale. An exhaustive-list polishing descent runs on every new global best, its result being adopted only where it strictly improves. And the search restarts from the incumbent on stagnation, measured either in iterations or, more usefully, in \emph{work units}. One work unit is one candidate pricing entered after the cheap screens. Work units are nearly rate-invariant across sizes, so a work-based window fires at the same effective stall on every size whereas a flat iteration count fits one size class only. The whole search is deterministic at a fixed seed. Appendix \ref{app:solver-anytime} gives the removal range, the history length, the further polish points and the measured rate invariance. An alternative MAX-MIN ant colony strategy, reviving the dormant time-dependent ACO stream \cite{donatiTimeDependentVehicle2008, balseiroAntColonyAlgorithm2011}, remains available and is compared against the default in section \ref{subsec:components}.

\begin{algorithm}
\setstretch{1}\small
\caption{The KAYROS time-dependent iterated local search}\label{algo:ils}
\begin{algorithmic}[1]
\Require An instance, an absolute deadline, a seed; $z(\cdot)$ the objective of section \ref{subsec:dm-def}, always computed by the checker-consistent fold; the exhaustive polish shown below is an option, on by default
\Ensure The best solution $\mathcal{S}^{*}$ and a strictly decreasing stream of published incumbents. \textbf{publish} is a guarded operation: it emits its argument only when the value strictly improves on the last published one, and, when the route-count cap of section \ref{subsec:kayros-fleet} is armed, only at or below that cap
\State $\mathcal{S} := \textsc{GreedyConstruct}()$; \textbf{publish} $\mathcal{S}$ \Comment{the raw seed, before any descent}
\If{the objective does not price routes} \Comment{\textit{Duration} only}
    \State $\mathcal{S}' := \textsc{SplitRoutes}(\mathcal{S}, 2 \left| \mathcal{S} \right|)$ \Comment{seed split factor 2 by default}
    \State \textbf{if} $\mathcal{S}'$ is feasible \textbf{then} $\mathcal{S} := \mathcal{S}'$; \textbf{publish} $\mathcal{S}$ \textbf{end if} \Comment{the guard usually refuses}
\EndIf
\State $\mathcal{S} := \textsc{Descend}(\mathcal{S}, \textit{granular})$; $\mathcal{S}^{*} := \mathcal{S}$; \textbf{publish} $\mathcal{S}^{*}$
\State $\mathcal{S} := \textsc{Descend}(\mathcal{S}, \textit{exhaustive})$ \Comment{optional polish, on by default}
\State \textbf{if} $z(\mathcal{S}) < z(\mathcal{S}^{*})$ \textbf{then} $\mathcal{S}^{*} := \mathcal{S}$; \textbf{publish} $\mathcal{S}^{*}$ \textbf{end if}
\State Initialize the late-acceptance history $\mathit{late}[1..\mathit{HIST}]$ at $z(\mathcal{S})$, slot index $k := 1$
\While{the deadline is not reached}
    \If{the stagnation window has elapsed (iterations or work units)}
        \State $\mathcal{S} := \mathcal{S}^{*}$ \Comment{restart to best}
        \State \textbf{if} the objective prices routes \textbf{then} \Call{FleetDescent}{} \textbf{end if}
        \State reset $\mathit{late}$ and the stagnation counters
    \EndIf
    \State \textbf{if} the objective prices routes \textbf{and} the fleet-descent work period has elapsed \textbf{then} \Call{FleetDescent}{} \textbf{end if}
    \State $\mathcal{S}' := \textsc{RuinAndRecreate}(\mathcal{S})$ \Comment{undone and redrawn if repair fails}
    \State $\mathcal{S}' := \textsc{Descend}(\mathcal{S}', \textit{granular})$ \Comment{restricted to customers whose context changed}
    \State \textbf{if} $z(\mathcal{S}') < z(\mathcal{S}^{*})$ \textbf{then} $\mathcal{S}' := \textsc{Descend}(\mathcal{S}', \textit{exhaustive})$; $\mathcal{S}^{*} := \mathcal{S}'$; \textbf{publish} $\mathcal{S}^{*}$ \textbf{end if}
    \State \textbf{if} $z(\mathcal{S}') < \mathit{late}[k]$ \textbf{or} $z(\mathcal{S}') < z(\mathcal{S})$ \textbf{then} $\mathcal{S} := \mathcal{S}'$ \textbf{end if} \Comment{first Burke-Bykov enhancement}
    \State \textbf{if} $z(\mathcal{S}) < \mathit{late}[k]$ \textbf{then} $\mathit{late}[k] := z(\mathcal{S})$ \textbf{end if} \Comment{second enhancement}
    \State $k := (k \bmod \mathit{HIST}) + 1$
\EndWhile
\State \Return $\mathcal{S}^{*}$
\Statex
\Procedure{FleetDescent}{} \Comment{section \ref{subsec:kayros-fleet}, run on the working solution $\mathcal{S}$}
    \State erase one route by the ejection ladder; \Return with $\mathcal{S}$ unchanged on failure
    \State $\mathcal{S} := \textsc{Descend}(\mathcal{S}, \textit{granular})$; $\mathcal{S} := \textsc{Descend}(\mathcal{S}, \textit{exhaustive})$ \Comment{the $K - 1$ basin}
    \State \textbf{if} $z(\mathcal{S}) < z(\mathcal{S}^{*})$ \textbf{then} $\mathcal{S}^{*} := \mathcal{S}$; \textbf{publish} $\mathcal{S}^{*}$ \textbf{end if} \Comment{the phase publishes its own improvements}
\EndProcedure
\end{algorithmic}
\end{algorithm}

\subsection{Fleet-Aware Search under FleetCostDuration}\label{subsec:kayros-fleet}

Under \textit{FleetCostDuration} (section \ref{subsec:dm-def}) the objective is $c_{\mathrm{fleet}} \left| \mathcal{S} \right| + \sum_{\mathbf{r} \in \mathcal{S}} \Delta_{\mathbf{r}}^{*}$. On the \textit{Blauth2024} family of section \ref{subsec:bks}, $c_{\mathrm{fleet}}$ is worth roughly 2.5 route durations, so dissolving one route can dominate a substantial duration increase. The fixed cost enters the search at exactly two points. The acceptance test prices each side with $c_{\mathrm{fleet}}$ per non-empty route, an exact no-op when $c_{\mathrm{fleet}} = 0$. A relocation moving a singleton donor's only customer is credited $c_{\mathrm{fleet}}$, without which every duration-increasing merge would be screened out and removals worth up to $c_{\mathrm{fleet}}$ never tried. This leaves a structural asymmetry: no local-search move opens a route, so the descent can only lower $\left| \mathcal{S} \right|$, and the kick is the only mechanism that can raise it. The run starts from whatever the greedy construction, or a warm start, produces, and the fleet bound $K$ acts as an upper bound only. The seed-splitting device of section \ref{subsec:kayros-anytime} stays disarmed. A first cheap mechanism dissolves the smallest route whole inside the kick. Instrumentation exposed its limit, which motivated the dedicated phase below (appendix \ref{app:solver-fleet}).

\emph{Fleet descent} is an ejection-ladder route elimination in the lineage of route-minimization heuristics for the VRPTW \cite{nagataPowerfulRouteMinimization2009}, run on the incumbent at every restart-to-best trigger and, by default, on a work-based period. One attempt erases a victim route and drains its customers through a two-rung ladder, feasible best insertion first and insertion with ejection second, with no singleton fallback: reopening a route is exactly the failure mode the phase exists to remove. The phase is all-or-nothing: any dead end, exhausted budget or deadline restores the pre-attempt solution exactly. Because no rung opens a route, success means exactly one route fewer with every customer served. All budgets are counted in work units rather than wall-clock seconds. Appendix \ref{app:solver-fleet} details the ladder.

The reading is mixed: the phase is decisive at $n = 500$ and hits a hard ceiling at $n \geq 1000$. Section \ref{subsubsec:fleet-eval} gives the figures and appendix \ref{app:solver-fleet-eval} the full reading.

The evaluated release adds a bounded \emph{penalty-tolerant squeeze}, opt-in and inert by default. It relaxes the time-window channel only, accumulating lateness in a separate non-negative warp channel. It then runs a penalized descent on a copy of the state, either as a post-drain polish or as an in-ladder rescue. Only a state returned to exactly zero warp with a strictly better objective is ever banked and returned. Hence warp never crosses the phase boundary, and a phase that banks nothing is a strict no-op.\footnote{From the machine instruction NOP, \emph{no operation}, which does nothing by design: a phase that banks nothing leaves the search state exactly as it was.} The inert default is a measured decision: the two readings of section \ref{subsubsec:fleet-eval} point opposite ways at $n = 1000$ and at $n = 500$, so there is no size at which arming is unconditionally safe. Campaigns arm it explicitly, and default streams stay bitwise pre-squeeze. A companion option confines the search to a route-count cap, the publication guard of algorithm \ref{algo:ils} refusing every above-cap state. A run therefore publishes nothing at all if the cap is never attained, which makes a negative fleet result reportable rather than silently downgraded. Section \ref{subsec:bks} uses this option. The exact component does not support \textit{FleetCostDuration}, so nothing in this section carries an optimality claim. Appendix \ref{app:solver-fleet} states the mechanics of the squeeze and of the route-count cap.

\subsection{The Exact Component and Optimality Certificates}\label{subsec:kayros-exact}

Rather than reimplementing a decade of exact time-dependent machinery, \kayros{} vendors the branch-price-and-cut solver of \citeauthor{lera-romeroLinearEdgeCosts2020} \cite{lera-romeroLinearEdgeCosts2020} into a contained subtree under its MIT license, with an attribution notice: the set-partitioning master of section \ref{subsec:spf}, the labeling algorithm for the pricing problem with its heuristic and exact levels, the dominance rules and the cutting planes. Heuristics and standalone entry points were not vendored.

Five additions produce the component this paper evaluates. Three concern infrastructure. The first is an open LP backend, HiGHS \cite{huangfuParallelizingDualRevised2018} built statically into the distributed wheels.\footnote{A \emph{wheel} is the standard built distribution format of Python packages, the \texttt{.whl} archive installers download from the PyPI package repository \cite{pythonsoftwarefoundationPyPI2026}. Building HiGHS statically into it means a plain \texttt{pip install kayros} ships the exact mode ready to run, with no compiler and no separately installed LP solver.} The second is anytime compliance, one absolute deadline per run from which every component takes its residual budget. The third is warm starts through columns, caller-supplied routes being repriced and added as initial columns. Two concern exactness. One is checker-consistent binary64 column costs: every route entering the master is repriced by the fold of section \ref{subsec:kayros-folding}, so its objective coefficient follows the canonical checker. This agreement does not make LP bounds mathematically exact. They retain the backend's stated tolerances. The other is the exact value-jump path used for production pricing on stepwise instances, where verticals travel through the piecewise-linear machinery as tagged first-class objects. The reason is that the arithmetic that is safe for a genuine travel-time jump is wrong for the inverse of a flat interval, and selecting between the two by inspecting intermediate operands had produced a class of over-certification. Two verdicts complete the picture. A run reaching its deadline without completing the proof conditions returns open, with a solver lower bound valid under the same standard LP and pricing tolerances the certificates carry rather than as a rigorous bound. A run stopped by a resident-set self-guard returns a resource limit. Appendix \ref{app:solver-exact} details each addition.

What a certificate establishes is one sentence, which the stamp itself carries: the solution is \emph{optimal under checker-exact route costs and standard LP/pricing tolerances, completeness modulo Lera epsilon dominance}. These are \emph{conditional computational results}, not formal proof objects, and the three conditions are distinct. Here \enquote{checker-exact} names agreement with the checker's deterministic binary64 arithmetic, not exact real arithmetic. The independent checker verifies returned-route feasibility and cost. Route costs agree with its deterministic binary64 convention by construction. The dual bounds carry the LP solver's standard tolerances rather than interval-arithmetic safe bounds, and turning them into rigorous bounds is future work. Pricing completeness is modulo the vendored labeling's epsilon comparisons, a residue we deliberately did not remove. A certificate is therefore not a portable artifact checkable without rerunning the solver, and it is not a claim of exact real arithmetic. It follows the practical tradition of the exact literature.

Accordingly, no certificate is issued from a single run. The publication gate is a four-run protocol, crossing cold and warm starts with the two bidirectional labeling modes.\footnote{The pricing labeling extends labels forward from $o$ and backward from $d$ on the time-reversed instance, and completed routes come from merging a forward and a backward label at a boundary time. The two modes place that boundary at the horizon end (asymmetric: the forward direction spans the whole horizon and the merge is minimal) or at its midpoint (symmetric: both directions grow to half the horizon, halving label growth at the price of a genuine merge). The two traverse different code paths over different label populations, which is what makes their agreement an independent soundness signal.} All four runs must reach proven optimality at one agreeing checker-evaluated binary64 value not above the stored reference, each with at least one exact-pricing iteration in an audited pricing census and with routes the checker actually repriced. Refusal is a first-class outcome alongside open and resource-limited, and appendix \ref{app:solver-exact} states the rules in full. The limit of this instrument must be stated as plainly as its strength. Four runs can detect a defect that depends on initialization or configuration, and the campaign record contains such detections. Four copies of one faulty reasoning path, however, provide no protection against a shared deterministic error, and the external checker cannot verify that the solver searched every improving column. A better checker-valid solution is therefore a decisive refutation, whereas the absence of one is not a proof. This is why the certificates of section \ref{subsec:bks} are published as falsifiable, retractable artifacts, and why certificates have in fact been withdrawn, two of them under this very protocol (section \ref{subsec:bks}).

\section{Experimental Evaluation}\label{sec:results}

\subsection{Benchmarks and Protocol}\label{subsec:benchmarks}

\paragraph{Five families} The campaign draws on five instance families, spanning the classic literature benchmark, real road networks and controlled variation at a scale the classic benchmark does not reach. \textit{Dabia2013} \cite{dabiaBranchPriceTimeDependent2013} is the TD-Solomon benchmark described below, the reference point of the exact literature. \textit{Rifki2020} \cite{rifkiImpactSpatiotemporalGranularity2020} is built on the Lyon road network and its travel-time functions carry discontinuities, which stress exactly the arithmetic of section \ref{subsec:composition}. \textit{Vu2020} contributes continuous-travel-time instances at $n = 59$ and $n = 99$ over two customer profiles and four congestion depths. \textit{Lera2026} extends paired time-dependent coverage of the Solomon structure to $n = 1000$ over three generation series. \textit{Poryos2026} is a road-network family built over five cities, pairing static and time-dependent variants with reproducible synthetic traffic, and reaching $n = 1000$. Both \textit{Lera2026} and \textit{Poryos2026} were constructed during this thesis: the former extends established Gehring--Homberger and IGP ingredients, while the latter is the original road-network benchmark contribution. Section \ref{sec:discussion} states their combined share under the panel-equal weighting.

The classic TDVRPTW benchmark is the Solomon VRPTW set \cite{solomonAlgorithmsVehicleRouting1987} under the TD-IGP profiles \cite{ichouaVehicleDispatchingTimedependent2003}, used by \cite{dabiaBranchPriceTimeDependent2013, lera-romeroEnhancedBranchPrice2018, lera-romeroLinearEdgeCosts2020, panHybridAlgorithmTimedependent2021, blauthVehicleRoutingTimedependent2024}. Derived sets extend it further, its instances to the multi-trip variant \cite{panMultitripTimedependentVehicle2021} and its speed-profile recipe to the electric one \cite{lera-romeroBranchCutandPriceTimeDependent2024}. What circulates under that name is not one dataset, since several slightly different variants exist \cite{blauthVehicleRoutingTimedependent2024}. We take the instances of the seminal work of \citeauthor{dabiaBranchPriceTimeDependent2013} \cite{dabiaBranchPriceTimeDependent2013}, the variant reused since by \citeauthor{lera-romeroLinearEdgeCosts2020} and \citeauthor{panHybridAlgorithmTimedependent2021} \cite{lera-romeroLinearEdgeCosts2020, panHybridAlgorithmTimedependent2021}, and refer to this dataset as \textit{Dabia2013}, vendored in MAMUT-routing \cite{rascoussierMAMUTrouting2026}.

\paragraph{Canonical preprocessing} \textit{Dabia2013} is not fully specified by ``Solomon instances plus the speed profiles of \citeauthor{ichouaVehicleDispatchingTimedependent2003}'': the actual instances are the output of a preprocessing pipeline that was never published as data. The pipeline is hard-coded in the original Java solver of \citeauthor{dabiaBranchPriceTimeDependent2013}, which the authors kindly provided to us. It scales every Solomon quantity by 10, \emph{floor-truncates} every arc distance to an integer, and integrates a stepwise speed profile along that distance, which yields the FIFO-compliant TTFs of section \ref{sec:math}. Appendix \ref{app:protocol-dabia} states the pipeline step by step, and the resulting instances are identical to those distributed by \citeauthor{lera-romeroLinearEdgeCosts2020} \cite{lera-romeroLinearEdgeCosts2020}. The floor truncation deserves emphasis: it is invisible in the papers, yet it is part of the very definition of the benchmark, just as the DIMACS 2021 and SINTEF conventions are for the classic VRPTW \cite{DIMACS2021, Sintef2008}. The truncation also makes the canonical instances an \emph{outer} approximation of their full-precision reading \cite{rascoussierImpactScalingRounding2026}, so that feasibility transfers one way only and 11 of the 146 published BKS of \cite{lera-romeroLinearEdgeCosts2020} lose it under that reading (appendix \ref{app:bks}). This is why the curated \textit{Dabia2013} family shipped with \kayros{} \cite{rascoussierMAMUTrouting2026} distributes byte-exact canonical data with checksums, never re-derivation recipes.

\paragraph{The frozen instance set} The comparison runs on 212 instances fixed before any campaign record existed, organized as 10 \emph{panels}, a panel being one (family, size) pair, two per family. Two panels are taken whole, \textit{Dabia2013} at $n = 100$ (all 56 instances) and \textit{Poryos2026} at $n = 1000$ (all 60). These are the two canonical panels the anytime figures below are drawn for. The other 8 are stratified samples of 12 instances each. A deterministic, solver-blind generator draws them from the strata each family exposes, as a pure function of the enumerated pool and one seed, never of any solver result (appendix \ref{app:protocol} gives the strata and the draw rule). The frozen manifest carries a digest of its entry list (\texttt{24c18aed}) and the commit of the instance store it was drawn from, so the set is fully reconstructible. Table \ref{tab:panels} gives its composition.

\begin{table}[htbp]
\centering
\small
\caption{The 10 panels of the frozen 212-instance benchmark, two per family, with the scheduled seed budget and, in Runs, the number of cells over all 5 arms. The last column reads as a ratio: of the panel's reference values, how many are optimality certificates under the post-fold publication snapshot, the remainder being best-known solutions with no optimality claim.}
\label{tab:panels}
\begin{tabular}{llrrrrrr}
\toprule
Panel & Family & $n$ & Pool & Instances & Seeds & Runs & Cert. / BKS \\
\midrule
p01 & Dabia2013 & 50 & 56 & 12 & 10 & 600 & 7 / 12 \\
p02 & Dabia2013 & 100 & 56 & 56 & 10 & 2800 & 25 / 56 \\
p03 & Rifki2020 & 50 & 30 & 12 & 10 & 600 & 2 / 12 \\
p04 & Rifki2020 & 60 & 30 & 12 & 10 & 600 & 0 / 12 \\
p05 & Vu2020 & 59 & 56 & 12 & 10 & 600 & 11 / 12 \\
p06 & Vu2020 & 99 & 56 & 12 & 10 & 600 & 11 / 12 \\
p07 & Lera2026 & 200 & 96 & 12 & 10 & 600 & 0 / 12 \\
p08 & Lera2026 & 1000 & 96 & 12 & 5 & 300 & 0 / 12 \\
p09 & Poryos2026 & 100 & 20 & 12 & 10 & 600 & 0 / 12 \\
p10 & Poryos2026 & 1000 & 60 & 60 & 5 & 1500 & 0 / 60 \\
\midrule
\multicolumn{4}{l}{Total, 5 families} & 212 & & 8800 & 56 / 212 \\
\bottomrule
\end{tabular}
\end{table}

\paragraph{Arms, budget and grid} Five solver arms were run, with their pinned versions and model contracts in table \ref{tab:arms}. Four are \emph{Tier-1}, meaning they consume the rich per-arc travel-time functions of the instance without approximation: \kayros{} 1.6.0 at released defaults with an empty parameter set, Timefold Solver 2.3.0, Hexaly 15.0.20260724 through the vendor's C++ external-function binding (section \ref{subsubsec:threads} states why that binding and not another), and jsprit 2.0.0. The fifth is a \emph{Tier-2} arm, the same Hexaly build on a disclosed time-sliced approximation of those functions at 96 slices, likewise through its C++ binding. It is reported for information, separated from the Tier-1 results. Every arm receives the same budget, one hour of wall-clock time per cell in a single-threaded process pinned to one physical core, and no arm is given more. Instances at $n \leq 200$ run on 10 seeds and instances at $n \geq 1000$ on the first 5 of the same list, which yields $5 \times 1760 = 8\,800$ cells. All 8\,800 completed, ended by exhausting their time limit, so no arm is scored on a truncated stream. Cost is never the solver's own number: every published incumbent of every arm is re-evaluated by the same external checker (\texttt{mamut-routing-lib} 0.9.0), which reschedules each route's departures optimally and returns the canonical \textit{Duration}. It is on that stream that every statistic below is computed.

\begin{table}[htbp]
\centering
\small
\caption{The five solver arms, their pinned versions and their model contracts, all run single-threaded at one hour per cell and rescored by the same checker. The time-sliced Hexaly tier is descriptive and never enters the ranking or the confirmatory family.}
\label{tab:arms}
\begin{tabularx}{\textwidth}{llX}
\toprule
Arm & Version & Model contract \\
\midrule
\href{https://github.com/0nyr/kayros}{\kayros{}} & 1.6.0 & Tier-1, native. Released defaults with an empty parameter set, Duration objective, iterated local search; the penalty-tolerant squeeze is inert under Duration. \\
\href{https://github.com/TimefoldAI/timefold-solver}{Timefold} & 2.3.0 & Tier-1. Time-dependent travel evaluation implemented in model code on the open-source edition, checker-rescored like every other arm. \\
\href{https://www.hexaly.com/hexaly-optimizer}{Hexaly} (Tier-1) & 15.0.20260724 & Tier-1. Rich per-arc piecewise-linear arrival functions through a C++ external function, the binding that evaluates them at the full requested width. \\
\href{https://github.com/graphhopper/jsprit}{jsprit} & 2.0.0 & Tier-1. Time-dependent transport-time callbacks, routes kept at their fixed earliest departure by the solver, then rescheduled by the checker on the same sequences. \\
\midrule
\href{https://www.hexaly.com/templates/time-dependent-routing-problem-with-time-windows-tdcvrptw}{Hexaly} (Tier-2, sliced) & 15.0.20260724 & Tier-2, descriptive only. The same build on a disclosed time-sliced approximation at 96 slices, rescored on the rich functions, outside the ranking and the confirmatory family. \\
\midrule
\multicolumn{3}{l}{Checker pin: \href{https://github.com/ANR-MAMUT/MAMUT-routing-lib}{mamut-routing-lib} 0.9.0, source commit \texttt{3702a30}.} \\
\bottomrule
\end{tabularx}
\end{table}

\paragraph{Hardware and load} The campaign ran on a homogeneous cluster of the Grid'5000 testbed\footnote{The \texttt{grvingt} cluster of the Nancy site, as batch jobs on its abaca queue: each node carries two Intel Xeon Gold 6130 processors, 32 physical cores in total, and 192~GiB of memory. The thread-scaling study of section \ref{subsubsec:threads} ran on the same cluster.}, with one solver process per physical core and never more processes than cores. The final wave ran on 30 nodes and completed in about 22.3 hours of makespan, producing 41\,GB of streams, with no memory exhaustion at any point. The budget $\tlim$ is measured end to end by the harness, from process launch, and each solver is handed $\tlim$ less the start-up and loading time already spent, so every incumbent is timestamped on that one external clock rather than on the solver's own. The start-up cost that a Java arm and a Python-bound arm pay, and a native arm does not, was therefore measured as a \emph{load fraction} of $\tlim$. The largest panel-mean load fraction is 1.62\% (58.4\,s) for \kayros{} on \textit{Poryos2026} at $n = 1000$. Individual runs can take longer, reaching 3.552\% (127.9\,s) on that panel. All loading remains inside the common budget, and no arm is credited a warm-up allowance (appendix \ref{app:protocol-load}).

\paragraph{The reference and its zero-gap property} Anytime scoring needs a per-instance reference value, and ours is a single frozen snapshot of the public solution store, cut once and used unchanged for all five arms. It holds one entry per manifest instance, of which 56 hold an optimality certificate. It was cut \emph{after} the campaign's own improvements were folded back into the store: the final wave produced 52 solutions strictly better than the stored value, and all 52 were re-verified by the canonical checker and folded. The resulting snapshot improves on the pre-campaign freeze at 54 of the 212 entries. Because the reference is post-fold, the squeezed gap is non-negative everywhere by construction, and on an instance whose record an arm holds, that arm's best run ends at exactly zero gap. Section \ref{sec:discussion} states what that costs and what it buys. The reference class of every panel is recorded in table \ref{tab:panels}, because a zero gap against a certificate means proven optimality under the conditions of section \ref{subsec:kayros-exact}, whereas a zero gap against a BKS means only that nobody has done better yet. Note that references are keyed by the full identity (problem variant, family, size, instance path) and never by instance name since names repeat across sizes in three of the five families.

\subsection{Metrics and Confirmatory Design}\label{subsec:metrics}

Evaluation rests on two metrics, reported side by side throughout this section: a normalized anytime score $\ascore$, which charges the whole trajectory of a run, and the final gap $g_\tlim$, the classical end-of-budget distance to the reference. The paragraphs below define both. The confirmatory design is built on the first, and every final-gap aggregate is descriptive. The design of $\ascore$ itself is the subject of a companion paper \cite{rascoussierAnytimeSolverEvaluation2026}, which also reports the sensitivity analysis this section does not repeat: how stable the rankings it produces are across metric variants and across reference-policy regimes. The dispersion the seeds carry is measured on this campaign's own Hexaly ladders in the thread-scaling report \cite{rascoussierHexalyThreadScaling2026}. What follows recalls only what the present campaign uses.

\paragraph{A bounded gap} Let $z$ be a checker-evaluated solution cost and $z^* > 0$ the reference value of its instance. The \emph{squeezed gap} is
\[
\sgap(z, z^*) = \frac{z - z^*}{z + z^*} \in [-1, 1],
\]
strictly increasing in $z$. When the reference is clear from context we write $\sgap(z)$, so that $\sgap(0) = -1$, $\sgap(z^*) = 0$ and $\sgap(z) \to 1$ as $z \to +\infty$. The closed lower endpoint is attained only at $z = 0$, a cost no practical instance admits, and the upper one only by the no-incumbent convention stated below, so on strictly positive finite costs the values lie in the open interval $(-1, 1)$. It is the classical \emph{gap}, the relative excess $g(z, z^*) = (z - z^*)/z^*$ over the reference, put on a bounded scale: $\sgap = g/(2 + g)$ and $g = 2\sgap/(1 - \sgap)$. Hence the two orderings coincide, and every figure below carries the $g$ scale as a second axis. The bounded form is what makes an anytime score well defined without arbitrary truncation, the state before the first incumbent being scored $\sgap = 1$ as the natural limit rather than a cap. The signed construction keeps the metric well defined against a reference a solver can beat. Appendix \ref{app:protocol-metrics} develops both points.

\paragraph{The anytime score} Let $z(t)$ be a run's best-so-far checker-evaluated incumbent at elapsed time $t$, a right-continuous step function equal to $+\infty$ before the first incumbent. Over a horizon $\tlim$, the \emph{normalized signed primal integral} is
\[
\ascore(\tlim) = \frac{1}{\tlim} \int_{0}^{\tlim} \sgap\big(z(t), z^*\big) \, \mathrm{d}t \in [-1, 1],
\]
evaluated exactly as a finite sum over the incumbent steps. Both bounds are inherited from $\sgap$. In this campaign the stronger property holds: the reference being the post-fold snapshot of section \ref{subsec:benchmarks}, no incumbent ever lands below it, so $\sgap \geq 0$ pointwise and every score lies in $[0, 1]$. The instrument itself stays signed for the sake of reference policies a solver can beat, and strictly positive costs already keep the lower endpoint unattained in general. The score charges both how long a run takes to become useful and how well it ends, which is the property an anytime comparison needs that a single final-value table does not have. The quantity integrated is the checker's \textit{Duration} so lower is better.

\paragraph{The final gap} The second metric is the end-of-budget value of the same stream. Since $z(\tlim)$ is the final incumbent, its gap is $g(z(\tlim), z^*)$, which we abbreviate $g_\tlim$ and report in percent alongside every $\ascore$ aggregate, under the same weighting. It answers the question the anytime score deliberately does not privilege: how good is the final incumbent reached at budget exhaustion.

\paragraph{Weighting and estimand} Within a panel every instance has equal weight, the per-instance statistic being the mean over that instance's scheduled seeds and the panel statistic the unweighted mean over its instances. The pooled aggregate is the unweighted mean of the 10 panel means, so each family contributes exactly one fifth of the headline whatever its instance count. Every scheduled run enters every aggregate. A run that produced no valid incumbent contributes $\sgap = 1$ over the whole horizon rather than being dropped, and no failure status is ever excluded. The unequal panel sizes are a deliberate design arbitrage: the stratified 12-instance panels buy breadth, every family and size contributing equally to the headline, while the two panels taken whole study the two key sizes in depth. The pooled estimate consequently mixes precisions, a 56-instance panel weighing the same as a 12-instance one, and the bootstrap below resamples inside each panel, which carries that unequal precision into the reported interval.

\paragraph{Pre-committed confirmatory design} The statistical design was frozen in writing on 2026-08-06, in the same pass as the instance manifest, before any pilot, calibration or campaign record existed or had been inspected. It fixes exactly three confirmatory contrasts, \kayros{} against each Tier-1 contender on the pooled panel-equal mean $\ascore$ at the one-hour budget, and declares that nothing else is confirmatory. Each contrast is paired at instance level and tested by a two-sided bootstrap over $10\,000$ resamples at a fixed seed, the resampling unit being the instance drawn with replacement within its own panel with all of its seeds. The pooled distribution is rebuilt per replicate by recomputing the 10 panel means and averaging them. Multiplicity is controlled by a Holm step-down across the three contenders at the 5\% level, and across nothing else, because there are no per-panel confirmatory tests to correct for. Every per-panel number in this section is descriptive and exploratory, and is labeled as such wherever it appears. Reported goal-attainment levels are 1\%, 5\% and 10\% excess over the reference on the raw scale, with all scheduled runs in the denominator. A median hitting time is quoted only when at least half the runs attain the goal inside the budget.

\paragraph{Seed budget} Because the two $n \geq 1000$ panels carry 5 seeds and the other 8 carry 10, we re-ran the unmodified estimator on the 5-seed subset that every panel shares, which removes that asymmetry at no solver cost. Nothing moved: the arm ordering, all 10 panel winners and all three Holm verdicts are identical, and every pooled mean shifts by at most 0.7\% in relative terms. The confirmatory readout therefore remains the pre-committed full-seed one.

\subsection{Contender Comparison}\label{subsec:contenders}

\paragraph{Pooled result} Table \ref{tab:pooled} gives the pooled panel-equal mean anytime score of every arm with its bootstrap interval, and the three pre-committed contrasts. \kayros{} attains a pooled mean $\ascore$ of 0.008368, against 0.02282 for Timefold, the strongest contender, a 63.3\% lower mean score. Hexaly and jsprit are further behind, and every Tier-1 interval is disjoint from ours. All three contrasts are Holm-significant at the 5\% level, each with an adjusted $p$-value at the resolution floor of the 10\,000-resample bootstrap, and each with a paired difference interval lying entirely below zero. The floor is worth reading literally: the bootstrap cannot resolve a smaller $p$-value than $3 \times 2/10\,001$ after Holm adjustment. The correct statement is therefore that no resample ever reversed the sign of a contrast, not that a particular small number was estimated. The final gap tells the same story at the end of the budget: a pooled mean $g_\tlim$ of 0.84\% for \kayros{} against 2.37\% for Timefold, 7.42\% for Hexaly and 51.3\% for jsprit, a 64.6\% lower mean final gap against the strongest contender. The three descriptive paired contrasts on $g_\tlim$, in the bottom block of table \ref{tab:pooled}, all have intervals entirely below zero. The two metrics thus rank the Tier-1 arms identically, and neither the trajectory weighting of $\ascore$ nor the end-point weighting of $g_\tlim$ changes any conclusion. Figure \ref{fig:pooled-curves} draws the trajectories behind those numbers.

\begin{table}[htbp]
\centering
\small
\caption{Pooled panel-equal results. Top: mean anytime score with its 95 percent bootstrap interval and the mean final gap $g_{\tlim}$. The mean final gap of the time-sliced tier is undefined because 8 of its runs end the hour without any incumbent, the failure mode the bounded score absorbs. Middle: the three pre-committed confirmatory contrasts on $\ascore$ under Holm correction across contenders. Bottom: the same paired contrasts on the mean final gap, in percentage points. These are descriptive, the confirmatory family being the $\ascore$ contrasts alone.}
\label{tab:pooled}
\begin{tabular}{lccc}
\toprule
Arm & Mean $\ascore$ & 95\% CI & Mean $g_{\tlim}$ (\%) \\
\midrule
\kayros{} & 0.008368 & $[0.006950,\, 0.009979]$ & 0.839 \\
Timefold & 0.02282 & $[0.01950,\, 0.02668]$ & 2.37 \\
Hexaly (Tier-1) & 0.04162 & $[0.03581,\, 0.04771]$ & 7.42 \\
jsprit & 0.1416 & $[0.1270,\, 0.1566]$ & 51.3 \\
\midrule
Hexaly (Tier-2, sliced), descriptive & 0.05906 & $[0.04460,\, 0.07725]$ & --- \\
\bottomrule
\end{tabular}

\vspace{1.5ex}

\setlength{\tabcolsep}{4pt}
\begin{tabular}{lccc}
\toprule
Confirmatory contrast, $\ascore$ & Mean difference & 95\% CI & $p_{\text{Holm}}$ \\
\midrule
\kayros{} vs Timefold & $-0.01445$ & $[-0.01843,\, -0.01064]$ & $< 6.0 \times 10^{-4}$ (bootstrap floor) \\
\kayros{} vs Hexaly & $-0.03325$ & $[-0.03907,\, -0.02790]$ & $< 6.0 \times 10^{-4}$ (bootstrap floor) \\
\kayros{} vs jsprit & $-0.1332$ & $[-0.1474,\, -0.1192]$ & $< 6.0 \times 10^{-4}$ (bootstrap floor) \\
\bottomrule
\end{tabular}

\vspace{1.5ex}

\begin{tabular}{lcc}
\toprule
Descriptive contrast, $g_{\tlim}$ (pp) & Mean difference & 95\% CI \\
\midrule
\kayros{} vs Timefold & $-1.534$ & $[-2.524,\, -0.5960]$ \\
\kayros{} vs Hexaly & $-6.576$ & $[-8.289,\, -5.057]$ \\
\kayros{} vs jsprit & $-50.49$ & $[-65.50,\, -38.67]$ \\
\bottomrule
\end{tabular}
\end{table}

\paragraph{Per panel} Table \ref{tab:perpanel} breaks the same runs down by panel, on both metrics, and appendix \ref{app:panels} adds the bootstrap intervals and the paired differences. \kayros{} has the lowest mean score on 7 of the 10 panels, and the mean final gap crowns the same arm as the anytime score on all 10. Timefold takes the other three, \textit{Lera2026} at $n = 200$ and \textit{Vu2020} at $n = 59$ and $n = 99$. On all three the paired difference interval straddles zero, the \textit{Lera2026} $n = 200$ margin being 0.0001 in absolute score. The intervals therefore leave the sign of these three differences unresolved. They do not establish equivalence. On the \textit{Vu2020} panels an hour is long enough for both arms to close most of the distance to a reference that is a proven optimum on 11 of the 12 instances. On all 27 remaining panel contrasts the difference is negative with an interval clear of zero. The widest separations are at the two extremes of the benchmark: \textit{Lera2026} at $n = 1000$, where our mean score is the only one of the four below 0.1, and \textit{Vu2020} at $n = 99$ against jsprit.

\begin{table}[htbp]
\centering
\footnotesize
\setlength{\tabcolsep}{2pt}
\caption{Per-panel mean anytime score and mean final gap of the four Tier-1 arms, best Tier-1 value in bold, with the descriptive time-sliced Hexaly tier closing each block. The two metrics crown the same arm on all 10 panels. The em-rule marks the one panel on which the sliced tier ends at least one run without any incumbent, leaving its mean final gap undefined. Instance counts and reference classes are those of table \ref{tab:panels}.}
\label{tab:perpanel}
\begin{tabular}{lrrrrrcrrrrc}
\toprule
& & \multicolumn{5}{c}{Mean $\ascore$} & \multicolumn{5}{c}{Mean $g_{\tlim}$ (\%)} \\
\cmidrule(lr){3-7} \cmidrule(lr){8-12}
Family & $n$ & \kayros{} & Timefold & Hexaly & jsprit & sliced & \kayros{} & Timefold & Hexaly & jsprit & sliced \\
\midrule
Dabia2013 & 50 & \textbf{0.001322} & 0.01049 & 0.01675 & 0.008715 & 0.01448 & \textbf{0.206} & 2.02 & 2.27 & 1.48 & 1.87 \\
Dabia2013 & 100 & \textbf{0.003198} & 0.008140 & 0.04074 & 0.02479 & 0.04130 & \textbf{0.420} & 1.20 & 6.62 & 4.63 & 6.40 \\
Lera2026 & 200 & 0.01322 & \textbf{0.01308} & 0.06986 & 0.1241 & 0.07768 & 2.02 & \textbf{1.09} & 13.4 & 28.9 & 15.4 \\
Lera2026 & 1000 & \textbf{0.02910} & 0.1290 & 0.1975 & 0.2973 & 0.3488 & \textbf{3.12} & 14.1 & 39.9 & 91.8 & --- \\
Poryos2026 & 100 & \textbf{0.0008584} & 0.002156 & 0.006754 & 0.007288 & 0.006907 & \textbf{0.0638} & 0.175 & 0.897 & 1.03 & 0.872 \\
Poryos2026 & 1000 & \textbf{0.02102} & 0.04366 & 0.02586 & 0.05664 & 0.04354 & \textbf{0.212} & 2.34 & 2.57 & 6.57 & 2.66 \\
Rifki2020 & 50 & \textbf{0.001386} & 0.005786 & 0.01232 & 0.03756 & 0.01282 & \textbf{0.149} & 0.757 & 1.81 & 9.30 & 1.92 \\
Rifki2020 & 60 & \textbf{0.002100} & 0.006615 & 0.01599 & 0.02574 & 0.01657 & \textbf{0.240} & 0.672 & 2.45 & 4.45 & 2.57 \\
Vu2020 & 59 & 0.006172 & \textbf{0.004822} & 0.01244 & 0.3057 & 0.01097 & 1.10 & \textbf{0.800} & 1.56 & 86.0 & 1.38 \\
Vu2020 & 99 & 0.005301 & \textbf{0.004415} & 0.01801 & 0.5281 & 0.01752 & 0.866 & \textbf{0.588} & 2.63 & 279 & 2.43 \\
\bottomrule
\end{tabular}
\end{table}

\begin{figure}[htbp]
\centering
\includegraphics[width=\linewidth]{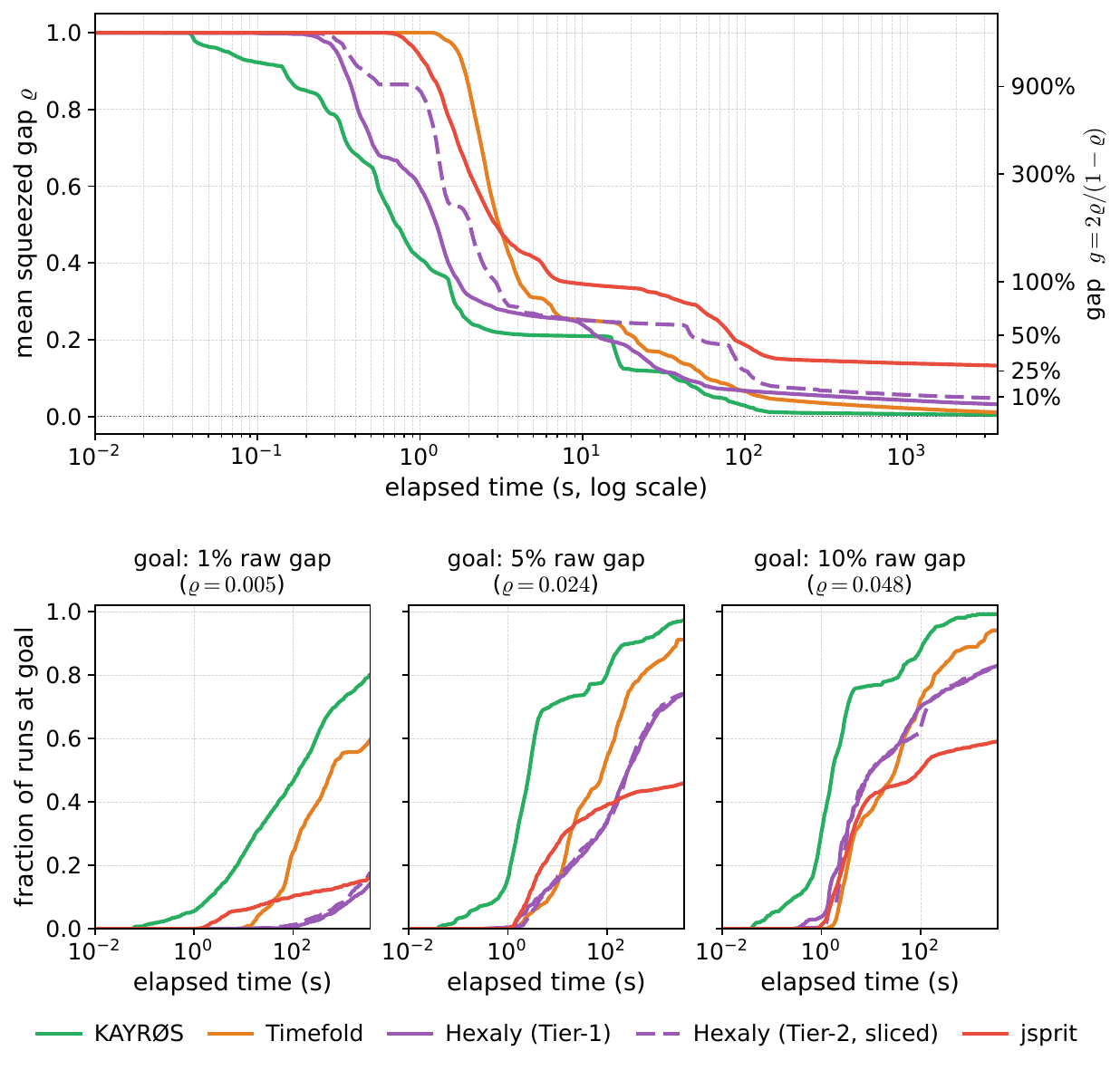}
\caption{Anytime behavior over the whole benchmark, panel-equally weighted over the 10 panels and computed on the same 8\,800 checker-rescored streams as table \ref{tab:pooled}. Top: mean squeezed gap $\sgap$ against elapsed time on a logarithmic time axis, with the classic excess gap $g = (z - z^*)/z^*$ of section \ref{subsec:metrics} as the scale on the right. Bottom: fraction of scheduled runs whose incumbent has reached a 1\%, 5\% and 10\% excess over the reference, with all scheduled runs in the denominator, so a run that never reaches a goal keeps counting against the arm. The time-sliced Hexaly tier is drawn for information and is not part of the ranking. The area under a top-panel curve reproduces the corresponding pooled mean $\ascore$ of table \ref{tab:pooled} to within the grid discretization of the plot. The tables quote the exact estimator and the figure the curve.}
\label{fig:pooled-curves}
\end{figure}

\begin{figure}[htbp]
\centering
\includegraphics[width=\linewidth]{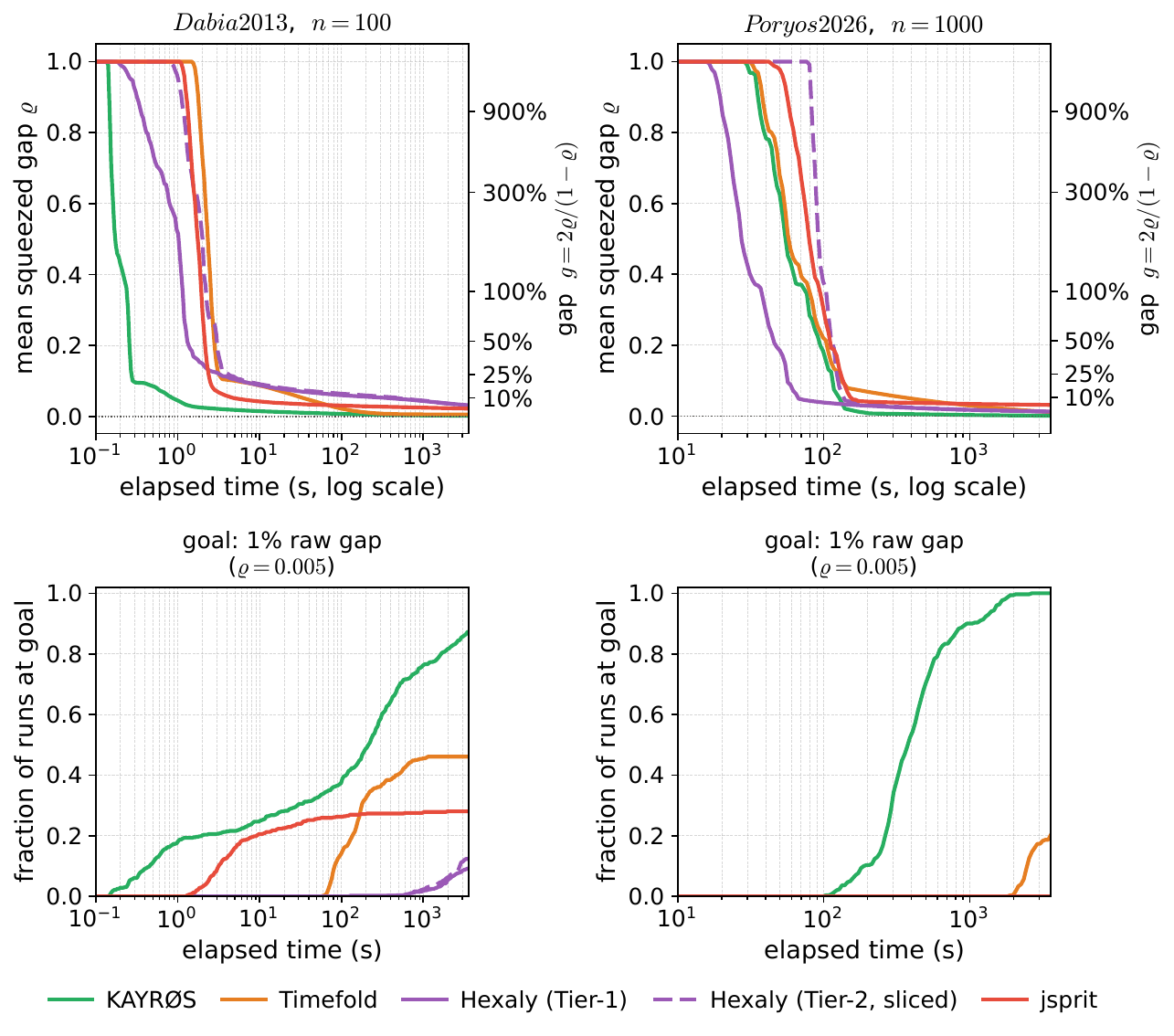}
\caption{Top: mean squeezed gap against elapsed time on the two panels taken whole, the 56 \textit{Dabia2013} instances at $n = 100$ (left, 2\,800 runs) and the 60 \textit{Poryos2026} instances at $n = 1000$ (right, 1\,500 runs), with the gap scale $g$ on the right of each plot. Below each panel, the fraction of that panel's scheduled runs whose incumbent has reached a 1\% excess over the reference, with all scheduled runs in the denominator. The two panels bracket the benchmark. The left one is the classic literature instance set, where 25 of the 56 references are proven optima. The right one is a road-network family at a size no classic time-dependent benchmark reaches, and there the arms separate by an order of magnitude.}
\label{fig:canonical-curves}
\end{figure}

\paragraph{Timefold} Timefold is arguably the strongest contender to beat, and the record should say so rather than let a pooled ratio speak for it. It takes 3 of the 10 panels, and the distance from it to the next arm in table \ref{tab:pooled} is larger than the distance from us to it. It also produced the single largest improvement of the entire campaign: on one \textit{Lera2026} instance at $n = 1000$ it found a solution 13.5\% below the stored record. That solution was folded into the reference like every other campaign improvement, so it is part of the baseline every arm in table \ref{tab:pooled} is scored against.

\paragraph{jsprit} jsprit collapses on two families, with mean scores of 0.31 and 0.53 on the two \textit{Vu2020} panels and 0.12 and 0.30 on the two \textit{Lera2026} panels, far outside its behavior elsewhere. This is a consequence of the model contract recorded in table \ref{tab:arms}, not of the adapter. The solver consumes time-dependent transport-time callbacks but keeps each route at its earliest feasible departure, so it searches for sequences that are good when every route leaves as early as it can. The objective it is scored on, by contrast, is the duration obtained after the checker reschedules those same sequences optimally. Where the two coincide closely, the arm is competitive. On families whose congestion structure makes deliberate waiting the whole point, a solver that cannot postpone a departure is optimizing a different function, and no adapter can repair that from outside. We report the arm at its documented contract rather than dropping it, because a fixed-departure time-dependent solver is a real and common configuration, and because the size of the effect is itself a result.

\paragraph{Hexaly} The commercial contender sits third among the Tier-1 arms, with no panel win and no collapse. Two facts must travel with its number. The score reported here is the one obtained after the measurement-fairness repair of section \ref{subsubsec:threads}, which re-bound the same model through the vendor's C++ external-function interface and lowered the pooled mean $\ascore$ from 0.07577 to 0.04162, so the comparison is made against the better of the two bindings. And the solver is designed for multi-threaded use, a capability the single-core protocol deliberately does not price and which section \ref{subsubsec:threads} measures separately.

\paragraph{The time-sliced tier} The Tier-2 arm pools at 0.05906 against 0.04162 for the same Hexaly build on the rich functions. It is behind Tier-1 on 7 of the 10 panels, taking \textit{Dabia2013} at $n = 50$ and both \textit{Vu2020} panels. Under the tested single-core configurations, the sliced Hexaly arm has a worse pooled score than the external-function arm. Representation and search behavior differ together, so this comparison alone does not identify the source of the gap, a difficult task given the solver is closed-source. On the three panels it does take, the two encodings finish within a quarter of a percentage point of each other. That comparison is descriptive and stops there, for the reason section \ref{sec:discussion} states, and the tier never enters the ranking or the confirmatory family. One contract check must travel with the number. Of the 1\,760 Tier-2 streams, 79 end on a final incumbent the checker rejects (59 on \textit{Lera2026}, 14 on \textit{Dabia2013}, 4 on \textit{Poryos2026} and 2 on \textit{Vu2020}). Separately, 8 runs, all on tight-window \textit{Lera2026} instances at $n = 1000$, produce no checker-valid incumbent at all. A rejected final candidate does not erase an earlier checker-valid incumbent in the other runs. This observed mismatch between model-valid and checker-valid candidates is why the tier is disclosed as a separate model tier of the same solver. Each of the four Tier-1 arms shows zero invalid finals across its own 1\,760 streams. The scores are unaffected by construction, since only checker-valid incumbents ever enter a stream and a run with none is scored at $\sgap = 1$ throughout.

\subsection{Component Evaluations}\label{subsec:components}

\subsubsection{NDCPWLF compositions}\label{subsec:ndcpwlf-benchmark}

\begin{table}[htbp]
\centering
\begin{tabular}{lrr}
\toprule
\textbf{Method} & \textbf{Time (s)} & \textbf{Memory (bytes)} \\
\midrule
\rowcolor{bestmem}
alternative & 1.51137 & 1339072 \\
alternative+tree & 2.02532 & 1350208 \\
lera2020 & 1.44900 & 12258816 \\
lera2020+tree & 0.84868 & 9664512 \\
visser+nor & 0.36802 & 1340736 \\
\rowcolor{besttime}
visser+nor+tree & \textbf{0.35395} & 1428864 \\
visser2020 & 25.35266 & 54170864 \\
visser2020+tree & 0.75250 & 54170944 \\
\bottomrule
\end{tabular}
\caption{Aggregate of the NDCPWLF composition benchmark of section \ref{subsec:ndcpwlf-benchmark}: for each of the 8 method variants, the sum over the 12 grid cells (4 breakpoint bounds by 3 composition-chain sizes) of the mean composition time over 5 repetitions, and the total memory of the composed functions. The fastest variant is shaded green and the most memory-frugal one blue. Per-cell values are in appendix \ref{app:composition}.}
\label{tab:composition-benchmark}
\end{table}

Four composition algorithms were compared, each in a sequential and a tree variant, so eight C++ implementations in all.\footnote{The full benchmarking code, with every implementation and the grid runner, is released at \url{https://github.com/0nyr/pwlf_compare}.} They are \texttt{lera2020} \cite{lera-romeroLinearEdgeCosts2020}; \texttt{visser2020} \cite{visserEfficientMoveEvaluations2020}; \texttt{alternative}, a variant of slightly worse complexity with more aggressive pruning; and \texttt{visser+nor}, our two-pointer formulation of the sweep with active normalization and explicit tail-flush steps. Appendix \ref{app:composition} describes each variant and gives every cell. The benchmark composes randomly generated continuous functions (NDCPWLFs, the jump-free subclass of section \ref{subsec:taxonomy}) standing for arc RTFs, with at most $p \in \{6, 100, 1000, 10000\}$ breakpoints each, along composition chains of $k \in \{2, 100, 1000\}$ functions each ($k - 1$ sequential compositions), and every $(p, k)$ cell is run 5 times. Table \ref{tab:composition-benchmark} aggregates time and memory over the grid. The fastest implementation overall is \texttt{visser+nor+tree}. Less expected, the sequential \texttt{visser+nor} is already about twice as fast as every variant of the other three algorithms, tree variants included, and 69 times faster than the \texttt{visser2020} it extends. The reason is that the redundant breakpoints \texttt{visser2020} keeps are carried through every later composition. On memory, \texttt{alternative} is the most frugal, with \texttt{visser+nor} only 0.1\% above it. Across time and memory, \texttt{visser+nor} is therefore the implementation of choice, in its sequential or its tree form. The event merge \kayros{} ships (algorithm \ref{algo:compose_NDLCF}) performs the same sweep with different bookkeeping and no normalization, for the reason given next.

These gains are measured in a synthetic stress regime that canonical TD instances cannot reach, unrestricted chains of up to $1000$ compositions over functions with up to $10^4$ breakpoints. The time horizon and the TW clamps bound every composition chain, and composed route functions stay orders of magnitude smaller. Random walks on the \textit{Dabia2013} $n{=}100$ instances leave the horizon after roughly $30$ arcs with peak functions of well under $10^2$ breakpoints, and even the heaviest vertical-step families peak below $3 \cdot 10^3$. The stress regime does occur on road-network ATFs, which is where the gains apply. Only the flat/vertical part of the normalization is bit-neutral for exact cost semantics (section \ref{subsec:composition} and appendix \ref{app:taxonomy}), so \kayros{} adopts the exact un-normalized composition as its canonical checker-identical semantics.

\subsubsection{Search strategy head-to-head}

The anytime layer of section \ref{subsec:kayros-anytime} is an ILS, and that is a measured choice. An earlier line of the solver used an ant-colony construction over the same local search. We compared the two head to head, together with a hybrid that warm-starts the iterated local search from the colony's half-budget incumbent, over 20\,808 runs spanning both problem variants, five instance families and sizes from $n = 10$ to $n = 1000$. The iterated local search wins 5\,714 of the 6\,936 paired cells for a mean improvement of 2.52\% in final cost, with a margin that grows with size to 5.44\% at $n = 1000$. At the tested fixed half-budget allocation, the colony warm start did not improve final quality over the iterated local search alone. This evidence is historical with respect to section \ref{subsec:contenders}: it fixed the default strategy at release 0.4.0, long before the 1.6.0 build evaluated here, so it is a decision record and not a same-version ablation on the 212-instance set. Appendix \ref{app:components-ils} gives the full readings, the anytime view of the same verdict and the four small heavily congested instances on which the colony won, which is why it remains available as an option.

\subsubsection{Thread scaling of the commercial contender}\label{subsubsec:threads}

\paragraph{A disclosed change of binding} The Hexaly Tier-1 arm reported above is not the one the first analysis of this campaign produced, and the substitution belongs in the record rather than in a footnote. The pre-committed Tier-1 contract is a rich per-arc piecewise-linear encoding evaluated by an external function. That external function was first bound through the vendor's Python interface, which also declares external functions to the solver as requiring the Global Interpreter Lock (GIL). Rebinding the same model through the vendor's C++ external-function interface multiplies the single-threaded throughput at $n = 100$ by more than 17: the published runs had measured the cost of our bridge as much as the solver behind it. The arm was therefore re-implemented against the C++ binding and the whole Tier-1 slice re-run, on the same model, returning bit-identical values on every gate instance over more than $1.8 \times 10^{5}$ external queries. The descriptive Tier-2 time-sliced arm was rebound and re-run the same way at the same cells. This is a measurement-fairness repair and not a choice among configurations: no confirmatory contender other than Hexaly was re-run, and no parameter was tuned. The instance manifest, the statistical pre-commitment and the frozen scoring reference are the ones fixed before the campaign. None of the solutions the rebuilt arms produced improved on any reference entry, so the baseline every number of this section is scored against is untouched. The repair runs entirely in the contender's favor. Hexaly's pooled mean $\ascore$ falls from 0.07577 under the Python binding to 0.04162, and every comparison reported above is made against the better of the two.

\paragraph{What threads buy} Section \ref{subsec:benchmarks} gives every arm one core, which equalizes resources rather than constraining any particular solver. The convention is still worth defending, because the commercial contender is designed for multi-threaded use, as its authors told us during the campaign. We therefore ran a dedicated thread-scaling study on that solver and report it separately \cite{rascoussierHexalyThreadScaling2026}, with its raw runs on Zenodo \cite{rascoussierHexalyThreadScalingData2026}. It solves 10 instances spanning the five families at one hour per run, on 10 seeds, at 1, 2, 4, 8 and 16 requested threads in both model encodings. Each process is pinned to exactly its requested number of physical cores, for 500 runs per ladder. The two encodings answer differently. The external-function encoding does not convert width into quality: its mean final gap $g_\tlim$ stays between 5.59\% and 5.98\% across the whole ladder, and a single 1000-customer \textit{Lera2026} instance cancels the improvement the 9 others show. However, the time-sliced encoding does: its mean final gap falls by 23.5\% from 1 to 16 threads and its anytime score by 25.4\%, with no saturation at the widest level. What slicing buys is thread scaling, not fidelity traded for throughput. At the single-core budget of this campaign the two encodings are close to a tie on this panel, while the full 212-instance comparison of section \ref{subsec:contenders} puts the exact encoding clearly ahead. None of this changes the convention: it prices a separate trade-off, better and more stable solutions at equal wall time in exchange for up to 16 times the nominal core allocation. Appendix \ref{app:components-threads} reproduces the ladder in table \ref{tab:hexaly-ladder} with the core-occupancy checks, and its gaps are measured against a pre-campaign reference, so they are not comparable with the scores of section \ref{subsec:contenders}.

\subsubsection{Fleet descent and penalty-tolerant squeeze}\label{subsubsec:fleet-eval}

This is the reading of the fleet-aware machinery that section \ref{subsec:kayros-fleet} deferred, including where it fails. Appendix \ref{app:solver-fleet-eval} carries the full figures. Before the fleet-descent phase existed, no run out of 586 matched the reference fleet size on any \textit{Blauth2024} city at $n = 500$ while the summed route durations were already within 5\% of the reference. This located the whole gap in the fleet rather than in the routing. At $n = 500$ the phase settles that, winning the route count 23 times against 0 in an 80-cell paired leg and producing the first solutions to improve on the published references of \citeauthor{blauthVehicleRoutingTimedependent2024} \cite{blauthVehicleRoutingTimedependent2024}. At $n \geq 1000$ it has a ceiling we state as such: across 3\,000 cells the route count reaches the reference on 0.1\% of runs, and no budget or seed volume we ran changed that. The residual duration gap at a matched route count is only 0.3\% to 1.0\%, so what remains at this size is a fleet problem and not a routing problem. The penalty-tolerant squeeze reached the reference fleet on 4 of 30 pairs at $n = 1000$, where the default reached it on 0 of 30. The squeeze failed its pre-committed acceptance gate at $n = 500$, which is why the released default is inert. Searches confined to one vehicle below the published \textit{Blauth2024} record, 90 of them at a 12-hour single-core budget, all ended without attaining the reduced fleet: evidence that the record route counts are tight at that budget.

\subsection{Best-Known Solutions and Certificates}\label{subsec:bks}

\paragraph{Against the exact literature} On \textit{Dabia2013}, \kayros{} holds, at the store state of 2026-08-08, 114 optimality certificates for the TDVRPTW under the \textit{Duration} objective. Of those, 103 reproduce a value published as a proven optimum, and no certificate contradicts one. Another 10 establish a value strictly below a published reference that earlier work left unproven, by 0.06\% to 2.15\%. These are listed per instance in appendix \ref{app:bks}, table \ref{tab:dabia-improvements}. One further instance differs from its published value only through the checker's re-evaluation of the published routes, by less than $10^{-9}$ in relative terms, and stays below the reporting threshold. The comparison base is the canonical re-pricing of the 146 published solutions distributed by \citeauthor{lera-romeroLinearEdgeCosts2020} \cite{lera-romeroLinearEdgeCosts2020}, on the byte-exact canonical data of section \ref{subsec:benchmarks}. That is the only base on which such a comparison is well posed at all.

\paragraph{Records under FleetCostDuration} The \textit{Blauth2024} family \cite{blauthVehicleRoutingTimedependent2024} carries 40 instances, 30 of which come with a high-effort reference solution produced by the authors' own solver at $n = 500$, 1\,000 and 2\,000. Long-budget \kayros{} searches improve all 30 at an unchanged route count on every instance. The improvement runs from 0.032\% to 1.314\% with a mean of 0.254\%, the margin shrinking as size grows. Table \ref{tab:blauth-improvements} in appendix \ref{app:bks} gives every instance. These are accumulated best-effort searches, chained and warm-started over several waves, and section \ref{sec:discussion} states what such a comparison licenses. Read that way, the result says two things at once: that our solver is competitive with a specialized industrial code on measured-traffic delivery instances, and that the reference solutions remain strong baselines that a well-resourced search moves by fractions of a percent.

\paragraph{Certificates} At the store state of 2026-08-08 the public collection holds 704 computational optimality certificates over 2\,072 instances: 466 over the 1\,352 legacy instances of four families in both problem variants, and 238 over the 720 time-dependent \textit{Poryos2026} instances. What one of them establishes, and the four-run protocol that gates it, are stated in section \ref{subsec:kayros-exact}. Coverage runs from $n = 10$ to $n = 100$ and is uneven in an instructive way. On \textit{Vu2020}, 165 of the 168 TDVRPTW instances certify, whereas removing time windows weakens pruning and drives more runs of the same instances to the memory frontier. On \textit{Poryos2026}, almost everything certifies at $n \leq 25$, while only 6 TDVRPTW instances do at $n = 50$ and none above. No certificate is claimed for \textit{Lera2026} or \textit{Blauth2024}, whose sizes or objective lie outside the component's reach. Runs that end open or resource-limited are retained as such rather than downgraded into a weaker claim. Of the certificates, 173 improved a previously stored best-known solution, mostly one of our own heuristic solutions, which measures the relationship between our two solver modes rather than a position against the literature. Certificates have also been withdrawn: one family's entire set of 160 after a travel-time representation defect, and two more after the four-run protocol was re-run from scratch over all 1\,352 legacy instances. Of those 1\,352 legacy instances, 466 certificates were re-derived at their stored values. That record, the re-derivation campaigns it forced, and the per-instance curation behind the tables of this subsection are documented in the companion technical report \cite{rascoussierKAYROSTechReport2026}.

\section{Discussion}\label{sec:discussion}

This section states what the results of section \ref{sec:results} do not establish, and why each boundary is where it is.

\paragraph{What the fleet-cost comparison licenses} The improvements over the published reference solutions under \textit{FleetCostDuration} (section \ref{subsec:bks}) come from long searches, chained and warm-started over several waves, measured against solutions whose own budget and trajectory we do not hold. No anytime, equal-budget or per-unit-of-compute claim is constructible from that, in either direction. The result supports a comparison of final quality between two best-effort processes, and nothing about how either got there. That boundary cannot be pushed from outside, since crossing it would require the other side's traces or its code.

\paragraph{The time-sliced tier has no causal reading} The Tier-2 arm of section \ref{subsec:contenders} is reported for information and never ranked, because the control that would license a causal reading was not run: our own solver on the sliced instances. Without it the representation and the search are confounded, and neither the tier's advantage on its three panels nor its deficit on the other seven can be attributed to one of the two. Its checker-invalid final solutions are an observation about optimizing on an approximation while being scored on the truth, not a measure of the approximation's worth.

\paragraph{Benchmark surface} The benchmark includes two families constructed during this thesis, \textit{Lera2026} and \textit{Poryos2026}, which together occupy two of the five family slots and contribute two fifths of the pooled aggregate under the weighting of section \ref{subsec:metrics}. The original road-network design of \textit{Poryos2026} occupies one of those slots. \textit{Lera2026} extends established benchmark ingredients. These shares were fixed by the protocol before any campaign record existed, and the results are published per panel in full.

\paragraph{The single-core convention} The single-core convention of section \ref{subsec:benchmarks} is defended on measurement in section \ref{subsec:components}. Its scope should still be read narrowly: it prices no claim about multi-threaded behavior. One contender's time-sliced encoding does scale with threads and buys quality that way, whereas \kayros{} is single-threaded by design and takes its parallelism at the campaign level. A multi-core comparison would therefore measure a different trade-off in a different experiment, and it is not one this paper has run.

\paragraph{An earlier decision, honestly dated} The choice of iterated local search over the ant-colony strategy (section \ref{subsec:components}) is a decision record and not an ablation of the build evaluated here. It was measured at an earlier release, on a different instance design, and it fixed a default that has not been re-contested since. It is evidence about two implementations at that point in the solver's history, not a general statement about the two metaheuristic families.

The companion technical report and the public artifacts carry the depth behind several of these results, but no claim here rests on a document a reader cannot check: every number above is regenerated from the frozen campaign artifacts of section \ref{subsec:benchmarks}.

\section{Conclusions and Perspectives}\label{sec:conclusion}

This paper set out four contributions in section \ref{subsec:contributions}, and each can now be read against what precedes. \kayros{} itself is released as an open-source solver that is anytime and exact over one time-dependent engine, under the principle that turned out to be load-bearing throughout: an external, epsilon-free checker defines the objective. Every value the solver publishes is that checker's recomputation of the routes it returns (section \ref{subsec:kayros-arch}). The algorithmic core received the explicit treatment that, to the best of our knowledge, its earlier descriptions had stopped short of. Algorithm \ref{algo:compose_NDLCF} states the composition of non-decreasing left-continuous piecewise-linear functions in full, proved exact under one domain hypothesis that route evaluation always satisfies, with the continuous results of the literature recovered as the jump-free case. That treatment also covers when normalization is exact in floating-point arithmetic, and three independent families of defects that the implicit description of this operation has cost the field (section \ref{subsec:composition}).

The evaluation methodology was applied rather than merely proposed. A normalized signed primal integral, scored against a frozen published reference and aggregated panel-equally under a design pre-committed before any campaign record existed, ranks \kayros{} at a 63.3\% lower pooled score than the strongest of the three solvers a practitioner can actually obtain, with all three contrasts Holm-significant. It also records the three panels where a contender is ahead and the components where our own machinery reaches a ceiling (section \ref{sec:results}). The exact component and its certificate semantics close the set: 704 published optimality certificates, stated as conditional computational results under explicit arithmetic, tolerance and search assumptions, 10 of which establish values below references the exact literature had left unproven. Alongside them come 30 improved high-effort reference solutions under the fleet-cost objective.

Several directions follow from this work. Making the dual bounds of the exact component independently verified safe bounds, by interval arithmetic or by rational certification of the final relaxation, is the single change that would turn a certificate into an artifact checkable without rerunning the solver. The engine and the anytime layer are not specific to the multi-vehicle case, so a port to single-vehicle time-dependent problems, where exact and anytime search is already established \cite{fontaineExactAnytimeApproach2023, fontaineExactAnytimeHeuristic2024}, is a natural next target. Now that the solver is released, a best-known-solution hardening campaign started from published references would test the records of section \ref{subsec:bks} by the only means capable of refuting them, namely a better checker-valid solution found by someone else. The time-sliced observation of section \ref{subsec:contenders} leaves a genuine question open: whether a coarse representation can repay the fidelity it gives up, a question to which the two encodings of the commercial contender answer differently at 1 core and at 16. Answering it calls for a staged-fidelity design in which the search runs on cheap approximations while the canonical checker keeps defining the objective. Two extensions of the model itself remain: the generalized \textit{Travel Time} objective, in which the departure time at every customer becomes a decision \cite{heDynamicDiscretizationDiscovery2022}, and the integration of DDD \cite{bolandPerspectivesIntegerProgramming2019, heDynamicDiscretizationDiscovery2022} into the exact component, the breakpoints of its piecewise-linear functions being exactly the structure such a scheme exploits.

The solver \cite{rascoussierKAYROS2026}, the benchmark instances with the solution store \cite{rascoussierMAMUTrouting2026}, and the campaign data \cite{rascoussierAnytimeSolverComparisonData2026} are public. That is the condition under which everything above can be contested rather than merely believed, and it is the condition under which this work would like to be read.

\section*{Acknowledgements}

This work continues the time-dependent-routing research developed under Christine Solnon\textquotesingle s supervision through the theses of Pénélope Aguiar Melgarejo \cite{aguiar-melgarejoConstraintProgrammingApproach2016} and Romain Fontaine \cite{fontaineExactAnytimeHeuristic2024, fontaineExactAnytimeApproach2023}. This solver stands on open source: We thank Gonzalo Lera-Romero for releasing his branch-price-and-cut solver, on which the exact component of \kayros{} builds, and Leon Lan, Niels Wouda, Wouter Kool and the other PyVRP contributors, whose framework inspired the anytime heuristic layer. We thank Thibaut Vidal, whose open solvers set the standard the field now follows. We thank the solver teams who engaged during the comparison campaign: Geoffrey de Smet and the Timefold team, Maxime Rougier, Fred Gardi and the Hexaly team for the academic license and the exchange on multi-threading, and Jannis Blauth and Dirk Müller for their answers on the Blauth2024 benchmark and its reference solutions. We thank Adrien Pichon for the collaboration on the MAMUT-routing benchmark platform, together with Marc Sevaux and Alexandru-Liviu Olteanu, fellow members of the MAMUT project. We also thank the OpenStreetMap contributors for the geographic data behind the Poryos2026 road networks. We thank Romain Fontaine for his help with Grid'5000, with special thanks to Guillaume Beslon and the BIOTIC team for their support.

Experiments presented in this paper were carried out using the Grid'5000 testbed, supported by a scientific interest group hosted by Inria and including CNRS, RENATER and several universities as well as other organizations (see \url{https://www.grid5000.fr}).

This work is funded by the French National Research Agency (ANR) as part of the MAMUT project, ANR-22-CE22-0016, \enquote{Machine learning And Matheuristics algorithms for Urban Transportation}.

\section*{Data Availability}

\kayros{} 1.6.0, the release evaluated here, is open source under the MIT license \cite{rascoussierKAYROS2026}, on PyPI \cite{pythonsoftwarefoundationPyPI2026} under the name \texttt{kayros} and archived by Software Heritage. The benchmark instances, the canonical checker and the versioned solution store with its certificates are published by the MAMUT-routing platform \cite{rascoussierMAMUTrouting2026}. The campaign archive is deposited on Zenodo under CC BY 4.0 \cite{rascoussierAnytimeSolverComparisonData2026}: 8\,800 solver runs with their checker-validated trajectories, the frozen instance manifest and the reference snapshot of section \ref{subsec:benchmarks}. The analysis code that regenerates every table, statistic and figure of section \ref{sec:results} from that archive is released at \url{https://github.com/0nyr/kayros-campaign-analysis}. The raw runs behind the thread-scaling study of section \ref{subsubsec:threads} are deposited separately, 2\,010 solver runs under CC BY 4.0 \cite{rascoussierHexalyThreadScalingData2026}, accompanying version 2 of the companion report.

\section*{Declaration of AI Use}

Generative AI tools were used throughout the preparation of this work, including frontier models from Anthropic (Claude) and OpenAI (ChatGPT). They contributed to the solver and benchmark tooling, to experiment orchestration and analysis, and to the drafting of this paper. Florian Rascoussier led the review and revision of outputs from these tools. The authors take responsibility for the final content of the work.

\printbibliography

\appendix

\section{Per-Instance Best-Known-Solution Improvements}\label{app:bks}

As part of the validation campaign of the exact reference checker, we re-priced all 146 published best-known solutions of \cite{lera-romeroLinearEdgeCosts2020} for the \textit{Duration} objective on \textit{Dabia2013} under the two readings of the distance model of section \ref{subsec:benchmarks}: the \emph{canonical} one, where the floor-truncated distances $d_{ij} = \lfloor 10\,\lVert \cdot \rVert_2 \rfloor$ (the classic distance letter, kept from the source model and not to be read against the destination depot $d$) are copied byte-exact from the original data, and the \emph{full-precision} one, where distances are re-derived as $10\,\lVert \cdot \rVert_2$ from the Solomon coordinates. On the canonical data, every published BKS is feasible and reproduced exactly by the epsilon-free checker. Under the full-precision reading, the very same routes inherit a mean per-arc inflation of $+0.40$, and 11 of them overshoot a TW deadline by $0.94$ to $4.8$ time units, because optimal \textit{Duration} solutions ride deadlines (appendix \ref{app:taxonomy}). As with the classic VRPTW, where DIMACS-2021-truncated and SINTEF-float evaluations coexist \cite{DIMACS2021, Sintef2008}, this is a matter of evaluation convention rather than an error on either side, but it makes the convention part of any reproducible comparison. The exercise also required correcting a handful of defects in the original 2020 solver, one of the very few open-source TD routing codebases, before it could serve as a reference: two silent index-out-of-bounds reads in the composition routine, a floating-point-to-boolean conversion reducing one TW-tightening rule to a near no-op, a division by zero in the segment-intersection primitive and occasional $10^{-14}$-scale violations of the non-decreasing invariant. None of them affects the 146 published BKS, in which the only observable trace is the solver's epsilon-based comparisons, whose tolerance is $\varepsilon = 10^{-5}$: the published departure time of one route of RC105\_100 stops $10^{-5}$ short of the duration-minimizing breakpoint, within one $\varepsilon$ of it and hence epsilon-equal under the original tolerance, so its published duration is improvable by $10^{-5}$. Per-instance values are given in the thesis appendix and in the public solution store \cite{rascoussierMAMUTrouting2026}.

\begin{table}[htbp]
\centering
\small
\caption{The 10 \textit{Dabia2013} instances for which the \kayros{}-certified optimum lies strictly below a published reference value that earlier work left unproven, under the \textit{Duration} objective on the canonical checker.}
\label{tab:dabia-improvements}
\begin{tabular}{lrrr}
\toprule
Instance & Published best value & \kayros{} certified optimum & Improvement \\
\midrule
RC106\_50 & 11830.264 & 11756.556 & $-0.62$\% \\
R104\_100 & 18344.469 & 17949.848 & $-2.15$\% \\
R107\_100 & 18890.956 & 18704.105 & $-0.99$\% \\
R108\_100 & 17626.206 & 17400.040 & $-1.28$\% \\
R110\_100 & 18786.882 & 18607.788 & $-0.95$\% \\
RC101\_100 & 24778.294 & 24762.734 & $-0.06$\% \\
RC102\_100 & 22689.967 & 22369.064 & $-1.41$\% \\
RC103\_100 & 20680.027 & 20325.853 & $-1.71$\% \\
RC106\_100 & 21382.243 & 21233.331 & $-0.70$\% \\
RC107\_100 & 20101.592 & 19918.207 & $-0.91$\% \\
\bottomrule
\end{tabular}
\end{table}

\begin{table}[htbp]
\centering
\small
\caption{The 30 high-effort BonnTour references of the \textit{Blauth2024} family of Blauth et al., all improved by long-budget \kayros{} searches at an unchanged route count, under the \textit{FleetCostDuration} objective.}
\label{tab:blauth-improvements}
\begin{tabular}{lrrrr}
\toprule
Instance & BonnTour (\$) & \kayros{} (\$) & $\left| \mathcal{S} \right|$ & Improvement \\
\midrule
berlin\_500 & 2826.49 & 2823.83 & 9 & $-0.094$\% \\
cincinnati\_500 & 2906.21 & 2896.52 & 9 & $-0.333$\% \\
kyiv\_500 & 3145.48 & 3143.59 & 10 & $-0.060$\% \\
london\_500 & 4820.82 & 4783.91 & 15 & $-0.766$\% \\
madrid\_500 & 3168.94 & 3167.92 & 10 & $-0.032$\% \\
nairobi\_500 & 3112.93 & 3072.03 & 10 & $-1.314$\% \\
new\_york\_500 & 3417.74 & 3376.62 & 11 & $-1.203$\% \\
san\_francisco\_500 & 3853.58 & 3851.57 & 12 & $-0.052$\% \\
sao\_paulo\_500 & 3909.96 & 3904.53 & 12 & $-0.139$\% \\
seattle\_500 & 3178.87 & 3167.01 & 10 & $-0.373$\% \\
berlin\_1000 & 4980.18 & 4975.37 & 16 & $-0.097$\% \\
cincinnati\_1000 & 5058.90 & 5055.28 & 16 & $-0.072$\% \\
kyiv\_1000 & 5390.71 & 5383.07 & 17 & $-0.142$\% \\
london\_1000 & 7870.40 & 7844.02 & 24 & $-0.335$\% \\
madrid\_1000 & 5420.04 & 5411.69 & 17 & $-0.154$\% \\
nairobi\_1000 & 5310.90 & 5286.61 & 17 & $-0.457$\% \\
new\_york\_1000 & 5454.59 & 5449.87 & 17 & $-0.087$\% \\
san\_francisco\_1000 & 6455.29 & 6437.71 & 20 & $-0.272$\% \\
sao\_paulo\_1000 & 6752.29 & 6734.43 & 21 & $-0.264$\% \\
seattle\_1000 & 5388.10 & 5379.64 & 17 & $-0.157$\% \\
berlin\_2000 & 8849.64 & 8844.57 & 28 & $-0.057$\% \\
cincinnati\_2000 & 8722.92 & 8717.06 & 27 & $-0.067$\% \\
kyiv\_2000 & 9567.13 & 9552.16 & 30 & $-0.157$\% \\
london\_2000 & 13273.04 & 13247.73 & 40 & $-0.191$\% \\
madrid\_2000 & 9308.40 & 9289.93 & 29 & $-0.198$\% \\
nairobi\_2000 & 8887.42 & 8867.89 & 28 & $-0.220$\% \\
new\_york\_2000 & 9560.46 & 9552.46 & 30 & $-0.084$\% \\
san\_francisco\_2000 & 11021.79 & 11010.77 & 34 & $-0.100$\% \\
sao\_paulo\_2000 & 11258.19 & 11247.72 & 35 & $-0.093$\% \\
seattle\_2000 & 9179.54 & 9173.76 & 29 & $-0.063$\% \\
\bottomrule
\end{tabular}
\end{table}

\section{Per-Panel Statistics}\label{app:panels}

This appendix carries the per-panel detail behind table \ref{tab:perpanel}, which prints the panel means without intervals. Table \ref{tab:perpanel-ci} adds the 95 percent bootstrap interval of every arm on every panel, including the descriptive time-sliced Hexaly tier. Table \ref{tab:perpanel-diffs} gives the 30 paired differences of \kayros{} against each Tier-1 contender with their own intervals. Both tables are descriptive and exploratory, as section \ref{subsec:contenders} states: the confirmatory family is the three pooled contrasts alone. No per-panel contrast was pre-committed, and none of the 30 carries a multiplicity correction. The two tables are published in full so that the three differences whose interval contains zero can be read at the same resolution as the 27 that are clear of it.

\begingroup
\small
\setlength{\tabcolsep}{5pt}
\begin{longtable}{lrrlrc}
\caption{Per-panel mean anytime score of every arm with its 95 percent bootstrap interval, the detail behind table \ref{tab:perpanel}. The interval resamples instances with replacement inside the panel, keeping every seed of a resampled instance, over the same $10\,000$ resamples and the same seed as the pooled aggregate. The lowest Tier-1 mean of each panel is in bold and the time-sliced Hexaly tier is reported for information only. All per-panel results are descriptive and exploratory, never confirmatory.}
\label{tab:perpanel-ci} \\
\toprule
Family & $n$ & Inst. & Arm & Mean $\ascore$ & 95\% CI \\
\midrule
\endfirsthead
\toprule
Family & $n$ & Inst. & Arm & Mean $\ascore$ & 95\% CI \\
\midrule
\endhead
\midrule
\multicolumn{6}{r}{\textit{continued on the next page}} \\
\endfoot
\bottomrule
\endlastfoot
Dabia2013 & 50 & 12 & \kayros{} & \textbf{0.001322} & $[0.0002840,\, 0.002962]$ \\*
 & &  & Timefold & 0.01049 & $[0.005095,\, 0.01733]$ \\*
 & &  & Hexaly & 0.01675 & $[0.01050,\, 0.02328]$ \\*
 & &  & jsprit & 0.008715 & $[0.002980,\, 0.01792]$ \\*
 & &  & Hexaly (Tier-2, sliced) & 0.01448 & $[0.008933,\, 0.02022]$ \\
\addlinespace
Dabia2013 & 100 & 56 & \kayros{} & \textbf{0.003198} & $[0.002484,\, 0.003929]$ \\*
 & &  & Timefold & 0.008140 & $[0.006983,\, 0.009292]$ \\*
 & &  & Hexaly & 0.04074 & $[0.03511,\, 0.04635]$ \\*
 & &  & jsprit & 0.02479 & $[0.01945,\, 0.03018]$ \\*
 & &  & Hexaly (Tier-2, sliced) & 0.04130 & $[0.03518,\, 0.04758]$ \\
\addlinespace
Lera2026 & 200 & 12 & \kayros{} & 0.01322 & $[0.006692,\, 0.02029]$ \\*
 & &  & Timefold & \textbf{0.01308} & $[0.01010,\, 0.01625]$ \\*
 & &  & Hexaly & 0.06986 & $[0.05043,\, 0.08942]$ \\*
 & &  & jsprit & 0.1241 & $[0.07512,\, 0.1737]$ \\*
 & &  & Hexaly (Tier-2, sliced) & 0.07768 & $[0.05486,\, 0.1007]$ \\
\addlinespace
Lera2026 & 1000 & 12 & \kayros{} & \textbf{0.02910} & $[0.01799,\, 0.04321]$ \\*
 & &  & Timefold & 0.1290 & $[0.09655,\, 0.1665]$ \\*
 & &  & Hexaly & 0.1975 & $[0.1434,\, 0.2540]$ \\*
 & &  & jsprit & 0.2973 & $[0.2118,\, 0.3767]$ \\*
 & &  & Hexaly (Tier-2, sliced) & 0.3488 & $[0.2071,\, 0.5283]$ \\
\addlinespace
Poryos2026 & 100 & 12 & \kayros{} & \textbf{0.0008584} & $[0.0006223,\, 0.001207]$ \\*
 & &  & Timefold & 0.002156 & $[0.001826,\, 0.002492]$ \\*
 & &  & Hexaly & 0.006754 & $[0.005524,\, 0.008037]$ \\*
 & &  & jsprit & 0.007288 & $[0.005188,\, 0.009560]$ \\*
 & &  & Hexaly (Tier-2, sliced) & 0.006907 & $[0.005584,\, 0.008331]$ \\
\addlinespace
Poryos2026 & 1000 & 60 & \kayros{} & \textbf{0.02102} & $[0.01913,\, 0.02300]$ \\*
 & &  & Timefold & 0.04366 & $[0.03926,\, 0.04846]$ \\*
 & &  & Hexaly & 0.02586 & $[0.02435,\, 0.02748]$ \\*
 & &  & jsprit & 0.05664 & $[0.05349,\, 0.05989]$ \\*
 & &  & Hexaly (Tier-2, sliced) & 0.04354 & $[0.04153,\, 0.04583]$ \\
\addlinespace
Rifki2020 & 50 & 12 & \kayros{} & \textbf{0.001386} & $[0.001037,\, 0.001779]$ \\*
 & &  & Timefold & 0.005786 & $[0.004979,\, 0.006655]$ \\*
 & &  & Hexaly & 0.01232 & $[0.01090,\, 0.01357]$ \\*
 & &  & jsprit & 0.03756 & $[0.009210,\, 0.09152]$ \\*
 & &  & Hexaly (Tier-2, sliced) & 0.01282 & $[0.01109,\, 0.01441]$ \\
\addlinespace
Rifki2020 & 60 & 12 & \kayros{} & \textbf{0.002100} & $[0.001571,\, 0.002649]$ \\*
 & &  & Timefold & 0.006615 & $[0.005996,\, 0.007234]$ \\*
 & &  & Hexaly & 0.01599 & $[0.01504,\, 0.01710]$ \\*
 & &  & jsprit & 0.02574 & $[0.01359,\, 0.04075]$ \\*
 & &  & Hexaly (Tier-2, sliced) & 0.01657 & $[0.01566,\, 0.01741]$ \\
\addlinespace
Vu2020 & 59 & 12 & \kayros{} & 0.006172 & $[0.004098,\, 0.008453]$ \\*
 & &  & Timefold & \textbf{0.004822} & $[0.003557,\, 0.006189]$ \\*
 & &  & Hexaly & 0.01244 & $[0.01117,\, 0.01404]$ \\*
 & &  & jsprit & 0.3057 & $[0.2545,\, 0.3585]$ \\*
 & &  & Hexaly (Tier-2, sliced) & 0.01097 & $[0.009342,\, 0.01270]$ \\
\addlinespace
Vu2020 & 99 & 12 & \kayros{} & 0.005301 & $[0.003456,\, 0.007544]$ \\*
 & &  & Timefold & \textbf{0.004415} & $[0.003568,\, 0.005303]$ \\*
 & &  & Hexaly & 0.01801 & $[0.01554,\, 0.02061]$ \\*
 & &  & jsprit & 0.5281 & $[0.4494,\, 0.6167]$ \\*
 & &  & Hexaly (Tier-2, sliced) & 0.01752 & $[0.01513,\, 0.01989]$ \\
\end{longtable}
\endgroup

\begingroup
\small
\setlength{\tabcolsep}{5pt}
\begin{longtable}{lrrlrc}
\caption{Per-panel paired difference on mean anytime score, \kayros{} minus contender, with its 95 percent bootstrap interval. A negative value favors \kayros{}. The three rows marked $\dagger$ are the ones whose interval contains zero, all against Timefold, and the \textit{Lera2026} $n = 200$ difference of 0.0001 is the closest of the 30. The remaining 27 intervals lie entirely below zero. These contrasts are descriptive and exploratory, carry no multiplicity correction and are never confirmatory: the confirmatory family is the three pooled contrasts of table \ref{tab:pooled}.}
\label{tab:perpanel-diffs} \\
\toprule
Family & $n$ & Inst. & Contender & Mean difference & 95\% CI \\
\midrule
\endfirsthead
\toprule
Family & $n$ & Inst. & Contender & Mean difference & 95\% CI \\
\midrule
\endhead
\midrule
\multicolumn{6}{r}{\textit{continued on the next page}} \\
\endfoot
\bottomrule
\endlastfoot
Dabia2013 & 50 & 12 & Timefold & $-0.009172$ & $[-0.01603,\, -0.003976]$ \\*
 & &  & Hexaly & $-0.01543$ & $[-0.02167,\, -0.009422]$ \\*
 & &  & jsprit & $-0.007393$ & $[-0.01511,\, -0.002476]$ \\
\addlinespace
Dabia2013 & 100 & 56 & Timefold & $-0.004941$ & $[-0.005692,\, -0.004195]$ \\*
 & &  & Hexaly & $-0.03754$ & $[-0.04259,\, -0.03233]$ \\*
 & &  & jsprit & $-0.02160$ & $[-0.02640,\, -0.01701]$ \\
\addlinespace
Lera2026 & 200 & 12 & Timefold$^{\dagger}$ & $0.0001362$ & $[-0.005301,\, 0.005966]$ \\*
 & &  & Hexaly & $-0.05664$ & $[-0.07223,\, -0.04179]$ \\*
 & &  & jsprit & $-0.1109$ & $[-0.1548,\, -0.06751]$ \\
\addlinespace
Lera2026 & 1000 & 12 & Timefold & $-0.09992$ & $[-0.1379,\, -0.06296]$ \\*
 & &  & Hexaly & $-0.1684$ & $[-0.2238,\, -0.1167]$ \\*
 & &  & jsprit & $-0.2682$ & $[-0.3435,\, -0.1889]$ \\
\addlinespace
Poryos2026 & 100 & 12 & Timefold & $-0.001298$ & $[-0.001690,\, -0.0008890]$ \\*
 & &  & Hexaly & $-0.005896$ & $[-0.006980,\, -0.004822]$ \\*
 & &  & jsprit & $-0.006429$ & $[-0.008530,\, -0.004475]$ \\
\addlinespace
Poryos2026 & 1000 & 60 & Timefold & $-0.02264$ & $[-0.02690,\, -0.01874]$ \\*
 & &  & Hexaly & $-0.004836$ & $[-0.006175,\, -0.003530]$ \\*
 & &  & jsprit & $-0.03562$ & $[-0.03803,\, -0.03319]$ \\
\addlinespace
Rifki2020 & 50 & 12 & Timefold & $-0.004400$ & $[-0.005130,\, -0.003703]$ \\*
 & &  & Hexaly & $-0.01093$ & $[-0.01222,\, -0.009515]$ \\*
 & &  & jsprit & $-0.03617$ & $[-0.09013,\, -0.007751]$ \\
\addlinespace
Rifki2020 & 60 & 12 & Timefold & $-0.004515$ & $[-0.004911,\, -0.004120]$ \\*
 & &  & Hexaly & $-0.01389$ & $[-0.01501,\, -0.01282]$ \\*
 & &  & jsprit & $-0.02364$ & $[-0.03871,\, -0.01128]$ \\
\addlinespace
Vu2020 & 59 & 12 & Timefold$^{\dagger}$ & $0.001349$ & $[-0.001262,\, 0.004114]$ \\*
 & &  & Hexaly & $-0.006273$ & $[-0.008307,\, -0.004202]$ \\*
 & &  & jsprit & $-0.2995$ & $[-0.3479,\, -0.2482]$ \\
\addlinespace
Vu2020 & 99 & 12 & Timefold$^{\dagger}$ & $0.0008866$ & $[-0.0007802,\, 0.002792]$ \\*
 & &  & Hexaly & $-0.01271$ & $[-0.01484,\, -0.01053]$ \\*
 & &  & jsprit & $-0.5228$ & $[-0.6077,\, -0.4440]$ \\
\end{longtable}
\endgroup

\section{NDCPWLF Composition Benchmark Grid}\label{app:composition}

The four algorithms, each in a sequential and a \texttt{+tree} variant, are the following. \texttt{lera2020} is the open-source implementation of \citeauthor{lera-romeroLinearEdgeCosts2020} \cite{lera-romeroLinearEdgeCosts2020}, which imposes normalization, that is, the active removal of redundant breakpoints. \texttt{visser2020} is the algorithm described in \cite{visserEfficientMoveEvaluations2020}, of complexity $\mathcal{O}(\phi_f + \phi_g)$. \texttt{alternative} is a variant of slightly worse complexity, $\mathcal{O}((\phi_f + \phi_g) \log_2 (\phi_f + \phi_g))$, with more aggressive pruning. \texttt{visser+nor} is \texttt{visser2020} with active normalization and explicit tail-flush steps, our two-pointer formulation of the sweep. The event merge shipped in \kayros{} (algorithm \ref{algo:compose_NDLCF}) performs the same $\mathcal{O}(\phi_f + \phi_g)$ sweep with different bookkeeping and was not part of this benchmark, whose functions are continuous. The \texttt{+tree} variant of each uses the tree-based composition of \cite{visserEfficientMoveEvaluations2020, blauthVehicleRoutingTimedependent2024}, which performs its compositions in $\log_2(k)$ levels rather than sequentially along the $k$ composed functions and is therefore expected to scale better on long composition chains. The tree variant of \texttt{visser2020} recovers most of the time that its sequential form loses to \texttt{visser+nor}, and none of the memory, which is where normalization separates the methods.

Table \ref{tab:composition-grid-time} and table \ref{tab:composition-grid-memory} give the per-cell detail behind the aggregate of section \ref{subsec:ndcpwlf-benchmark}, with one row per pair formed by a maximum breakpoint count $p$ and a number of composed functions $k$, and one column per algorithm in its sequential and in its tree variant. Each time cell is the mean over the 5 repetitions of that configuration, and each memory cell is the size of the resulting composed function. The totals printed at the bottom of the two tables are therefore the two columns of table \ref{tab:composition-benchmark}. The grid shows how concentrated that aggregate is: every cell at $k = 2$ costs at most $1.3$ milliseconds for every method, the four rows at $k = 1000$ carry between $88.5\%$ and $99.2\%$ of each method's total time, and the single largest cell, at $p = 10^4$ and $k = 10^3$, accounts on its own for $90\%$ of the \texttt{visser2020} total. The synthetic stress regime discussed in section \ref{subsec:ndcpwlf-benchmark} is precisely those bottom rows, and the caveat stated there applies to them without being repeated here.

\begin{table}[htbp]
\centering
\small
\setlength{\tabcolsep}{3pt}
\caption{Per-cell composition time in seconds of the NDCPWLF benchmark of section \ref{subsec:ndcpwlf-benchmark}, for every maximum breakpoint count $p$ and every number of composed functions $k$, with the sequential and the tree variant of each of the four algorithms. Each cell is the mean over the 5 repetitions of that configuration, and the fastest method of each row is in bold when it is unique. The last row repeats the totals of table \ref{tab:composition-benchmark}, taken from the same recorded run at full precision. Cells are printed with the six significant digits of the benchmark log, so a column re-summed from them can differ from its total in the last digit. The stress-regime caveat of section \ref{subsec:ndcpwlf-benchmark} applies to every cell here, and to the bottom right corner of the grid above all.}
\label{tab:composition-grid-time}
\begin{tabular}{rrrrrrrrrr}
\toprule
& & \multicolumn{2}{c}{\texttt{alternative}} & \multicolumn{2}{c}{\texttt{lera2020}} & \multicolumn{2}{c}{\texttt{visser+nor}} & \multicolumn{2}{c}{\texttt{visser2020}} \\
\cmidrule(lr){3-4} \cmidrule(lr){5-6} \cmidrule(lr){7-8} \cmidrule(lr){9-10}
$p$ & $k$ & seq. & tree & seq. & tree & seq. & tree & seq. & tree \\
\midrule
6 & 2 & $2.24\mathrm{e}{-6}$ & $2.96\mathrm{e}{-6}$ & $2.10\mathrm{e}{-6}$ & $1.91\mathrm{e}{-6}$ & {\boldmath$1.24\mathrm{e}{-6}$} & $1.76\mathrm{e}{-6}$ & $1.29\mathrm{e}{-6}$ & $1.53\mathrm{e}{-6}$ \\
6 & 100 & $8.47\mathrm{e}{-5}$ & $1.17\mathrm{e}{-4}$ & $7.48\mathrm{e}{-5}$ & $8.75\mathrm{e}{-5}$ & {\boldmath$5.46\mathrm{e}{-5}$} & $7.55\mathrm{e}{-5}$ & $5.60\mathrm{e}{-5}$ & $7.89\mathrm{e}{-5}$ \\
6 & 1000 & $3.95\mathrm{e}{-4}$ & $1.18\mathrm{e}{-3}$ & $4.07\mathrm{e}{-4}$ & $9.45\mathrm{e}{-4}$ & {\boldmath$3.77\mathrm{e}{-4}$} & $7.84\mathrm{e}{-4}$ & $3.92\mathrm{e}{-4}$ & $7.95\mathrm{e}{-4}$ \\
\addlinespace
100 & 2 & $1.29\mathrm{e}{-5}$ & $1.08\mathrm{e}{-5}$ & $8.68\mathrm{e}{-6}$ & $6.72\mathrm{e}{-6}$ & $3.29\mathrm{e}{-6}$ & $3.86\mathrm{e}{-6}$ & {\boldmath$2.62\mathrm{e}{-6}$} & $3.24\mathrm{e}{-6}$ \\
100 & 100 & $1.75\mathrm{e}{-2}$ & $4.89\mathrm{e}{-3}$ & $9.10\mathrm{e}{-3}$ & $1.72\mathrm{e}{-3}$ & $3.38\mathrm{e}{-3}$ & $8.28\mathrm{e}{-4}$ & $2.30\mathrm{e}{-3}$ & {\boldmath$6.13\mathrm{e}{-4}$} \\
100 & 1000 & $9.95\mathrm{e}{-1}$ & $7.23\mathrm{e}{-2}$ & $6.88\mathrm{e}{-1}$ & $2.51\mathrm{e}{-2}$ & $2.38\mathrm{e}{-1}$ & $1.24\mathrm{e}{-2}$ & $2.09\mathrm{e}{-1}$ & {\boldmath$9.23\mathrm{e}{-3}$} \\
\addlinespace
1000 & 2 & $1.88\mathrm{e}{-4}$ & $1.82\mathrm{e}{-4}$ & $5.04\mathrm{e}{-5}$ & $4.89\mathrm{e}{-5}$ & $2.00\mathrm{e}{-5}$ & $1.90\mathrm{e}{-5}$ & $1.42\mathrm{e}{-5}$ & {\boldmath$1.19\mathrm{e}{-5}$} \\
1000 & 100 & $1.12\mathrm{e}{-1}$ & $4.76\mathrm{e}{-2}$ & $4.98\mathrm{e}{-2}$ & $1.52\mathrm{e}{-2}$ & $2.06\mathrm{e}{-2}$ & $6.94\mathrm{e}{-3}$ & $3.25\mathrm{e}{-2}$ & {\boldmath$6.73\mathrm{e}{-3}$} \\
1000 & 1000 & $2.28\mathrm{e}{-1}$ & $4.83\mathrm{e}{-1}$ & $4.04\mathrm{e}{-1}$ & $1.85\mathrm{e}{-1}$ & {\boldmath$4.61\mathrm{e}{-2}$} & $7.32\mathrm{e}{-2}$ & $2.03\mathrm{e}{0}$ & $8.74\mathrm{e}{-2}$ \\
\addlinespace
10000 & 2 & $1.30\mathrm{e}{-3}$ & $1.22\mathrm{e}{-3}$ & $3.29\mathrm{e}{-4}$ & $3.20\mathrm{e}{-4}$ & $1.51\mathrm{e}{-4}$ & $1.53\mathrm{e}{-4}$ & {\boldmath$1.22\mathrm{e}{-4}$} & $1.26\mathrm{e}{-4}$ \\
10000 & 100 & $4.24\mathrm{e}{-2}$ & $1.26\mathrm{e}{-1}$ & $3.70\mathrm{e}{-2}$ & $5.42\mathrm{e}{-2}$ & {\boldmath$9.84\mathrm{e}{-3}$} & $2.17\mathrm{e}{-2}$ & $2.00\mathrm{e}{-1}$ & $4.14\mathrm{e}{-2}$ \\
10000 & 1000 & $1.15\mathrm{e}{-1}$ & $1.29\mathrm{e}{0}$ & $2.60\mathrm{e}{-1}$ & $5.66\mathrm{e}{-1}$ & {\boldmath$5.01\mathrm{e}{-2}$} & $2.38\mathrm{e}{-1}$ & $2.29\mathrm{e}{1}$ & $6.06\mathrm{e}{-1}$ \\
\midrule
\multicolumn{2}{l}{Total} & 1.51137 & 2.02532 & 1.44900 & 0.84868 & 0.36802 & 0.35395 & 25.35266 & 0.75250 \\
\bottomrule
\end{tabular}
\end{table}

\begin{table}[htbp]
\centering
\small
\setlength{\tabcolsep}{3pt}
\caption{Per-cell memory of the composed function in bytes for the same benchmark and the same cells as table \ref{tab:composition-grid-time}, with the smallest value of each row in bold when it is unique. The last row is the exact integer sum of the 12 cells of its column, and reproduces the memory column of table \ref{tab:composition-benchmark}. The deepest chains separate the methods most: at $k = 1000$ and $p \geq 1000$, the composed functions of \texttt{visser2020}, which removes no redundant breakpoint, are several orders of magnitude larger than those of the methods that normalize.}
\label{tab:composition-grid-memory}
\begin{tabular}{rrrrrrrrrr}
\toprule
& & \multicolumn{2}{c}{\texttt{alternative}} & \multicolumn{2}{c}{\texttt{lera2020}} & \multicolumn{2}{c}{\texttt{visser+nor}} & \multicolumn{2}{c}{\texttt{visser2020}} \\
\cmidrule(lr){3-4} \cmidrule(lr){5-6} \cmidrule(lr){7-8} \cmidrule(lr){9-10}
$p$ & $k$ & seq. & tree & seq. & tree & seq. & tree & seq. & tree \\
\midrule
6 & 2 & 200 & 200 & 656 & 656 & 200 & 200 & 200 & 200 \\
6 & 100 & 56 & 56 & 80 & 80 & 56 & 56 & 56 & 56 \\
6 & 1000 & 56 & 56 & 80 & 80 & 56 & 56 & 56 & 56 \\
\addlinespace
100 & 2 & 3224 & 3224 & 18512 & 18512 & 3224 & 3224 & 3224 & 3224 \\
100 & 100 & 142264 & 142264 & 1179728 & 1179728 & 142264 & 142264 & 144808 & 144808 \\
100 & 1000 & \textbf{596440} & 596744 & 4718672 & 4718672 & 596520 & 597240 & 1060664 & 1060664 \\
\addlinespace
1000 & 2 & 30888 & 30888 & 147536 & 147536 & 30888 & 30888 & 31416 & 31416 \\
1000 & 100 & \textbf{394232} & 405000 & 2359376 & 2359376 & 395816 & 408120 & 1493864 & 1493864 \\
1000 & 1000 & 1448 & 1512 & 1179728 & 18512 & 1448 & 30968 & 7766648 & 7766680 \\
\addlinespace
10000 & 2 & 157544 & 157544 & 1179728 & 1179728 & 157544 & 157544 & 210296 & 210296 \\
10000 & 100 & 11464 & 11464 & 1179728 & 36944 & 11464 & 44584 & 9770888 & 9770904 \\
10000 & 1000 & 1256 & 1256 & 294992 & 4688 & 1256 & 13720 & 33688744 & 33688776 \\
\midrule
\multicolumn{2}{l}{Total} & 1339072 & 1350208 & 12258816 & 9664512 & 1340736 & 1428864 & 54170864 & 54170944 \\
\bottomrule
\end{tabular}
\end{table}

\section{Function Taxonomy: Properties, Complexity and Exactness}\label{app:taxonomy}

This appendix carries the property and complexity statements of the six functions of section \ref{subsec:taxonomy}, in the order of that section and under the signatures given there. It also gives the proof of theorem \ref{theorem:comp_of_rrtf}, the proofs of the composition theorems and propositions of section \ref{subsec:taxonomy} with the example of what fails without hypothesis (H), and the full exactness analysis of the composition of section \ref{subsec:composition} with the evidence behind the three defect families it summarizes.

\subsection{Properties and complexity}

\paragraph{Arc Travel Time Function (TTF)}

\noindent\textbf{Properties:} The function $\tau_{ij}$ \cite{dabiaBranchPriceTimeDependent2013, lera-romeroLinearEdgeCosts2020} is a left-continuous PWL function over $\mathcal{T}_{\mathrm{arc}}$ whose values are durations, hence the codomain $\mathbb{R}_{\geq 0}$ rather than a time set. It satisfies the FIFO condition in the strong sense, as the associated arrival function $\alpha_{ij}(t) = t + \tau_{ij}(t)$ is non-decreasing. Consequently, $\tau_{ij}$ has slope bounded below by $-1$, i.e. $\frac{d \tau_{ij}(t)}{dt} \geq -1$ wherever the derivative is defined \cite{fontaineExactAnytimeHeuristic2024}. The number of breakpoints, counted as chain points with duplicate abscissae included, is bounded by the instance-wide parameter $p \in \mathbb{N}_{>0}$: $\forall \langle i, j \rangle \in \mathcal{A}, \phi_{\tau_{ij}} \leq p$.

\noindent\textbf{Complexity:} Considered given by a TDVRPTW instance, so $\mathcal{O}(1)$ in time and $\mathcal{O}(2p) = \mathcal{O}(p)$ in memory.

\paragraph{Arc Arrival Time Function (ATF)}

\noindent\textbf{Properties:} The function $\alpha_{ij}$ \cite{visserEfficientMoveEvaluations2020, blauthVehicleRoutingTimedependent2024, fontaineExactAnytimeHeuristic2024} is an NDLCF over $\mathcal{T}_{\mathrm{arc}}$, non-decreasing because it is FIFO-compliant, i.e. $\forall t, t' \in \mathcal{T}_{\mathrm{arc}}, t < t' \Rightarrow \alpha_{ij}(t) \leq \alpha_{ij}(t')$ \cite{fontaineExactAnytimeHeuristic2024}. By definition, it satisfies $\forall t \in \mathcal{T}_{\mathrm{arc}}, \alpha_{ij}(t) \geq t$, so it belongs to $\mathcal{F}^{\text{anc}}_{a_{\mathrm{arc}}}$ (proposition \ref{prop:taxonomy}). In particular, the codomain is $\mathcal{T}_{\mathrm{ext}}$ and not $\mathcal{T}$, since late departures may arrive past the horizon end $T$. The canonical data contract of the \kayros{} benchmark suite makes every ATF total on $\mathcal{T}_{\mathrm{arc}}$ precisely to give such late arrivals a well-defined value. Note that the triangle inequality may not hold, i.e. there may exist $i, j, k \in \mathcal{V}$ such that $\alpha_{ik}(t) \geq (\alpha_{jk} \circ \alpha_{ij})(t)$ \cite{lera-romeroLinearEdgeCosts2020}. The function admits at most $p$ breakpoints, as it is directly derived from $\tau_{ij}$ \cite{blauthVehicleRoutingTimedependent2024}.

\noindent\textbf{Complexity:} Computed in $\mathcal{O}(1)$ from the TTF, thus with no need for duplication in memory.

\paragraph{Vertex TW Ready Time Lower Bound Function (vertex RTF)}

\noindent\textbf{Properties:} The function $\theta_i$ of a customer $i \in \mathcal{C}$ \cite{visserEfficientMoveEvaluations2020} is a continuous PWL function over $[0, l_i] \subseteq \mathcal{T}$ and is non-decreasing by construction. It encodes both the time window and the service duration at that customer. If $e_i = 0$, the waiting piece is empty and $\theta_i(t) = t + s_i$ over $[0, l_i]$, a single affine piece with two breakpoints, whatever the value of $s_i$ (the identity when in addition $s_i = 0$, and a single breakpoint $(0, s_i)$ in the doubly degenerate case $l_i = e_i = 0$). If $e_i > 0$, $\theta_i$ has two pieces (a flat waiting piece, then a slope-one piece) and three breakpoints, except in the degenerate single-instant case $e_i = l_i$, where the slope-one piece vanishes and two breakpoints remain. The depot conventions of sections \ref{subsec:dm-def} and \ref{subsec:taxonomy} make both depot maps identities, $\theta_d = \mathrm{Id}_{[0, l_d]}$ and $\theta_o = \mathrm{Id}_{[e_o, l_o] \cap \mathcal{T}_{\mathrm{arc}}}$, the latter being the innermost identity of every route fold.

\noindent\textbf{Complexity:} For a customer $i \in \mathcal{C}$, the function $\theta_i$ has one to three breakpoints per the case analysis above, and each depot identity has two. Hence it is computed in $\mathcal{O}(1)$ and requires a constant amount of memory, stored as TWs.

\paragraph{Arc Ready Time Function (arc RTF)}

\noindent\textbf{Properties:} We introduce the function $\delta_{ij}$ as the composition of two NDLCFs, the pair $(\theta_j, \alpha_{ij})$ satisfying hypothesis (H) since $a_{\theta_j} = 0 \leq a_{\mathrm{arc}} \leq \alpha_{ij}(a_{\mathrm{arc}})$. From theorem \ref{theorem:closure}, it is also an NDLCF over its domain $\dom(\langle i, j \rangle) = [a_{\mathrm{arc}}, u_{\delta_{ij}}]$, the largest initial segment of $\mathcal{T}_{\mathrm{arc}}$ such that the arrival at $j$ respects its TW deadline $l_j$. When this domain is empty, i.e. when $\alpha_{ij}(a_{\mathrm{arc}}) > l_j$, the arc can never be traversed feasibly and is removed in preprocessing. This formalization makes explicit the domain restriction that the definition of \citeauthor{visserEfficientMoveEvaluations2020} \cite{visserEfficientMoveEvaluations2020} leaves implicit. Note the namespace proximity with the classic outgoing-arcs notation $\delta^+(i)$: the ready time functions always carry vertex-pair or route subscripts, which removes the ambiguity.

\noindent\textbf{Complexity:} From theorem \ref{theorem:merge}, the function $\delta_{ij}$ has $\phi_{\delta_{ij}} \leq \phi_{\alpha_{ij}} + \phi_{\theta_j} \leq p + 3$, hence $\mathcal{O}(p)$ breakpoints, so it has $\mathcal{O}(p)$ time and memory complexity. Since there are $\mathcal{O}(\left| \mathcal{A} \right|) = \mathcal{O}(\left| \mathcal{V} \right|^2) = \mathcal{O}(n^2)$ arcs, precomputing all the arc RTFs costs $\mathcal{O}(n^2 p)$ in time and memory.

\paragraph{Route Ready Time Function (route RTF)}

\noindent\textbf{Properties:} By propositions \ref{prop:temporal} and \ref{prop:taxonomy}, each $\delta_{\mathbf{r}_k}$ \cite{dabiaBranchPriceTimeDependent2013, lera-romeroLinearEdgeCosts2020, visserEfficientMoveEvaluations2020} is an NDLCF of $\mathcal{F}^{\text{in}}_{a_{\mathrm{arc}}}$ defined recursively from arc-level components, every composition of the recursion satisfying (H). The overall route ready time $\delta_{\mathbf{r}}$ is an NDLCF over its recursively computed domain $\dom(\mathbf{r}) \subseteq [e_o, l_o] \cap \mathcal{T}_{\mathrm{arc}}$, non-empty exactly when $\mathbf{r}$ is feasible (section \ref{subsec:dm-def}). Note the codomains: a prefix ready time $\delta_{\mathbf{r}_k}(t) \leq l_{v_k} + s_{v_k}$ may exceed the horizon end $T$ at an intermediate customer, hence $\mathcal{T}_{\mathrm{ext}}$, while the complete route ends upon arrival at the depot $d$, no later than $l_d$ by the destination clamp $\theta_d = \mathrm{Id}_{[0, l_d]}$. Stated differently, partial-function composition being associative, $\delta_{\mathbf{r}} = (\theta_{v_m} \circ \alpha_{v_{m-1} v_m} \circ \dots \circ \theta_{v_k} \circ \alpha_{v_{k-1} v_k} \circ \dots \circ \theta_{v_2} \circ \alpha_{v_1 v_2})(t) = (\delta_{v_{m-1} d} \circ \dots \circ \delta_{v_{k-1} v_k} \circ \dots \circ \delta_{o v_2})(t)$.

\noindent\textbf{Complexity:} The BBT $\bbt^{\mathbf{r}}$ of section \ref{subsec:taxonomy} is computed bottom to top, from the $m - 1$ arc RTFs at its leaves to $\delta_{\mathbf{r}}$ at its root, over its $\mathit{BBTL} = \left\lceil \log_2(m - 1) \right\rceil + 1$ levels (for $m = 2$, the single arc RTF is the root and no composition is needed), following the recursion displayed there \cite{visserEfficientMoveEvaluations2020, blauthVehicleRoutingTimedependent2024}. Every node is in $\mathcal{F}^{\text{anc}}_{a_{\mathrm{arc}}}$, so that (H) holds and theorem \ref{theorem:merge} applies at each internal composition. Its cost, and hence that of $\delta_\mathbf{r}$, is that of theorem \ref{theorem:comp_of_rrtf}, proved here.

\begin{proof}
At the bottom level $\ell = 1$, the path's $m$ vertices define $m - 1$ arcs and as many arc RTFs, the leaves of $\bbt^{\mathbf{p}}$, obtained in $\mathcal{O}(1)$ time each thanks to preprocessing and carrying $\mathcal{O}(p)$ breakpoints each. Each level $\ell \in [2, \mathit{BBTL}]_{\mathbb{N}}$ has $\left\lceil (m - 1) / 2^{\ell - 1} \right\rceil$ nodes, each computing the composition $\delta_{v_i v_j}^{\mathbf{p}} = \delta_{v_k v_j}^{\mathbf{p}} \circ \delta_{v_i v_k}^{\mathbf{p}}$ of its two children from level $\ell - 1$ (a node left without a sibling passes its function up unchanged, at no cost). A node at level $\ell$ covers at most $2^{\ell - 1}$ consecutive arcs, so by repeated application of theorem \ref{theorem:merge} its function has $\mathcal{O}(2^{\ell - 1} p)$ breakpoints, and composing its two children, of $\mathcal{O}(2^{\ell - 2} p)$ breakpoints each, costs $\mathcal{O}(2^{\ell - 1} p)$ operations. Level $\ell$ therefore costs $\left\lceil (m - 1) / 2^{\ell - 1} \right\rceil \cdot \mathcal{O}(2^{\ell - 1} p) \subseteq \mathcal{O}((m - 1) p + 2^{\ell - 1} p) \subseteq \mathcal{O}(m p)$ in time and memory, the last inclusion because $2^{\ell - 1} \leq 2^{\mathit{BBTL} - 1} = 2^{\left\lceil \log_2(m - 1) \right\rceil} < 2 (m - 1)$. Summing over the $\mathit{BBTL} - 1$ composing levels, computing $\bbt^{\mathbf{p}}$ costs $\mathcal{O}((\mathit{BBTL} - 1) \cdot m p) = \mathcal{O}(\left\lceil \log_2(m - 1) \right\rceil \cdot m p) \subseteq \mathcal{O}(m p \log_2(m))$ in time and memory, which is also the complexity of obtaining $\delta_{\mathbf{p}}$ at the root, whose representation has $\mathcal{O}(\phi_{\delta_{\mathbf{p}}}) = \mathcal{O}(m p)$ breakpoints. Remark that without the arc-RTF preprocessing, the leaves require $m - 1$ additional compositions $\theta_{v_{k+1}} \circ \alpha_{v_k v_{k+1}}$ of $\mathcal{O}(p)$-breakpoint functions, an extra $\mathcal{O}(m p)$ that leaves the complexity unchanged \cite{visserEfficientMoveEvaluations2020, blauthVehicleRoutingTimedependent2024}.
\end{proof}

\paragraph{Route Duration Function (RDF)}

\noindent\textbf{Properties:} The function $\Delta_\mathbf{r}$ \cite{lera-romeroLinearEdgeCosts2020, panHybridAlgorithmTimedependent2021, visserEfficientMoveEvaluations2020} is a left-continuous PWL function that gives the duration (a value of $\mathbb{R}_{\geq 0}$, not a date of the horizon) of a route $\mathbf{r}$ when departure from its first vertex, the origin depot $o$, occurs at $t$. It belongs to $\mathcal{F}$ but need not be non-decreasing, so it is the only function the route algebra builds that may fall outside $\mathcal{F}_{\text{ND}}$, the given arc TTF being the other potentially non-monotone member of the taxonomy. Its minimum value $\Delta_{\mathbf{r}}^{*} = \min_{t \in \dom(\mathbf{r})} \Delta_\mathbf{r}(t)$ gives the best route duration, the quantity minimized by the DM-TDVRPTW objective of section \ref{subsec:dm-def}. It allows solving the Optimal Starting Time Problem \cite{hashimotoIteratedLocalSearch2007} to determine the route \textit{dispatch time}, i.e. the earliest optimal departure time from vertex $o$: $t_\mathbf{r}^{*} = \min \operatorname{arg\,min}_{t \in \dom(\mathbf{r})} \Delta_\mathbf{r}(t)$ \cite{visserEfficientMoveEvaluations2020}.

\noindent\textbf{Complexity:} Computed in $\mathcal{O}(1)$ from the route RTF $\delta_{\mathbf{r}}$, with no need for duplicated memory. Computing $\Delta_{\mathbf{r}}^{*}$ requires finding the minimum of $\delta_{\mathbf{r}}(x) - x$ over the breakpoints of $\beta_{\delta_{\mathbf{r}}}$ in $\mathcal{O}(\phi_{\delta_{\mathbf{r}}}) = \mathcal{O}(m p)$ time, from which $t_{\mathbf{r}}^{*}$, the earliest breakpoint attaining it, is obtained directly in $\mathcal{O}(1)$. Lemma \ref{lemma:minimum} below proves that both readings are exact and independent of the representing chain.

\subsection{Chains, closure and the correctness of the event merge}\label{app:taxonomy-proofs}

This subsection proves the statements of section \ref{subsec:taxonomy}: that chains represent exactly the NDLCFs, that the duration minimum read off a chain is exact, the closure theorem \ref{theorem:closure}, the correctness theorem \ref{theorem:merge} for algorithm \ref{algo:compose_NDLCF}, the temporal subclasses of propositions \ref{prop:temporal} and \ref{prop:taxonomy}, and the recovery of the continuous literature. Throughout, $f(t^+)$ denotes the right limit of a function of $\mathcal{F}_{\text{ND}}$ at $t < u_f$, and $f(u_f^+) := f(u_f)$ at the right endpoint. For a chain, the smallest and largest ordinates carried at an abscissa $t$ are the values $\mathsf{s}(\beta)(t)$ and $\mathsf{s}(\beta)(t^+)$ of its selected function, except at the last abscissa where the largest ordinate may exceed $\mathsf{s}(\beta)(x_\phi)$ (trailing vertical). An element of $\mathcal{F}_{\text{ND}}$ may have a single-point domain, on which it takes a single value.

\begin{lemma}[Chains represent exactly the NDLCFs]\label{lemma:chains}
For every non-empty chain $\beta$, $\mathsf{s}(\beta) \in \mathcal{F}_{\text{ND}}$. Conversely every non-empty $h \in \mathcal{F}_{\text{ND}}$ is represented by its \emph{canonical chain}, which lists, for each cut abscissa $t_k$ of $h$, the point $(t_k, h(t_k))$ followed by $(t_k, h(t_k^+))$ when $h(t_k) < h(t_k^+)$ and $t_k < u_h$. Two chains represent the same function exactly when they agree after the removal of redundant points and of a trailing vertical.
\end{lemma}

\begin{proof}
Between consecutive distinct abscissae $\mathsf{s}(\beta)$ interpolates linearly, hence is affine, and non-decreasing since ordinates are non-decreasing. At an abscissa $t = x_k > x_1$, the left limit of $\mathsf{s}(\beta)$ is the ordinate of the last point before the group at $t$, interpolated up to $t$, which is the smallest ordinate of the group, so $\mathsf{s}(\beta)$ is left-continuous, and it has finitely many pieces. Conversely, on each piece $(t_k, t_{k+1}]$ the affine extension of $h$ to $t_k$ from the right is $h(t_k^+)$, so the canonical chain interpolates exactly $h$ there, and at $t_k$ its smallest ordinate is $h(t_k)$. The last claim follows since $\mathsf{s}$ ignores exactly those points.
\end{proof}

\begin{lemma}[Chain-invariant duration minimum]\label{lemma:minimum}
Let $h \in \mathcal{F}_{\text{ND}}$ be non-empty and $\beta$ any chain representing it. Then $\Delta_h(t) := h(t) - t$ attains its minimum on $\dom(h)$, $\min_t \Delta_h(t) = \min_k (y_k - x_k)$, and the earliest abscissa of $\beta$ attaining the right-hand minimum is $\min \operatorname{arg\,min}_t \Delta_h(t)$.
\end{lemma}

\begin{proof}
$\Delta_h$ is left-continuous with $\Delta_h(t) \leq \Delta_h(t^+)$, hence lower semicontinuous on the compact $\dom(h)$, so it attains its minimum and $\operatorname{arg\,min} \Delta_h$ is closed and non-empty, with a least element $t^\dagger$. On a piece $(t_k, t_{k+1}]$ the map $\Delta_h$ is affine, so its infimum over the piece is $\min(\Delta_h(t_k^+), \Delta_h(t_{k+1}))$, and $\Delta_h(t_k^+) \geq \Delta_h(t_k)$ is attained at $t_k$. Hence $\min \Delta_h$ is attained at some $t_k$, the $t_k$ being taken minimal (the abscissae where $h$ jumps or changes slope, and the two endpoints), and every chain representing $h$ carries a point at each of them with smallest ordinate $h(t_k)$, since $\mathsf{s}$ would otherwise interpolate straight across. So the minimum over the chain points $(x, h(x))$ equals $\min \Delta_h$. Every other chain point $(x, y)$ has $y > h(x)$ (a higher ordinate of a vertical run, or the trailing vertical) or lies on an affine run of $\Delta_h$ between two chain abscissae, so it cannot undercut the minimum. If $t^\dagger$ were not a chain abscissa it would lie strictly inside a piece $(x_k, x_{k+1})$ on which $\Delta_h$ is affine, and an affine function minimal at an interior point is constant on the piece, so $\Delta_h(x_k^+) = \min \Delta_h$ and $\Delta_h(x_k) \leq \Delta_h(x_k^+)$ would give an earlier minimizer $x_k$. So $t^\dagger$ is a chain abscissa, its smallest ordinate is $h(t^\dagger)$, and no earlier chain point attains the minimum.
\end{proof}

This is the one-pass reading of $\Delta_{\mathbf{r}}^{*}$ and $t_{\mathbf{r}}^{*}$ of section \ref{subsec:taxonomy}, the lemma being applied at $h = \delta_{\mathbf{r}}$ with $\Delta_{\mathbf{r}}$ standing for $\Delta_{\delta_{\mathbf{r}}}$ (the same subscript shorthand as $\dom(\cdot)$), and its result is therefore the same for every chain representing $\delta_{\mathbf{r}}$.

\begin{proof}[Proof of theorem \ref{theorem:closure}]
Since $g$ is non-decreasing, $g(t) \geq g(a_g) \geq a_f$ for all $t \in \dom(g)$ by (H), so $g(t) \in \dom(f)$ reduces to $g(t) \leq u_f$. The set $\mathcal{U} = \{t \in [a_g, u_g] : g(t) \leq u_f\}$ is an initial segment of $[a_g, u_g]$ by monotonicity, empty exactly when $g(a_g) > u_f$. When non-empty it is closed: if $t_n \uparrow t_\infty$ with $t_n \in \mathcal{U}$ then $g(t_\infty) = \lim g(t_n) \leq u_f$ by left continuity. A non-empty closed initial segment of a compact interval is $[a_g, \max \mathcal{U}]$, which gives the domain and $a_{f \circ g} = a_g$. The composition is non-decreasing as a composition of non-decreasing maps. It is left-continuous at $t \in (a_g, u_{f \circ g}]$: if $t_n \uparrow t$ with $t_n$ increasing then $g(t_n) \uparrow g(t)$, and either $g(t_n) = g(t)$ for large $n$, or $g(t_n) < g(t)$ so that $g(t) > a_f$ lies where $f$ is left-continuous, and $f(g(t_n)) \to f(g(t))$ in both cases. It is piecewise linear with finitely many pieces: for each piece boundary $v$ of $f$, the level set $g^{-1}(\{v\}) \cap \dom(f \circ g)$ is empty, a point, or an interval closed on the right (left continuity closes it on the right only, and it is left-open exactly when $g$ jumps to the value $v$, i.e. $g(t_0) < v = g(t_0^+)$ at some $t_0$). Cut $[a_g, u_{f \circ g}]$ at its endpoints, at the piece boundaries of $g$ and at the infimum and supremum of each non-empty such level set, a finite set of cut points listed once each. On each open cut piece $g$ is affine, and either constant, in which case $f \circ g$ is constant, or strictly increasing, in which case its image is an open interval containing no piece boundary of $f$ (a boundary $v$ inside the image would put the infimum or the supremum of $g^{-1}(\{v\})$ inside the cut piece), so $f$ is affine on that image and $f \circ g$ is affine.
\end{proof}

\begin{proof}[Proof of theorem \ref{theorem:merge}]
The proof reads the listing locals of algorithm \ref{algo:compose_NDLCF}, $\mathit{ys}$, $\mathit{ts}$, $\mathit{lo}$ and $\mathit{hi}$, quoted as named there. Write $h = f \circ g$, and let $g(t)$, $g(t^+)$, $f(v)$, $f(v^+)$ denote the selected functions of the input chains. Under (H), $\mathit{lo} = \max(a_f, y^g_1) = y^g_1 = g(a_g)$, and $\mathit{hi} = \min(u_f, y^g_{\phi_g})$ with $y^g_{\phi_g} \geq g(u_g)$. So $\mathit{lo} > \mathit{hi}$ if and only if $g(a_g) > u_f$, which by theorem \ref{theorem:closure} is the empty case. Assume now $g(a_g) \leq u_f$.

\emph{Events.} Let $\mathcal{E}$ be the set of abscissae of $\beta_f$ and ordinates of $\beta_g$ lying in $[\mathit{lo}, \mathit{hi}]$. The loop processes exactly the values of $\mathcal{E}$ in increasing order, once each: at each iteration $\mathit{ev}$ is the smallest unprocessed coordinate in range from either input, and all points of $\beta_f$ with abscissa $\mathit{ev}$ and all points of $\beta_g$ with ordinate $\mathit{ev}$ are consumed, being consecutive in their chains. The first event is $\mathit{lo} = y^g_1$, carried by the point $(a_g, g(a_g))$. For $\mathit{ev} \in \mathcal{E}$: if $\mathit{ev}$ is an ordinate of $\beta_g$, $\mathit{ts}(\mathit{ev})$ lists the abscissae of the points carrying it, in order. Otherwise $y^g_{j-1} < \mathit{ev} < y^g_j$ and $\mathit{ts}(\mathit{ev}) = (t)$ with $t$ interpolated, which is the unique solution of $g(t) = \mathit{ev}$ when $x^g_{j-1} < x^g_j$ ($g$ continuous and strictly increasing there), and the jump abscissa $x^g_j$ when $x^g_{j-1} = x^g_j$ ($\mathit{ev}$ inside a connector). Symmetrically, if $\mathit{ev}$ is an abscissa of $\beta_f$, $\mathit{ys}(\mathit{ev})$ lists the ordinates at $\mathit{ev}$, from $f(\mathit{ev})$ up to $f(\mathit{ev}^+)$ when $\mathit{ev} < x^f_{\phi_f}$, and up to the top of a possible trailing vertical of $\beta_f$ when $\mathit{ev} = x^f_{\phi_f}$. Otherwise $\mathit{ys}(\mathit{ev}) = (f(\mathit{ev}))$ by interpolation and $f$ is continuous at $\mathit{ev}$. Interpolation never indexes outside a chain: in the $\mathit{ys}$ branch $x^f_1 \leq \mathit{lo} \leq \mathit{ev} \leq \mathit{hi} \leq x^f_{\phi_f}$ with $\mathit{ev}$ not an abscissa, so index 1 was skipped by the initialization or consumed at the first event and some abscissa exceeds $\mathit{ev}$, and the $\mathit{ts}$ branch is symmetric. Two structural facts follow. (S1) Every emitted abscissa is a chain abscissa of $\beta_g$ or a point where $g$ is continuous and strictly increasing with $g(t) = \mathit{ev}$ for the event $\mathit{ev}$ that produced it, so a point strictly interior to a plateau of $g$ is never emitted unless it is a redundant chain abscissa. (S2) If $t \in \mathit{ts}(\mathit{ev})$ then $\mathit{ev} \geq g(t)$.

\emph{Output abscissae.} Let $\mathcal{X}$ be the set of emitted abscissae. By (S2) every $t \in \mathcal{X}$ satisfies $g(t) \leq \mathit{ev} \leq \mathit{hi} \leq u_f$, so $\mathcal{X} \subseteq [a_g, u_h]$, and $a_g \in \mathit{ts}(\mathit{lo}) \subseteq \mathcal{X}$. Also $u_h \in \mathcal{X}$: let $v = g(u_h) \leq u_f$. If $u_h = u_g$ then $v$ is the smallest ordinate of $\beta_g$ at $u_g$, so $v \leq \mathit{hi}$, $v \in \mathcal{E}$ and $u_h \in \mathit{ts}(v)$. If $u_h < u_g$, then $g(s) > u_f$ for every $s > u_h$, hence $y^g_{\phi_g} > u_f$ and $\mathit{hi} = u_f$. Either $g$ jumps at $u_h$, which is then a chain abscissa with smallest ordinate $v \leq \mathit{hi}$ and $u_h \in \mathit{ts}(v)$, or $g$ is continuous at $u_h$ with $g(u_h) = u_f$ exactly, and then $u_f \in \mathcal{E}$ is an abscissa of $\beta_f$ and $\mathit{ts}(u_f)$ contains $u_h$.

\emph{The output is a chain.} Within one event the emitted points are non-decreasing in both coordinates by construction. Across events $\mathit{ev} < \mathit{ev}'$, every $t \in \mathit{ts}(\mathit{ev})$ and $t' \in \mathit{ts}(\mathit{ev}')$ satisfy $t \leq t'$ because $(t, \mathit{ev})$ and $(t', \mathit{ev}')$ both lie on the monotone completed graph of $\beta_g$, and every $y \in \mathit{ys}(\mathit{ev})$ and $y' \in \mathit{ys}(\mathit{ev}')$ satisfy $y \leq y'$ on that of $\beta_f$. The exact-duplicate drop of \textit{emit} is therefore its only effect in exact arithmetic.

\emph{The output represents $h$.} It suffices to show: (a) at every $t \in \mathcal{X}$ the smallest emitted ordinate is $h(t)$; (b) at every $t \in \mathcal{X}$ with $t < u_h$ the largest emitted ordinate is $h(t^+)$; (c) between two consecutive elements $t < t'$ of $\mathcal{X}$, $h$ is affine on $(t, t')$ with end values $h(t^+)$ and $h(t')$. Then $\mathsf{s}(\beta_h)$ agrees with $h$ at every abscissa of $\mathcal{X}$ by (a), and on each open gap by (b), (c) and linear interpolation, ordinates emitted at $u_h$ above $h(u_h)$ forming a trailing vertical that $\mathsf{s}$ ignores.

(a) Fix $t \in \mathcal{X}$ and $v = g(t)$. Then $v \in \mathcal{E}$: if $t$ is a chain abscissa of $\beta_g$ (which includes every abscissa returned by the interpolation inside a connector), $v$ is its smallest ordinate. Otherwise by (S1) $t$ was interpolated on a strictly increasing piece and $v$ is the event that produced it. The event $v$ emits $(t, \mathit{ys}_1(v)) = (t, f(v)) = (t, h(t))$ since $t \in \mathit{ts}(v)$. Every other ordinate emitted at $t$ comes from an event $\mathit{ev} \geq v$ by (S2), with ordinate at least $f(v)$.

(b) Fix $t \in \mathcal{X}$ with $t < u_h$ and $w = g(t^+)$. Then $w \leq g(u_h) \leq u_f$ and $w \leq y^g_{\phi_g}$, so $w \leq \mathit{hi}$. If $g \equiv w$ on some $(t, t']$ (plateau), then by (S1) $t$ is a chain abscissa of $\beta_g$ whose largest ordinate is $w$, so $t \in \mathit{ts}(w)$ and $\mathit{ts}(w)$ ends at the plateau's end $t' > t$ or later: the event $w$ emits $(t, f(w))$ and places the jump ordinates of $f$ at $w$, if any, at the end of $\mathit{ts}(w)$, not at $t$, and no event $\mathit{ev} > w$ emits at $t$ since the completed graph of $\beta_g$ continues horizontally from $(t, w)$. Hence the largest ordinate at $t$ is $f(w) = h(t^+)$, as $h \equiv f(w)$ on $(t, t']$. Otherwise $g$ is strictly increasing right after $t$, so no chain point $(s, w)$ with $s > t$ exists (it would give $g(s) \leq w$), $t$ is the last element of $\mathit{ts}(w)$, and the event $w$ emits all of $\mathit{ys}(w)$ at $t$, whose largest element is $f(w^+)$ if $w$ is an abscissa of $\beta_f$ with $w < x^f_{\phi_f}$ and $f(w)$ if $f$ is continuous at $w$. The sub-case $w = x^f_{\phi_f} = u_f$ would force $g(s) > u_f$ for $s > t$, contradicting $t < u_h$. No event $\mathit{ev} > w$ emits at $t$, as $g$ leaves $w$ immediately and $w$ is the top of any connector at $t$. And $h(t^+) = \lim_{s \downarrow t} f(g(s))$ with $g(s) \downarrow w$ strictly from above, which is $f(w^+)$, equal to $f(w)$ when $f$ is continuous at $w$.

(c) Let $t < t'$ be consecutive in $\mathcal{X}$. No chain abscissa $x$ of $\beta_g$ lies in $(t, t')$, since it would satisfy $g(x) \leq g(t') \leq u_f$ and $g(x) \leq y^g_{\phi_g}$, hence $g(x) \in \mathcal{E}$ and $x \in \mathit{ts}(g(x)) \subseteq \mathcal{X}$. So $g$ is affine and continuous on $(t, t')$, with values in the open interval $(g(t^+), g(t'))$, or constant. No abscissa $v$ of $\beta_f$ lies in $(g(t^+), g(t'))$: such a $v$ would be an event, and it cannot be a chain ordinate of $\beta_g$ (a chain point $(s, v)$ with $s < x$, where $x$ is the unique solution of $g = v$ in $(t, t')$, would force $g \equiv v$ up to $x$, one with $s > x$ would force $g(s) = g(x)$, both contradicting strict increase, and a trailing vertical carries only ordinates above $g(u_g) \geq g(t')$), so $\mathit{ts}(v)$ is the interpolated singleton $(x)$ and $x$ would be in $\mathcal{X}$. So $f$ is affine on $(g(t^+), g(t'))$ and $h$ is affine on $(t, t')$, with end values $\lim_{s \downarrow t} f(g(s)) = h(t^+)$ and $\lim_{s \uparrow t'} f(g(s)) = f(g(t')) = h(t')$ by left continuity of $f$ and $g$. If $g$ is constant on $(t, t')$, $h$ is constant with the same end values.

\emph{Bound and cost.} An event with $c_f$ chain ordinates of $f$ and $c_g$ chain abscissae of $g$ emits $c_f + c_g - 1$ points and consumes $c_f + c_g$ input points. If one side was interpolated, the event emits $\max(c_f, c_g)$ points and consumes as many. Both sides interpolated is impossible since $\mathit{ev}$ is a coordinate of some input point. So $\phi_h \leq \phi_f + \phi_g$, and each event costs time linear in the points it consumes plus $\mathcal{O}(1)$, for $\mathcal{O}(\phi_f + \phi_g)$ in total, the initial skips included.
\end{proof}

Theorem \ref{theorem:merge} is stated for any representing chains, so any two chains representing $f$ and any two representing $g$ yield outputs representing the same $f \circ g$. Bitwise equality of the outputs is not claimed and does not hold in general, redundant points and trailing verticals of the inputs propagating to the output. The bit-identity of the checker and of \kayros{} is the property of one fixed pipeline running one listing on the same chains, not of the operation across representations.

\begin{proof}[Proof of proposition \ref{prop:temporal}]
For $f \in \mathcal{F}^{\text{out}}_a$ and $g \in \mathcal{F}^{\text{in}}_a$, $g(a_g) \geq a_g \geq a \geq a_f$ is (H). The values of $f \circ g$ satisfy $f(g(t)) \geq g(t) \geq t$, and $a_{f \circ g} = a_g \geq a$ by theorem \ref{theorem:closure}, so $f \circ g \in \mathcal{F}^{\text{in}}_a$, with $a_{f \circ g} = a$ when $a_g = a$. For the bracketings: if some internal node evaluates to $\bot$, then by associativity of partial-function composition every bracketing of the full product is $\bot$, and algorithm \ref{algo:compose_NDLCF} returns $\bot$ at that node and propagates it through its first line. Otherwise, by induction on the bracketing tree, the subtree containing $f_1$ evaluates to a member of $\mathcal{F}^{\text{in}}_a$ and every other subtree to a member of $\mathcal{F}^{\text{anc}}_a \subseteq \mathcal{F}^{\text{out}}_a$, because composition order is fixed and the subtree containing $f_1$ is always the inner operand of its parent, so the first part applies at every node and theorem \ref{theorem:merge} yields a representation of the pointwise composition there. Equality of the bracketings is associativity of partial-function composition.
\end{proof}

\begin{proof}[Proof of proposition \ref{prop:taxonomy}]
$a = a_{\mathrm{arc}} \geq 0$ since $\mathcal{T}_{\mathrm{arc}} \subseteq \mathcal{T} \subset \mathbb{R}_{\geq 0}$. A customer function, $i \in \mathcal{C}$, has domain $[0, l_i]$ with $0 \leq a$ and $\theta_i(t) = \max(t, e_i) + s_i \geq t$, so $\theta_i \in \mathcal{F}^{\text{out}}_a$, and so does $\theta_d = \mathrm{Id}_{[0, l_d]}$. An arc ATF has domain $\mathcal{T}_{\mathrm{arc}}$ and $\alpha_{ij}(t) \geq t$, so $\alpha_{ij} \in \mathcal{F}^{\text{anc}}_a$. The depot departure is the identity on $[e_o, l_o] \cap \mathcal{T}_{\mathrm{arc}}$, whose lower endpoint is at least $a$, so it is in $\mathcal{F}^{\text{in}}_a$. The rest is proposition \ref{prop:temporal}. The sequential fold alternates $f \circ g$ with $f \in \mathcal{F}^{\text{out}}_a$ (an ATF, a vertex function or the clamp) and $g \in \mathcal{F}^{\text{in}}_a$ (the departure identity, then each accumulator), so its first part applies at every step. An arc RTF $\theta_j \circ \alpha_{ij}$ and a leaf $\mathit{lf}_k$, $k \geq 2$, compose an outer map onto an anchored map and are anchored, the return leaf $\mathrm{Id}_{[0, l_d]} \circ \alpha_{v_{m-1} d}$ included, while $\mathit{lf}_1$ has the departure identity as innermost operand and is in $\mathcal{F}^{\text{in}}_a$, so the second part applies to every bracketing of the leaf list. Note that a bare vertex function is an outer map only ($a_{\theta_i} = 0 < a$ when $a_{\mathrm{arc}} > 0$), so re-bracketing the finer list of vertex and arc functions is not covered and not performed. The canonical data contract of section \ref{subsec:kayros-arch} enforces exactly the properties used, a chain non-decreasing in both coordinates spanning $\mathcal{T}_{\mathrm{arc}}$ with $y_k \geq x_k$ at every point, plus $a_{\mathrm{arc}} \geq 0$, which holds on every canonical instance of section \ref{subsec:benchmarks} (five families at 0, Blauth2024 at $54 \cdot 10^{6}$ ms).
\end{proof}

\begin{proposition}[Recovery of the continuous literature]\label{prop:continuous}
Recall that $\mathcal{F}_{\text{ND}}^{\text{c}}$ is the continuous subclass of $\mathcal{F}_{\text{ND}}$, i.e. the NDCPWLFs. (i) If $f, g \in \mathcal{F}_{\text{ND}}^{\text{c}}$ satisfy (H) then $f \circ g \in \mathcal{F}_{\text{ND}}^{\text{c}}$, and if $\beta_f$ and $\beta_g$ carry no vertical run then at every event $|\mathit{ys}(\mathit{ev})| = 1$, the second emit loop of algorithm \ref{algo:compose_NDLCF} never runs, the output carries no vertical run, and the listing degenerates to the breakpoint sweep of \cite{visserEfficientMoveEvaluations2020, blauthVehicleRoutingTimedependent2024}. (ii) The class of \citeauthor{visserEfficientMoveEvaluations2020} \cite{visserEfficientMoveEvaluations2020}, piecewise linear continuous non-decreasing functions on $[0, T_f]$ (their letters kept as published, $T_f$ being their per-function right endpoint, written $u_f$ here), is $\mathcal{F}_{\text{ND}}^{\text{c}} \cap \{a_f = 0\}$, on which (H) reads $g(0) \geq 0$ and holds for every function with non-negative values, hence for every time-valued function. Their theorem 3 is theorems \ref{theorem:closure} and \ref{theorem:merge} on that class, up to their breakpoint convention. (iii) The restrictions to $[t_{\min}, t_{\max}]$ of the ATFs of \citeauthor{blauthVehicleRoutingTimedependent2024} \cite{blauthVehicleRoutingTimedependent2024}, continuous non-decreasing functions on $(-\infty, t_{\max}]$ with $a(t) \geq t$ that are constant below some $t_{\min}$ ($a$ is their ATF letter, kept as published and not to be read as the anchor $a$ of proposition \ref{prop:temporal}), belong to $\mathcal{F}_{\text{ND}}^{\text{c}} \cap \mathcal{F}^{\text{anc}}_{t_{\min}}$ for a common $t_{\min}$, every member of which extends to such an ATF by the constant $f(t_{\min})$. Their proposition 2 is theorems \ref{theorem:closure} and \ref{theorem:merge} on that class up to their breakpoint convention, their lower-unbounded initially-constant domain being a device to make every ATF total on the left, which (H) with a common lower endpoint replaces.
\end{proposition}

\begin{proof}
(i) A composition of continuous non-decreasing maps is continuous. If neither input chain has a vertical run, every abscissa of $\beta_f$ carries one point, so $|\mathit{ys}(\mathit{ev})| = 1$ at every event. An output vertical run would need either $|\mathit{ys}(\mathit{ev})| \geq 2$ or two emitted points at the same abscissa from different events, which requires $t \in \mathit{ts}(\mathit{ev}) \cap \mathit{ts}(\mathit{ev}')$ with $\mathit{ev} < \mathit{ev}'$, hence a connector of $\beta_g$ at $t$. (ii) and (iii) are immediate from the definitions, the sign condition in (ii) being the one point where (H) is narrower than the cited statement, which admits $g(0) < 0$: there the pointwise domain starts at $\min\{t : g(t) \geq 0\} > 0$, a case that never arises for time values.
\end{proof}

\paragraph{What fails without (H).} The smallest example uses integer breakpoints. Let $\beta_f = ((8, 9))$ and $\beta_g = ((0, 0), (2, 6), (2, 9))$: $g$ is the continuous function $3t$ on $[0, 2]$, represented with a trailing vertical through values it never takes, and $a_f = 8 > g(0) = 0$ violates (H). Pointwise, $g(t) \leq 6 < 8$ for all $t$, so $f \circ g = \bot$. Algorithm \ref{algo:compose_NDLCF} computes $\mathit{lo} = \mathit{hi} = 8$, one event $\mathit{ev} = 8$ with $\mathit{ys} = (9)$ and $\mathit{ts}$ interpolated inside the connector to $(2)$, and returns $((2, 9))$, a non-empty chain for an empty composition. On the canonical chain $((0, 0), (2, 6))$ of the same $g$, it correctly returns $\bot$. With $\beta_{g''} = ((0, 2), (2, 2))$ the two bracketings of $f \circ g \circ g''$ then even disagree, $(f \circ g) \circ g''$ giving $((0, 9), (2, 9))$ and $f \circ (g \circ g'')$ giving $((2, 9))$, both wrong since the true product is $\bot$. In exact rational arithmetic on random chains violating (H), the merge remains exact whenever the pointwise domain is a closed interval. Its only failure modes are a non-empty answer for an empty composition and a closed answer for a composition whose domain is left-open at a jump of $g$, an object no chain can represent. (H) is thus the hypothesis under which the class is closed and the representation faithful, and propositions \ref{prop:temporal} and \ref{prop:taxonomy} show that route evaluation never leaves it. These statements were checked in exact rational arithmetic on the shipped listing, over the example above, the double tie, boundary ties, redundant points, trailing verticals, random pairs with and without (H), and the association of random temporal chains.

\subsection{Exactness of normalization and the three defect families}\label{app:taxonomy-exactness}

\noindent\textbf{Exactness of normalization:} Normalization, the removal of redundant points, is not a step of algorithm \ref{algo:compose_NDLCF} but a separate operation specified with \emph{exact} slope equality. A subtlety of floating-point arithmetic dictates how much of it may enter exact cost semantics. Removing interior breakpoints of exactly-horizontal or exactly-vertical runs is provably bit-neutral for evaluated values and the minimum duration: flat interpolation is exact, vertical interiors are unobservable under the selected-function rule of definition \ref{def:selected}, and the minimum of $y - x$ over breakpoints is never attained strictly inside either run. On a horizontal run, the earliest minimizing departure can nevertheless move to its right endpoint when binary64 subtraction rounds distinct $y-x$ values to the same value within one unit in the last place. Removing a breakpoint that lies on a \emph{sloped} piece, however, is not evaluation-neutral even under an exact collinearity predicate, because interpolation is not transitive in floating point: downstream compositions then interpolate over relocated endpoints. On ATFs with genuine vertical steps \cite{rifkiImpactSpatiotemporalGranularity2020}, we measured the resulting pointwise deviations reaching full step heights. For this reason, \kayros{} keeps the \emph{un-normalized} composition as its canonical cost semantics (the checker and the solver remain bit-identical by construction). The restricted flat/vertical deduplication remains available as a provably neutral optimization (10--28\% fewer stored breakpoints on TW-clamped and step families of the canonical instances, at no semantic cost). The benchmark implementation \href{https://github.com/0nyr/pwlf_compare}{on GitHub} predates this analysis and inherits epsilon-based comparators from the \texttt{lera2020} baseline it extends. The exact, epsilon-free form of the sweep is the event merge of algorithm \ref{algo:compose_NDLCF}, shipped in the canonical checker and in \kayros{}.

\noindent\textbf{Implications:} Our validation campaign around this operation uncovered three independent families of defects, each with a lesson for solver and benchmark design. \emph{(i) Reference implementations.} The performance assessment on randomly generated arc RTFs (section \ref{subsec:ndcpwlf-benchmark}) revealed silent index-out-of-bounds errors in the original Open Source piecewise-linear composition of \cite{lera-romeroLinearEdgeCosts2020}, which we corrected before benchmarking against our \texttt{visser+nor}. \emph{(ii) Exact boundary ties.} Optimal Duration-minimization solutions systematically \emph{ride} time-window deadlines: waiting at an upstream vertex clamps the arrival time to the (integral) TW opening, so composed route RTFs carry pieces that tie a downstream deadline in exact floating-point equality. Such ties have probability zero on random data, so random-function testing is powerless against this class, yet they arise systematically on real instances. Re-traced on the canonical, byte-exact instance data at their published departure times in plain IEEE-754 double arithmetic with no tolerance, 72 of the 146 published BKS of \cite{lera-romeroLinearEdgeCosts2020}, through 188 of the 898 published routes, carry at least one such tie bit-for-bit. An early two-pointer implementation of the sweep, predating the event merge of algorithm \ref{algo:compose_NDLCF} and lacking the tail-flush steps its formulation requires, silently collapsed the composed domain on these ties on 3 of the 898 published BKS routes. It remained bit-identical to the reference on all others. \emph{(iii) Benchmark data provenance.} Most consequential of all, re-evaluating the same 146 BKS on distances re-derived at full precision from the Solomon coordinates, instead of the floor-truncated canonical ones (section \ref{subsec:benchmarks}), inflates every arc by $+0.40$ time units on average and flips 11 of the BKS to infeasible. All 11 are among the 72 tie-carrying solutions above, provably optimal ones included, precisely because optimal solutions ride deadlines and the ties break against them. The solutions are not at fault: the evaluation has silently switched instance sets, exactly as classic-VRPTW results shift between the truncated DIMACS 2021 and float SINTEF conventions \cite{DIMACS2021, Sintef2008}. On the curated byte-exact \textit{Dabia2013} data shipped with \kayros{}, our exact, epsilon-free reference checker reproduces all 146 published BKS and even finds a strictly better solution on RC105\_100. Its published departure time on one route stops $10^{-5}$ short of the duration-minimizing breakpoint, a gap below the comparison tolerance $\varepsilon = 10^{-5}$ of the original solver and therefore invisible to its epsilon-based comparisons. Appendix \ref{app:bks} summarizes the two-perspective re-pricing, with per-instance values in the thesis and the public solution store \cite{rascoussierMAMUTrouting2026}. These findings fix two non-negotiable design rules of \kayros{} and its companion benchmark suite: solution costs are always the output of a canonical, epsilon-free, Open Source checker (\emph{the checker defines the objective}), and benchmark instances are distributed as byte-exact canonical data with checksums, never as re-derivation recipes. It is not the first time in the TD routing domain that such issues surface years after major publications \cite{foschiniComplexityTimeDependentShortest2011, dabiaErratumBranchPrice2024}.

\noindent\textbf{Two remarks on the shipped merge:} both concern the double tie of section \ref{subsec:composition}, where one plateau of the inner function meets one jump of the outer one at the same value. Advancing pairwise on such a tie, as a naive two-pointer merge does, fuses the bottom of the jump with the left end of the plateau and interpolates diagonally across the corner. The shipped implementation additionally clamps each new coordinate up to the previous one, an operation that never fires in exact arithmetic and exists only so that floating-point rounding cannot break the chain invariant. Thus a coordinate copied from an input chain is exact in binary64 unless an earlier rounded interpolation activates this monotone clamp.

\section{Solver Internals}\label{app:solver}

This appendix states in full the parts of section \ref{sec:solver} that the running text summarizes: the tree-ranked move evaluation and the measurement that forces the ranking-accounting split, the fleet-descent ladder and the penalty-tolerant squeeze, the five additions to the exact component and the four-run certificate gate, and the full readings of the fleet-aware machinery promised by section \ref{subsec:kayros-fleet}.

\subsection{Move evaluation and the bracketing measurement}\label{app:solver-splice}

Every candidate move is an instance of one splice primitive: replace a window of the receiving route by a window donated from another route, the donated window empty for a deletion and the removed window empty for an insertion. Relocations, swaps and 2-opt* tails all take this shape, and algorithm \ref{algo:splice} prices it with a constant number of tree queries plus a constant number of seam leaves. Intra-route relocation, which recomposes two seams inside one route, has its own variant of the same scheme.

\begin{algorithm}
\setstretch{1}\small
\caption{$\textit{rankSplice}$: tree-ranked evaluation of a candidate move}\label{algo:splice}
\begin{algorithmic}[1]
\Require Receiving route $\mathbf{r}_1$ and donor route $\mathbf{r}_2$ with their LCA-BSTs; a removed window of $\mathbf{r}_1$ and a possibly empty donated window of $\mathbf{r}_2$
\Ensure A ranking value for the spliced route, or $+\infty$ if it is infeasible
\State $h := \textit{query}(\mathbf{r}_1, \text{leaves strictly before the removed window})$ \Comment{one LCA query, $\leq 1$ composition}
\If{the donated window is non-empty}
    \State $h := \textit{query}(\mathbf{r}_2, \text{interior leaves of the donated window}) \circ \textit{bridge}(i, j) \circ h$ \Comment{$i$ last kept vertex, $j$ first donated vertex; $\textit{bridge}(i,j) = \theta_j \circ \alpha_{ij}$ belongs to neither stored route}
\EndIf
\If{$\mathbf{r}_1$ has a surviving suffix}
    \State $h := \textit{query}(\mathbf{r}_1, \text{leaves strictly after the removed window}) \circ \textit{bridge}(i', j') \circ h$ \Comment{$i'$ last placed, $j'$ first suffix vertex}
\Else
    \State $h := \mathrm{Id}_{[0, l_d]} \circ \alpha_{i' d} \circ h$ \Comment{closing leaf}
\EndIf
\State \textbf{if} $h = \bot$ \textbf{then} \Return $+\infty$ \textbf{end if} \Comment{empty domain: time-window infeasible}
\State \Return $\min_{t} \left( h(t) - t \right)$ \Comment{ranking only; never stored as a cost}
\end{algorithmic}
\end{algorithm}

Separating ranking from accounting is not a stylistic precaution: it is forced by measurement. Partial-function composition is associative in exact real arithmetic. Proposition \ref{prop:temporal} establishes closure of the temporal subclass under its anchor hypothesis. Each computed composition \emph{interpolates}, and floating-point interpolation is not associative. Different bracketings of one chain can therefore differ by units in the last place. On 4\,800 routes over 24 instances spanning every family, we compared the checker's interleaved left fold against a precomposed-leaf fold and against a balanced fold. Across the two comparisons, 73\% and 71\% of the route ready-time functions were bitwise identical and 91\% and 89\% of the durations exactly equal. Feasibility never disagreed on a single route, and almost all mismatches were dust of the order of $7 \cdot 10^{-12}$. The exception is instructive: on a stepwise instance \cite{rifkiImpactSpatiotemporalGranularity2020}, accumulated dust landed just above a step boundary instead of exactly on it. One fold then rode the top of the step and the other the bottom, and the two durations differed by a full step height. This is a property of floating-point composition, not a defect of either fold. \kayros{} therefore enforces the rule \emph{trees rank, the fold accounts}: an accepted move is committed only after both changed routes have been rebuilt and repriced by the sequential checker-identical fold, the acceptance test being a strict inequality on repriced totals with no tolerance. The number of tree-ranked candidates the fold then rejects is instrumented rather than assumed to be zero. Measured against naive recomposition over 50\,000 move evaluations per configuration on best-known solutions, the structures speed move evaluation up by a factor 2 at a mean route length of about 6 and by up to a factor 17 at about 50. The LCA-BST leads the composition tree everywhere by 15 to 35\%, and the set of feasible moves is identical across the three methods on every configuration. We found no size at which the sequential fold overtakes the structures.

\subsection{The anytime stack: parameters and screens}\label{app:solver-anytime}

This subsection holds the constants, the proximity formula and the screens of the anytime search of section \ref{subsec:kayros-anytime}, which the running text states only by their effect.

\paragraph{The proximity relation} Since \textit{Duration} carries no separate distance term, the proximity of a pair is read off the ATFs. The departure interval of $i$ is first intersected with the arc's own departure domain, giving $\mathit{lo}(i,j) = \max \left( e_i + s_i, \min \dom(\alpha_{ij}) \right)$ and $\mathit{hi}(i,j) = \min \left( l_i + s_i, \max \dom(\alpha_{ij}) \right)$. The pair is rejected outright when $\mathit{lo} > \mathit{hi}$. With $\mathit{mindur}(i,j) = \min \{ \tau_{ij}(t) : t \in [\mathit{lo}, \mathit{hi}] \}$, computed exactly over the breakpoints of $\alpha_{ij}$ lying strictly inside $(\mathit{lo}, \mathit{hi})$ plus the two endpoints, and $\mathit{minwait}(i,j) = \max ( 0, e_j - \alpha_{ij}(\mathit{hi}) )$, the waiting at $j$ that even the latest feasible departure cannot avoid, we set $\mathit{prox}(i,j) = \mathit{mindur}(i,j) + 0.2 \cdot \mathit{minwait}(i,j)$ (the screen quantities are quoted under their implementation names). Pairs screened out by $\alpha_{ij}(\mathit{lo}) > l_j$, whose earliest feasible departure already arrives past $l_j$, are excluded. The relation is then symmetrized and truncated to the 50 nearest others per customer.

\paragraph{Granular justifications and the staleness scheme} Each inter-route operator of the descent carries its own granular justification: a relocation requires the donor to be a neighbor of an insertion-seam customer, a swap requires the two exchanged customers to be neighbors, and a 2-opt* requires one created customer-customer arc to join neighbors. Depot reconnections do not justify a move on their own. Enumeration is restricted a second time over the run, since a staleness scheme re-enumerates a customer only if its own context or a neighbor's has changed since its last clean pass.

\paragraph{Driver constants and the work-unit rate} At each kick a target between 1 and 25 removals is drawn, and the seed customers are then visited in random order. Acceptance runs over a ring of 300 past values, with both enhancements of section 4.2 of \cite{burkeLateAcceptanceHillClimbing2017}: the candidate is accepted when it improves either the late value or the current one, and a history slot is rewritten only when the current solution improves on it. The exhaustive-list polishing descent runs on every new global best, and by default at two further points: on the output of the initial granular descent and inside the fleet-descent basin. The stagnation window is measured in work units because they are nearly rate-invariant across sizes. They spread by a factor 1.5 from $n = 500$ to $n = 2000$, where iteration velocity spreads by a factor 18. Determinism rests on one generator driving kick sizes and customer orders while the descent itself is generator-free.

\subsection{Fleet descent and the penalty-tolerant squeeze}\label{app:solver-fleet}

\emph{Fleet descent} is an ejection-ladder route elimination in the lineage of route-minimization heuristics for the VRPTW \cite{nagataPowerfulRouteMinimization2009}, run on the incumbent at every restart-to-best trigger and, by default, on a work-based period. One attempt selects a victim route uniformly, moves its customers into a last-in-first-out ejection pool, erases the route and drains the pool through a two-rung ladder. Rung one attempts a feasible best insertion, tree-ranked and fold-committed, with no singleton fallback, since reopening a route is exactly the failure mode the phase exists to remove. On failure the customer's difficulty counter is incremented and rung two attempts insertion with ejection: over receiving routes holding a granular neighbor of the customer and contiguous windows of at most two customers, one splice evaluation prices replacing the window by the customer. The feasible window minimizing the sum of the ejected customers' difficulty counters is then committed, and its ejected customers return to the pool. The phase is all-or-nothing: any dead end, exhausted budget or deadline restores the pre-attempt solution exactly, leaving the search state bitwise unperturbed. Since no step opens a route, success means exactly one route fewer with every customer served, and the result then receives a granular descent plus an exhaustive polish. All budgets are counted in work units rather than wall-clock seconds: an earlier wall-clock cap was the one machine-dependent decision left in the solver. It made the same seed on the same instance take different trajectories on different machines.

The evaluated release adds a bounded \emph{penalty-tolerant squeeze}, opt-in and inert by default. It relaxes the time-window channel only, clamping late arrivals at the deadline and accumulating the violation in a separate non-negative channel. A penalized descent on $\Psi_{\mathbf{r}} = \tilde{\Delta}_{\mathbf{r}} + \eta\, W_{\mathbf{r}}$, summed over the routes of the copy, then runs on a copy of the state, where $\tilde{\Delta}_{\mathbf{r}}$ is the duration of route $\mathbf{r}$ under the warp-tolerant evaluation (a single number, the tilde distinguishing it from the exact $\Delta_{\mathbf{r}}^{*}$), $W_{\mathbf{r}} \geq 0$ its accumulated warp and $\eta > 0$ the warp penalty weight, either as a post-drain polish or as an in-ladder rescue for a customer with no feasible position. Two properties keep it compatible with section \ref{subsec:kayros-arch}. Every state returning to exactly zero warp with a strictly better objective is banked, and a dominating repair leg drives residual warp to zero in the budget tail. Only a banked state is ever returned, so warp never crosses the phase boundary. And a phase that banks nothing is a strict no-op, so a successful route drop is never converted into a failure. The inert default is a measured decision: paired against the default at $n = 1000$, the squeeze reached reference fleet sizes the default never reached, but at $n = 500$ it failed its pre-committed acceptance gate on the paired mean cost. Section \ref{subsubsec:fleet-eval} gives both readings.

The dedicated phase replaced a cheaper device, whose limit was measured rather than assumed. That mechanism dissolves one smallest route whole inside the kick so that its customers repair into the others, the additive credit alone being unable to cross the plateau of emptying a route. Instrumentation exposed its limit: 99.0\% of dissolved kicks reopen a singleton route through the repair's own fallback.

\subsection{The exact component: additions and the four-run gate}\label{app:solver-exact}

Five additions produce the component this paper evaluates. First, an open LP backend: the master problems are solved by HiGHS \cite{huangfuParallelizingDualRevised2018} behind a solver-agnostic formulation layer, built statically into the distributed wheels so the exact mode works from a plain installation with no proprietary dependency. The faster commercial backend remains a source-build option. Second, anytime compliance: one absolute deadline is derived per run and every component takes its residual budget from it, including strong-branching probes and the freeze heuristic, which previously ran unbudgeted. New interruption points were added inside cut separation, label merging and solution-pool repricing. Third, warm starts through columns: caller-supplied routes, typically the anytime incumbent, are repriced under the master's own arithmetic and added as initial columns, their total becoming the initial upper bound when they partition the customers. Fourth, checker-consistent binary64 column costs: every route entering the master is repriced by the fold of section \ref{subsec:kayros-folding}, so its objective coefficient and the reported value follow the canonical checker. LP bounds still depend on the backend's stated numerical tolerances. The final value is re-summed in canonical route order. The vendored labeling keeps its own arithmetic for reduced costs, which is safe because a priced column must clear a threshold several orders of magnitude above the difference between the two arithmetics. A priced column the checker rejects is skipped and counted, and its count poisons the run. Fifth, the exact value-jump path used for production pricing on stepwise instances: verticals travel through the piecewise-linear machinery as tagged first-class objects instead of being smoothed into steep bridges. The tag distinguishes a genuine travel-time jump from the inverse of a flat interval, which contains several equally valid departure choices. The double-tie rule of algorithm \ref{algo:compose_NDLCF} formalizes that distinction. Arithmetic designed to be safe for the first meaning is wrong for the second. Selecting between the two by inspecting intermediate operands, where the distinction is no longer recoverable, produced a class of over-certification. The arithmetic is now selected once from the instance's travel-time model. Two verdicts complete the picture. A run reaching its deadline without completing the proof conditions returns open, with the best open node's bound as a solver lower bound whenever the root relaxation finished, valid under the standard LP and pricing tolerances of the certificate statement rather than as a rigorous bound. A run stopped by a resident-set self-guard returns a resource limit. Neither is converted into a weaker form of certificate, and neither can masquerade as node optimality. A truncated pricing pass or strong-branching probe stamps its status, so that \enquote{no new columns} is never read as optimality.

Accordingly, no certificate is issued from a single run. The publication gate is a four-run protocol: each instance is solved four times in separate processes, crossing cold and warm starts with the two bidirectional labeling modes. A certificate requires all four runs to reach proven optimality at one agreeing checker-evaluated binary64 value not above the stored reference, every certifying run to show at least one genuine exact-pricing iteration in an audited pricing census, and every certifying run to return routes the checker actually repriced. A run that skipped a checker-infeasible priced column cannot certify by itself, the skipped column possibly having dominated a feasible one away. A clean-run subset rule, detailed in the thesis, can still certify the instance on its unaffected runs. Agreement is tested at an absolute tolerance, which has refused certificates whose arms disagreed at a relative scale that looks like LP dust. Refusal is a first-class outcome alongside open and resource-limited.

One point about the bridge completes the description: the vendored component is fed in memory, with no instance file on disk. A loaded instance and its ATFs are converted to the travel-time pieces the vendored preprocessing expects, breakpoints emitted verbatim, so that duplicate abscissae become genuine zero-width vertical pieces.

\subsection{Fleet descent and squeeze: the full readings}\label{app:solver-fleet-eval}

Section \ref{subsec:kayros-fleet} promised the full reading of the fleet-aware machinery, including where it fails. Before the fleet-descent phase existed, no run out of 586 matched the reference fleet size on any \textit{Blauth2024} city at $n = 500$, while our summed route durations were already 0.952 to 0.990 of the reference at our own larger route count. That located the whole gap in the fleet rather than in the routing. At $n = 500$ the phase settles that: in an 80-cell paired leg it wins the route count 23 times against 0, taking the summed fleet over the panel from 496 to 471 routes. It also produced the first solutions to improve on the published references of \citeauthor{blauthVehicleRoutingTimedependent2024} \cite{blauthVehicleRoutingTimedependent2024}. At $n \geq 1000$ it has a ceiling and we state it as such. Across 3\,000 cells the route count lands one above the reference on 58.0\% of runs, two above on 38.2\% and three above on 3.7\%. It reaches the reference count on 0.1\% of runs, that is 2 cells, both on the same city. A dedicated lane of 12 cells at 36 hours each returned no reduction at all, and 2\,460 short lottery draws produced no fleet hit whatsoever. Budget rather than seed volume is therefore the currency at this size. One city holds its route count at every setting tested, with the phase on or off, at any trigger pressure, ladder depth or budget. Whether the missing route is unreachable by this mechanism or infeasible under these travel times remains open. At a matched route count the residual duration gap is only 0.3\% to 1.0\%, which confirms that what remains at $n \geq 1000$ is a fleet problem and not a routing problem. After a late drop the search is effectively frozen at the new count.

The penalty-tolerant squeeze was evaluated the same way, armed against the default from bitwise-identical starts. At $n = 1000$ on 12-hour cells it reached the reference fleet on 4 of 30 pairs where the default reached it on 0 of 30. It removed a net 12 routes over the 30 pairs and improved the paired mean objective by 1.32\%. Its pre-committed acceptance gate nevertheless failed on 2-hour cells at $n = 500$, where the armed arm was duration-neutral at a matched fleet but cost 0.277\% on the paired mean against a 0.05\% tolerance. The loss came entirely from route-count endgame variance, and one long pair ended with the armed run stalled one route above its twin. Since there is no size at which arming is unconditionally safe, the released default is inert and campaigns arm it explicitly, which is the honest resolution of a gate that half passed. The negative-result option of section \ref{subsec:kayros-fleet} was used to probe fleet tightness: 90 searches confined to one vehicle below the published record ended without ever attaining the reduced fleet, 90 times out of 90, at a 12-hour single-core budget. That is evidence that the record route counts are tight at that budget, and nothing stronger. The option's contract is what makes the negative reportable, since a run that never attains its cap publishes nothing at all rather than an above-cap solution.

\section{Protocol and Metric Reference}\label{app:protocol}

This appendix is the reference statement of the parts of the protocol of section \ref{subsec:benchmarks} and of the metric design of section \ref{subsec:metrics} that the running text compresses.

\subsection{The canonical preprocessing of Dabia2013}\label{app:protocol-dabia}

A subtle but essential point is that \textit{Dabia2013} is not fully specified by ``Solomon instances plus the speed profiles of \citeauthor{ichouaVehicleDispatchingTimedependent2003}'': the actual instances are the output of a preprocessing pipeline that was never published as data, but hard-coded in the original Java solver of \citeauthor{dabiaBranchPriceTimeDependent2013}, which the authors kindly provided to us. The pipeline first scales every Solomon quantity by a factor $10$. Since coordinates, TW bounds and service times are integers in the original files, this scaling is exact. The planning horizon becomes $T = 10\, l_o^{\text{Sol}}$, where $l_o^{\text{Sol}}$ is the depot due date of the Solomon file. The distance of every arc $\langle i, j \rangle \in \mathcal{A}$ is then computed from the original coordinates $(x_i, y_i)$ and \emph{floor-truncated} to an integer, i.e. an \texttt{(int)} cast in the original Java code:
\[
d_{ij} = \left\lfloor 10 \sqrt{(x_i - x_j)^2 + (y_i - y_j)^2} \right\rfloor.
\]
Each arc is assigned one of three speed categories (slow, normal, fast) by a fixed category matrix, shared by all instances and taken from \cite{ichouaVehicleDispatchingTimedependent2003}. The corresponding stepwise speed profile takes five constant values over the horizon subdivision $\{0, 0.2T, 0.3T, 0.7T, 0.8T, T\}$: $(90, 60, 100, 70, 80)/60$ on fast arcs, $(70, 40, 80, 50, 60)/60$ on normal arcs and $(60, 20, 40, 30, 50)/60$ on slow arcs. Finally, the IGP algorithm \cite{ichouaVehicleDispatchingTimedependent2003} integrates each speed profile until the distance $d_{ij}$ is covered, yielding the piecewise linear, FIFO-compliant TTFs $\tau_{ij}$ of section \ref{sec:math}. In-depth checking showed that the resulting instances are identical to those distributed by \citeauthor{lera-romeroLinearEdgeCosts2020} \cite{lera-romeroLinearEdgeCosts2020}, which are also used in the \citeauthor{panHybridAlgorithmTimedependent2021} \cite{panHybridAlgorithmTimedependent2021} solver.

The floor truncation deserves emphasis: it is invisible in the papers, yet it is part of the very definition of the benchmark. As already observed for the classic VRPTW, the DIMACS 2021 convention (scaled, truncated distances) \cite{DIMACS2021} coexists with the historical SINTEF one (full-precision floating-point distances) \cite{Sintef2008}. Scaling and rounding choices are therefore not innocuous implementation details, because they change the feasible region of every instance \cite{rascoussierImpactScalingRounding2026}. Since $\lfloor u \rfloor \leq u$, the truncated distances make every arc slightly shorter than its full-precision counterpart (mean deficit $\approx 0.40$ time units per arc over the \textit{Dabia2013} arcs). The canonical instances are therefore an \emph{outer} (``optimistic'') approximation of their full-precision reading: feasibility transfers from the full-precision instances to the canonical ones, but not conversely. The distinction is far from academic for the \textit{Duration} objective, whose optimal solutions systematically ride TW deadlines: the re-pricing of the 146 published BKS of \cite{lera-romeroLinearEdgeCosts2020} (appendix \ref{app:bks}) shows that 11 of them, while exactly correct on the canonical data, lose feasibility when the very same routes are re-evaluated with full-precision distances.

\subsection{The stratified draw of the frozen instance set}\label{app:protocol-draw}

Of the 10 panels of section \ref{subsec:benchmarks}, two are taken whole: \textit{Dabia2013} at $n = 100$ (all 56 instances) and \textit{Poryos2026} at $n = 1000$ (all 60). These are the two canonical panels of figure \ref{fig:canonical-curves}. The other 8 are stratified samples of 12 instances each, drawn from pools of 20 to 96 candidates by a deterministic generator that is a pure function of the enumerated pool and one seed, fixed at 14. Selection is solver blind: instances are ranked by a SHA-256 key computed from their identity in a campaign-specific namespace, never from any solver result. The seeded ordering only controls how a stratified round-robin visits the strata. The strata are the ones each family exposes: the fine Solomon classes for \textit{Dabia2013}, the customer profile and congestion depth for \textit{Vu2020}, the generation series and class for \textit{Lera2026}, the city with a coverage floor on every (traffic model, intensity) pair for \textit{Poryos2026}, and a single group for \textit{Rifki2020}, whose Lyon network offers no diversity axis. One panel required two draw rounds before its coverage floor was met.

\subsection{Load fractions}\label{app:protocol-load}

Because a Java arm and a Python-bound arm pay a start-up cost that a native arm does not, we measured what that cost is worth as a fraction of the budget rather than assuming it away. The measured \emph{load fraction} is the mean recorded loading time of a run, from the start of the run clock to the moment the solver takes over, expressed as a share of $\tlim$.\footnote{The arms bound that measurement slightly differently: the native Hexaly arms count from the entry of their binary, so the wrapper handoff and the license-seat wait (about one second in total on the worst panel) sit inside $\tlim$ but outside the recorded loading time. The residual differences are below $0.05$ percentage points of the budget.} The largest panel-mean load fraction across arms is 1.62\% of the budget (58.4\,s) for \kayros{} on the \textit{Poryos2026} $n = 1000$ panel. The final C++ Hexaly arms have panel means near 0.8\% there. The longest recorded individual load on that panel is 127.9\,s (3.552\% of the budget), so 1.62\% is a panel mean rather than a per-run bound. That is why no arm is credited a warm-up allowance.

\subsection{The metric design, developed}\label{app:protocol-metrics}

\paragraph{Why the gap is bounded and signed} The bounded form is what makes an anytime score well defined without arbitrary truncation: an arm that has not yet produced any feasible solution has an unbounded gap, and a raw relative gap would make the aggregate depend entirely on how that state is capped. Here it is not a cap but the natural limit, and the pre-first-incumbent state is scored $\sgap = 1$ by convention. The signed construction is equally deliberate: $\sgap$ and the score built on it remain well defined against a frozen reference set that a solver can beat, where gaps go negative. They are just as well defined in the setting of this campaign, where every campaign improvement was folded into the reference before scoring, so that no gap is negative, $\sgap \geq 0$ pointwise, and a run matching the latest record ends at exactly zero. Both $\sgap$ and the score $\ascore$ built on it take values in $[-1, 1]$ in general and in $[0, 1]$ on this campaign, and strictly positive costs keep the lower endpoint unattained even without the fold.

\paragraph{What the score charges} A negative value would mean a run beating the reference over a substantial part of the horizon, and nothing clips it away: the bounded gap makes clipping unnecessary, and its absence is what keeps the statistic signed. Timestamps come from a monotonic clock with no artificial quantization, are clamped to the horizon, and coalesce at equal times by keeping the better cost.

\paragraph{The final gap and its weakness} Its weakness is the mirror of its familiarity: a run that ends the budget without any incumbent has no finite final gap, and the campaign meets exactly that case (table \ref{tab:pooled}). That is the boundedness argument for $\sgap$ made concrete.

\paragraph{Two mechanical guards} The confirmatory design of section \ref{subsec:metrics} also carries two guards that never had to fire. A panel sharing fewer than two instances between the arms of a contrast would yield no $p$-value and could never be reported as significant. No panel was in that state, and any instance-set mismatch between two arms of a contrast is a hard error rather than a silent skip.

\section{Component Studies in Full}\label{app:components}

This appendix carries the full readings of two component studies of section \ref{subsec:components}: the search-strategy comparison that fixed the iterated local search as the default, and the thread-scaling ladder of the commercial contender.

\subsection{Search strategy head-to-head}\label{app:components-ils}

The anytime layer of section \ref{subsec:kayros-anytime} is an iterated local search, and that is a measured choice. An earlier line of the solver used an ant-colony construction over the same local search. The two were compared head to head, together with a hybrid that runs the colony for the first half of the budget and warm-starts the iterated local search from its incumbent. The three arms shared the code base, the local search, the per-run time limit and the seed, and differed only in the strategy parameter. The comparison covers 20\,808 runs over both problem variants, five instance families and sizes from $n = 10$ to $n = 1000$, on three seeds, with the time limit scaled by instance size. On the 6\,936 paired cells, the iterated local search is strictly better than the colony on 5\,714, tied on 917 and worse on 305, for a mean improvement of 2.52\% in final cost. The advantage is monotone in size within every family except \textit{Vu2020}, where it does not grow with size. The advantage runs from a fraction of a percent on the smallest instances, where ties dominate, to 5.44\% at $n = 1000$, where the colony loses 575 of 576 cells. The only cells where the colony has the edge are the two smallest \textit{Rifki2020} sizes, by about a fifth of a percent. The hybrid ties the pure iterated local search on the small families and loses at scale, by 0.56\% at $n = 200$ growing to 1.17\% at $n = 1000$. At this fixed half-budget split, the warm start did not improve final quality, while the colony iterations displaced iterated-local-search time at larger sizes. Other allocations and switching policies were not tested. A later campaign of 4\,160 runs gives the anytime reading of the same verdict: on the road-network family the iterated local search wins every paired cell from $n = 100$ upward. At $n = 1000$ all of its runs end below a 1\% gap, while the colony arm gets 11\% of its runs there and plateaus above 1.5\%. Two caveats keep this honest. The colony arm's best run beat the other's on four small, heavily congested instances, two of which survived as record improvements, so the colony is not dominated everywhere and remains available as an option. And this evidence is historical with respect to section \ref{subsec:contenders}: it fixed the default strategy at release 0.4.0, long before the 1.6.0 build evaluated here, and it is reported as that decision record rather than as a same-version ablation on the 212-instance set.

\subsection{Thread scaling of the commercial contender}\label{app:components-threads}

The thread-scaling study of section \ref{subsubsec:threads} is reported here in full. Its first published version carries the Python-binding measurement, and the numbers quoted here are those of its revision on the C++ binding (arXiv:2608.10079, version 2) \cite{rascoussierHexalyThreadScaling2026}, whose raw runs are deposited on Zenodo \cite{rascoussierHexalyThreadScalingData2026}.
Table \ref{tab:hexaly-ladder} summarizes what a reader of section \ref{subsubsec:threads} needs. The study solved 10 instances spanning the five families at one hour per run, on 10 seeds, at 1, 2, 4, 8 and 16 requested threads, in both model encodings, each process pinned to exactly its requested number of physical cores, for 500 runs per ladder. Both encodings largely occupy what they ask for. Measured CPU time over wall time tracks the requested width to within 5\% at every level for the sliced encoding, the largest deviation being a 4.5\% shortfall at 4 threads. The external-function encoding also tracks it to within 5\%, except at 2 threads, where 11 of the 100 cells fall short and take the level mean to 1.85 rather than 2. Those 11 cells were re-run alone on one host and every one of them returned to 2.0, so that shortfall is an artifact of co-tenancy in the packed wave rather than a property of the encoding. The ladder therefore answers what each encoding does with its cores rather than whether it takes them. The answers differ. The external-function encoding does not convert width into quality: its mean final gap $g_\tlim$ is 5.87\% at 1 thread, passes through a shallow optimum of 5.59\% at 8 threads and returns to 5.98\% at 16. The dispersion follows the same shape, the standard deviation across seeds falling from 0.908 to 0.648 percentage points up to 8 threads before widening again. Behind that flat pooled figure, 9 of the 10 instances do improve with width, by 17.8\% on their panel-equal mean between 1 and 16 threads. A single \textit{Lera2026} instance of 1000 customers cancels all of it by degrading 25.3\% over the same range. The time-sliced encoding behaves as one would hope: its mean final gap falls by 23.5\% from 1 to 16 threads and its anytime score by 25.4\%, with no saturation at the widest level. The seed standard deviation narrows from 1.31 to 0.70 percentage points. The two encodings therefore cross. At 1 thread they are within 0.1 percentage points, 5.87\% for the external function against 5.96\% for the sliced one. The whole of that edge comes from the same 1000-customer instance, the sliced encoding being ahead on 6 of the 10 instances at that width. By 16 threads the sliced one leads on all 10, 4.56\% against 5.98\%. What slicing buys is thread scaling, not fidelity traded for throughput. At the single-core budget of this campaign the two encodings are close to a tie on this panel, while the full 212-instance comparison of section \ref{subsec:contenders} puts the exact encoding clearly ahead. None of this changes the convention. It prices a separate trade-off: better and more stable solutions at equal wall time in exchange for up to 16 times the nominal core allocation. The gaps in table \ref{tab:hexaly-ladder} are measured against a pre-campaign reference and are not comparable with the scores of section \ref{subsec:contenders}. They are also stated on the raw scale $g_\tlim$, whereas the companion report tabulates the squeezed gap, so only the anytime-score columns of the two documents coincide digit for digit.

\begin{table}[htbp]
\centering
\small
\caption{Hexaly thread scaling on a 10-instance panel at one hour per run and 10 seeds per cell, in the two model encodings, for 500 runs per ladder. Both occupy the cores they request, so the comparison is one of what each encoding does with them: at one thread the two encodings are within 0.1 percentage points of each other, and only the sliced one converts additional threads into quality, the external-function one passing through a shallow optimum at 8 threads and ending no better at 16 than at 1. The dispersion column is the sample standard deviation across the 10 seeds, averaged over the 10 instances. Gaps are measured against a pre-campaign reference and are not comparable with the campaign scores.}
\label{tab:hexaly-ladder}
\begin{tabular}{lrrrrr}
\toprule
Encoding & Threads & Mean $g_{\tlim}$ (\%) & Mean $\ascore$ & Seed std (pp) & CPU / wall \\
\midrule
External function & 1 & 5.87 & 0.03891 & 0.908 & 1.00 \\
 & 2 & 5.95 & 0.03908 & 0.779 & 1.85 \\
 & 4 & 5.63 & 0.03716 & 0.710 & 3.84 \\
 & 8 & 5.59 & 0.03678 & 0.648 & 7.77 \\
 & 16 & 5.98 & 0.03862 & 0.732 & 15.75 \\
\midrule
Time sliced, 96 slices & 1 & 5.96 & 0.04621 & 1.31 & 1.00 \\
 & 2 & 5.67 & 0.04411 & 0.869 & 1.99 \\
 & 4 & 5.13 & 0.04129 & 0.740 & 3.82 \\
 & 8 & 5.00 & 0.03732 & 0.806 & 7.68 \\
 & 16 & 4.56 & 0.03448 & 0.699 & 15.69 \\
\bottomrule
\end{tabular}
\end{table}

\end{document}